\documentclass[11pt,reqno,tbtags]{amsart}
\usepackage{amssymb}
\usepackage{color}
\usepackage{natbib}
\bibpunct[, ]{[}{]}{;}{n}{,}{,}

\usepackage{fancyhdr}
\numberwithin{equation}{section}

\makeatletter
\newtheorem*{rep@theorem}{\rep@title}
\newcommand{\newreptheorem}[2]{%
\newenvironment{rep#1}[1]{%
 \def\rep@title{#2 \ref{##1}}%
 \begin{rep@theorem}}%
 {\end{rep@theorem}}}
\makeatother

\newtheorem{theorem}{Theorem}[section]
\newreptheorem{theorem}{Theorem}
\newtheorem{lemma}[theorem]{Lemma}
\newtheorem{proposition}[theorem]{Proposition}

\newtheorem{remark}[theorem]{Remark}

\theoremstyle{definition}

\theoremstyle{remark}

\newcounter{thmenumerate}

\newcounter{xenumerate}

\newcommand\ran{\operatorname{\mathrm ran}}

\newcommand\E{\operatorname{\mathbb E{}}}

\renewcommand\Pr{\operatorname{\mathbb P{}}}

\newcommand\Var{\operatorname{Var}}

\newcommand\Po{\operatorname{Po}}
\newcommand\Bi{\operatorname{Bi}}

\newcommand\eps{\varepsilon}

\renewcommand\phi{\varphi}
\newcommand\g{\gamma}

\newcommand\la{\lambda}

\def\ui{^{(1)}}
\def\uj{^{(j)}}
\def\ut{^{(2)}}
\def\uN{^{(N)}}

\newcommand\cA{\mathcal A}
\newcommand\cB{\mathcal B}

\newcommand\cF{\mathcal F}

\newcommand\cJ{\mathcal J}
\newcommand\cK{\mathcal K}

\newcommand\cP{\mathcal P}

\newcommand\cQ{\mathcal Q}

\newcommand\cX{\mathcal X}
\newcommand\wX{\widehat{\bX}}

\newcommand\wS{\widehat{S}}

\renewcommand\P{{\mathbb P}}

\newcommand\Z{{\mathbb Z}}

\newcommand\R{{\mathbb R}}
\newcommand\C{{\mathbb C}}

\def\Def{\ :=\ }

\newenvironment{proofof}[1]{\noindent {\bf
Proof of #1}.}{\hfill $\square$\par\smallskip\par}

\newcommand{\eqs}{\begin{eqnarray*}}
\newcommand{\ens}{\end{eqnarray*}}

\newcommand{\eq}{\begin{equation}}
\newcommand{\en}{\end{equation}}

\newcommand{\eqa}{\begin{eqnarray}}

\newcommand{\ena}{\end{eqnarray}}

\def\nin{\noindent}
\def\ex{{\mathbb E}}
\def\t{\tau}
\def\a{\alpha}
\def\th{\theta}
\def\Le{\ \le\ }

\def\law{{\mathcal L}}

\def\d{\delta}
\def\e{\varepsilon}
\def\Ref#1{(\ref{#1})}
\def\D{\Delta}

\def\dtv{d_{TV}}
\def\quarter{\tfrac14}
\def\half{\tfrac12}

\def\pr{{\mathbb P}}

\def\giv{\,|\,}
\def\Eq{\ =\ }
\def\r{\rho}
\def\b{\beta}
\def\s{\sigma}
\def\h{\eta}
\def\non{\nonumber}
\def\bx{{\mathbf x}}
\def\bu{{\mathbf u}}
\def\bv{{\mathbf v}}
\def\bw{{\mathbf w}}
\def\by{{\mathbf y}}
\def\bz{{\mathbf z}}
\def\bc{{\mathbf c}}
\def\bo{{\mathbf 0}}
\def\wrho{{\widehat \rho}}
\def\bJ{{\mathbf J}}
\def\trh{{\tilde\r}}
\def\Jmax{J_M^*}

\def\BBL{Barbour, Brightwell \& Luczak}
\def\BHKK{Barbour, Hamza, Kaspi \& Klebaner}
\def\BLX{Barbour, Luczak \& Xia}
\def\BP{Barbour \& Pollett}

\def\and{\ \mbox{and}\ }
\def\uii{^{(i)}}
\def\bone{{\mathbf{1}}}
\def\sJJ{\sum_{\bJ\in\cJ}}
\def\mJJ{\max_{\bJ\in\cJ}}

\def\hby{{\widehat\by}}
\def\DN{{\rm DN}}

\def\cB{{\mathcal B}}
\def\cP{{\mathcal P}}

\def\tm{{\widetilde m}}
\def\z{\zeta}
\def\bth{{\boldsymbol{\th}}}
\def\re{{\mathbb R}}
\def\f{\varphi}

\def\cY{{\mathcal Y}}
\def\hz{{\widehat\z}}
\def\be{{\mathbf e}}
\def\tbe{\mathbf{\tilde{e}}}
\def\tA{{\widetilde A}}

\def\nat{{\mathbb N}}
\def\ps{\psi}
\def\und{^{N,\d}}
\def\undt{^{N,\d_1/2}}
\def\undi{^{N,\d_1}}
\def\bX{{\mathbf X}}
\def\bY{{\mathbf Y}}
\def\bZ{{\mathbf Z}}
\def\lan{\langle}

\def\ran{\rangle}

\def\tit{{\widetilde t}}
\def\hS{{\widehat S}}
\def\bU{{\mathbf U}}
\def\bV{{\mathbf V}}

\def\hbX{{\widehat{\bX}}}
\def\hr{{\widehat r}}
\def\mJJ{\max_{\bJ\in\cJ}}
\def\bT{\mathbf{T}}

\def\la{\lambda}
\def\hbx{{\widehat\bx}}
\def\cZ{{\mathcal Z}}
\def\hC{{\widehat C}}
\def\tnu{\tilde{\nu}}
\def\nkt{\nu K_3}
\def\tnkt{\tnu K_3}
\def\sdd{\s}

\def\tcX{\widetilde{\cX}}

\def\hktt{p(K_3)}
\def\ts{\tilde{\sigma}}
\def\Xinf{\bZ^N_\infty}

\def\adbm{}
\def\adbe{} 
\def\adbf{} 
\def\grba{} 
\def\adbu{} 
\def\grbu{}
\def\adbd{}
\def\adbc{}
\def\adb{}
\def\adbb{}
\def\ignore#1{}
\def\adbh{} 
\def\adbi{}
\def\adbj{}
\def\adbk{}
\def\adbm{}

\begin{document}
\title
{Convergence to equilibrium for density dependent Markov jump processes}


\author{A. D. Barbour} \thanks{ADB: Work carried out in part at the University of Melbourne and at Monash University, and
supported in part by Australian Research Council Grants Nos DP150101459, DP150103588 and DP220100973}
\address{Institut f\"ur Mathematik, Universit\"at Z\"urich, Winterthurertrasse 190, CH-8057 Z\"urich}
\email{a.d.barbour@math.uzh.ch}

\author{Graham Brightwell}
\address{Department of Mathematics, LSE}
\email{g.r.brightwell@lse.ac.uk}

\author{Malwina Luczak} \thanks{MJL: Research currently supported by a Leverhulme International Professorship LIP-2022-005, previously supported by an ARC Future Fellowship FT170100409, previously supported by an EPSRC Leadership Fellowship EP/J004022/2, and by Australian Research Council Grants Nos DP150101459 and DP150103588.}
\address{Department of Mathematics, University of Manchester}
\email{malwina.luczak@manchester.ac.uk}


\begin{abstract}
We investigate the convergence to (quasi-)equilibrium of a
density dependent Markov chain in~$\Z^d$, whose drift satisfies a system
of ordinary differential equations having an attractive fixed point.
For a sequence of such processes~$\bX^N$, indexed by a size parameter~$N$, the time taken until
the distribution of~$\bX^N$, started in some given state, approaches its \adbj{(quasi-)}equilibrium distribution~$\pi^N$
typically increases with~$N$.  To first order, it corresponds to the time~$t_N$ at which the solution to the
drift equations reaches a distance of~$\sqrt N$ from their fixed point.  However, the length
of the time interval over which the total variation distance between $\law(\bX^N(t))$ and its \adbj{(quasi-)}equilibrium
distribution~$\pi^N$ changes from being close to~$1$ to being close to zero is asymptotically of
smaller order than~$t_N$.  In this sense, the chains exhibit `cutoff', and we are able to prove that the cutoff window 
is of (optimal) constant size.
\end{abstract}

\maketitle

\section{Introduction}\label{S:intro}
\grba{This paper is concerned with sequences $(\bX^N)$ of continuous-time Markov chains whose state space
is a subset of $\Z^d$ for some $d$, making
jumps $\bX^N \to \bX^N +\bJ$ at rate $N r_\bJ(\bX^N/N)$, for $\bJ$ in some finite set $\cJ$, where each $r_\bJ$ is
a fixed continuously differentiable
function defined on a suitable closed region $\widehat S$ of $\R^d$.  Here, $N$ is a ``scale parameter'', often
representing population size.  Markov chains
of this type are known as {\em density-dependent} Markov chains, as their transition rates depend on the density 
$\bx^N \grbu{:= N^{-1}\bX^N}$.  They have \adbh{long} been used to model ecological and epidemiological processes, chemical kinetics
and queueing networks; \adbh{see, for example, Bartlett~(1960), McQuarrie~(1967) and Bailey (1964, Chapter 11)}. }

\grba{
The behaviour of \grbu{$\bx^N$}, for large $N$, is naturally approximated by solutions of the \adbh{system of
ordinary differential equations}
\eq\label{ADB-DE-first}
     \frac{d\by}{dt} \Eq \sum_{\bJ \in \cJ} \bJ r_\bJ(\by);
\en
\adbh{see Kurtz~(1970,1971).}
We are interested in the case where this differential equation has a locally stable fixed point $\bc$ in the
interior of $\widehat S$.
Within the basin of attraction of the fixed point $\bc$, $\bx^N$ has a (quasi)-equilibrium distribution which is
strongly concentrated
near $\bc$.  In this paper, we show that, under only very mild conditions, a density-dependent Markov chain
exhibits `cutoff': from any given
starting state within the basin of attraction of $\bc$, the total variation distance between the distribution of
the process at time $t$ and the
equilibrium distribution goes from near~1 to near~0 across an interval of time of constant length.  We prove essentially
best possible bounds on the
rates of convergence.}   

In Kurtz~(1970,1971), laws of large numbers approximations, in the form of systems of ordinary
differential equations, and diffusion (central) limit theorems were established for a
large class of density dependent Markov population processes~$\bX^N$ in~$d$ dimensions.  \adbh{These theorems}
give a good description of the evolution of such processes over fixed,
finite time intervals, when the typical magnitude~$N$ of the interacting populations is large.
\adbb{In particular, the proportions $\bx^N \grbu{=} N^{-1}\bX^N$ closely follow the solution~$\by$ to
\adbh{the ODE system~\Ref{ADB-DE-first},}
and the process $N^{1/2}\{\bx^N(\cdot) - \by(\cdot)\}$ is approximately Gaussian.}
The theory has since been refined in a number of ways, giving rates of convergence,
strong approximation theorems and equilibrium approximation; see Alm~(1978), Kurtz~(1981,
Chapter~8) and Darling \& Norris (2008). 

However, there are circumstances of practical interest which do not fit into this framework.
The most obvious of these is when the numbers of individuals in some of the populations are
small, so that a description of their evolution in terms of the solution of an ODE system is unlikely
to be very good.  A typical example is the initial phase in the introduction of a new species, when a few
individuals are introduced into a habitat, and compete for resources with a resident
population.  Clearly, there is a reasonable probability of colonization being unsuccessful,
and the length of time until the new species reaches a density comparable to that of the resident population,
should it become established, may be very long.  In such circumstances, direct application
of the theorems in Kurtz (1970) indicates that, over finite time intervals, the proportion of
individuals that belong to the new species remains negligible.  The biologically relevant
analysis is given in \BHKK~(2016), where is shown that
a branching process approximation is appropriate in the initial stages, and that, if
the new species becomes established, the ODE trajectory is eventually followed, but with
a random time delay.

In this paper, \adbh{as indicated above, we investigate the approach to equilibrium, another setting that
is not covered by the theorems of Kurtz (1970,1971).
We suppose that the ODE system~\Ref{ADB-DE-first}}
has a locally stable equilibrium~$\bc$, with open basin of attraction~$\cB(\bc)$.  Under appropriate conditions, 
Theorem~8.5 of Kurtz~(1981) establishes that~$\bX^N$ has an equilibrium distribution~$\pi^N$ which,
when suitably centred and normalized, converges to a multivariate normal distribution as $N\to\infty$.
Moreover, it is shown in \adbb{\BLX~(2018a, Theorem~5.3 and 2018b, Theorem~2.3)} that,
under a natural irreducibility condition,
$\pi^N$ then differs in total variation from a discrete multivariate normal distribution on~$\Z^d$ by an amount of order 
$O(N^{-1/2}\log N)$.  However, even when~$\bX^N$ does not have an equilibrium distribution, it is typically
the case that trajectories starting with~$\bX^N(0)$ near~$N\bc$ exhibit apparent equilibrium behaviour,
over very long time periods;  see Bartlett~(1960) and Barbour~(1976), for example.  The trajectories
mimic those of a process~$\bX\und$
that has the same transition rates as~$\bX^N$, except that transitions that lead out of the ball~$B(N\bc,N\d)$ are set
to zero.  This process has an equilibrium distribution~$\pi\und$ that is approximately multivariate normal, as above.
It follows from \BP\ (2012), Theorems 4.1 and~2.3, that, if $\bX^N(0) \in B(N\bc,N\d)$, the total variation distance
between~$\pi\und$
and the distribution of~$\bX^N(t)$ is small for a range of times $t \in [s_1(N), s_2(N)]$, where~$s_1(N)$
grows polynomially in~$N$, and~$s_2(N)$ grows exponentially with~$N$ (Theorem~4.1 of \BP\ (2012) shows this for a
slightly different modification of the process~$X^N$, and Theorem~2.3 can be used to show that the difference in 
the modification is insignificant).
In this sense, $\pi\und$ can be understood as a {\it quasi-equilibrium\/} distribution for~$\bX^N$;  however,
as discussed in \BP\ (2012), it need not be
the equilibrium distribution of~$\bX^N$, or even a quasi-stationary distribution.

Our main result, Theorem~\ref{thm:main}, is concerned with describing the approach to quasi-equilibrium more
precisely than in \BP\ (2012).
Under quite broad assumptions, we show that, for $\bx^N(0) = N^{-1}\bX^N(0)$ away from the boundary of~$\cB(\bc) \setminus \{\bc\}$,
the total variation distance between the distributions $\law_{\bX^N(0)}(\bX^N(t))$ and~$\pi\und$
decreases from close to~$1$ to close to~$0$
within a window of fixed width about a time~$t_N(\bx^N(0))$, which itself
grows logarithmically with~$N$.  That is, for any $\e > 0$ small enough, there exists $\d_\e < \infty$, not depending on~$N$, 
such that
the total variation distance $\dtv\{\law_{\bX^N(0)}(\bX^N(t_N(\adbh{\bx^N(0)}) - \d_\e)),\pi\und\}$
is at least $1-\e$, whereas $\dtv\{\law_{\bX^N(0)}(\bX^N(t_N(\adbh{\bx^N(0)}) + \d_\e)),\pi\und\}$ is at most~$\e$.
Note that~$t_N(\adbh{\bx^N(0)})$ grows much more slowly with~$N$ than does~$s_1(N)$, confirming that the lower bound
given by \BP\ (2012), Theorem~4.1, on the interval of times~$t$ at which $\law(\grbu{\bX}^N(t))$ is close to quasi-equilibrium,
is unduly pessimistic.
We show further that~$\pi\und$ is strongly concentrated around~$N \bc$.

Cutoff has recently been examined in a particular, one-dimensional example of our setting in He, Luczak \& Ross~(2025).
\adbj{Two examples were also considered in \BBL\ (2022), for which contractive couplings could be established.
As mentioned following the proof of Theorem~\ref{conc-thm}, a major difficulty in the very general setting of this paper 
is that the couplings established in Theorem~\ref{thm:main1}~(iv) are not
{\it everywhere\/} contractive, and that arguments going beyond those of \BBL\ (2022) are needed.}  

\adbj{In Barrera \& Jara (2016, 2020), results of a very similar flavour have been obtained in 
a quite different context, that
of a differential flow with small superimposed standard Gaussian \adbm{noise}.  Near the equilibrium of the flow,
such a process behaves like an Ornstein--Uhlenbeck process, for which a very precise description
of the approach to equilibrium is known.  In the setting of Markov population processes, there
is no analogous standard family of processes available for comparison, and the arguments required here
(in particular, Lemma~\ref{ADB-coupling-lemma} and Proposition~\ref{ADB-close-TV}, to replace Mehler's formula)
are correspondingly more involved.  Barrera \& Jara (2020) base their arguments on the assumption of a strong coercivity
condition, to be satisfied throughout the state space.  Here, we only need a condition of this sort,
\adbm{\Ref{ADB-coercivity}},
to be satisfied at the equilibrium point, and \adbm{Lemma~\ref{lem:matrix-M} shows} that this holds automatically.
}

\subsection{Definition of the process, assumptions and main result}
\label{def-main-thm}

\smallskip

\grbu{
Let $\widehat S$ be a closed subset of $\R^d$.  For each $N \ge 1$, define $N\widehat S := \{ N {\bx} : \bx \in \widehat S\}$, and 
let $S_N = \Z^d \cap N \widehat S$ be the set of integer points in $N \widehat S$.  For each $N \ge 1$, let 
$\bX^N$ be a pure jump Markov process on $S_N$.  It follows that the scaled process $\bx^N := N^{-1}\bX^N$ is a Markov process on 
$N^{-1}\Z^d \cap \widehat S$.}  

We make the following assumptions:


\begin{description}
 \item [Assumption 1]
For $\bX \in S_N$, the only possible transitions are to states $\bX + \bJ$, where $\bJ$ belongs to
a fixed \adb{finite} set $\cJ \subset \Z^d$; the corresponding rates for~\grbu{$\bX^N$} are given by
\eq\label{transition-rates}
  \bX \to \bX + \bJ \quad \mbox{at rate}\quad N r_{\bJ}(N^{-1}\bX),
\en
where the functions $r_{\bJ} \colon \widehat S \to \R^+$ are continuously differentiable.

\medskip
\item [Assumption 2]
The set of jumps
${\mathcal J}$ is {\it spanning\/}, in that every vector in ${\mathbb Z}^d$ can be written as a sum of jumps 
in ${\mathcal J}$.
\end{description}

\nin In particular, Assumption~2 means that \adbu{any $\bX' \in \Z^d$ can be reached from any other
$\bX \in \Z^d$ via a path of jumps in~$\cJ$.} 
Note that the spanning condition is equivalent to the following: for each 
\adbk{coordinate vector~$\be\uii$, $1\le i\le d$,} it is
possible to write
\eq\label{ADB-irreducible}
  \adbe{ \be\uii \Eq \sum_{l=1}^{L_i^+} \bJ^{il}_+;\qquad -\be\uii \Eq \sum_{l=1}^{L_i^-} \bJ^{il}_-,}
\en
where, for all $i,l$, \adbe{$\bJ^{il}_+, \bJ^{il}_- \in \cJ$}.

\begin{description}
 \item [Assumption 3]
There is a fixed point $\bc$ of \adbe{the differential equation
\begin{equation} \label{eq:diff-eq}
    \frac{d\by}{dt} \Eq F(\by) \Def \sum_{{\bJ} \in \cJ} {\bJ} r_{\bJ}({\by})
\end{equation}
in} the interior of $\widehat{S}$, having basin of attraction~$\cB(\bc)$;
thus $F(\bc) = 0$. Assume further that $r_\bJ(\bc) > 0$ for all $\bJ \in \cJ$, and that there exists a
strictly positive constant~$\rho$
such that the real parts of all the eigenvalues of
$A := \sum_{{\bJ} \in \cJ} {\bJ} \nabla r_{\bJ}(\bc)$ are strictly less than~$-\rho$,
\adbu{where $\nabla r_{\bJ}({\by})$  denotes the row vector of partial derivatives of~$r_\bJ$.}
\end{description}



\smallskip
Under Assumption~3, we show in \adbj{Theorem~\ref{thm:main1} and}
Lemma~\ref{lem:matrix-M} that there is a norm $\|\cdot\|_M$ on~$\re^d$ and a $\d_0 > 0$
such that \grba{$B_M(\bc,\d_0) \subseteq \widehat{S}$, and} the ODE system~\Ref{eq:diff-eq} is $\|\cdot\|_M$-contractive 
within $B_M(\bc,\d_0)$, where
\eq\label{ADB-BM-def}
  B_M(\bz,\e) \Def \{\bx\in\re^d\colon\, \|\bx - \bz\|_M \le \e\};
\en
\adbk{and that, further, it is possible within $B_M(\bc,\d_0)$ to construct a coupling of pairs of copies of~$\bX^N$ that
is contractive as long as they are not very close to one another.}

\begin{remark}\label{ADB-d_1-def}{\rm
Since $r_{\bJ}(\bc) > 0$ for each~$\bJ \in \cJ$, by Assumption~3, continuity of the functions~$r_{\bJ}(\cdot)$
implies that there exist $r_0,\d_1 > 0$ such that $r_{\bJ}(\bx) \ge r_0 > 0$ for all $\bx \in B_M(\bc,\d_1)$ and
for all~$\bJ \in \cJ$.  
\adbu{Provided
that~$N$ is sufficiently large, it follows using Assumption~2 that
there is a path from any element of $\Z^d \cap B_M(N\bc,N\d_1/2)$ to any other, not leaving $B_M(N\bc,N\d_1)$
--- see Lemma~\ref{intvec} ---
and then Assumption~1 implies that all elements of $\Z^d \cap B_M(N\bc,N\d_1/2)$ belong to the
same communicating class for the Markov chain~$\bX^N$.}

Assumption~3 \grba{ensures, moreover}, that, if $\bX' = \bX + \be\uii$ and the chain starts
in~$\bX \in B_M(N\bc,N\d_1/2)$, then 
the chain follows exactly the path given in~\Ref{ADB-irreducible} from $\bX$ to~$\bX'$
within a time interval of length $N^{-1}$,
with probability bounded away from zero as $N\to\infty$.  Without loss of generality, we take $\d_1 \le \d_0$,
where~$\d_0$ is as in Theorem~\ref{thm:main1}.}
\end{remark}

For a given starting state~${\by}_0 \in \widehat S$, the differential \adbe{equation~\Ref{eq:diff-eq}}
has a unique solution~$\by_{[\by_0]}$ in~$\widehat S$, up to the time $T({\by}_0) \le \infty$ at which
it leaves~$\wS$.
For $\bx \in \cB(\bc)$, define
\begin{eqnarray}
\label{eq.cutoff-time}
    t_N(\bx) \Def \inf\{t > 0\colon \|\by_{[\bx]}(t) - \bc\|_M = N^{-1/2}\}.
\end{eqnarray}
Note that, under many circumstances, as $N \to \infty$ and for fixed $\bx \in \cB(\bc) \setminus \{\bc\}$,
\[
    t_N(\bx) \ \sim\ (1/2\wrho)\log N + t(\bx),
\]
as $N \to \infty$, for some $t(\bx) \in \R$ that does not vary with~$N$;
here, $-\wrho$ denotes the largest \adbb{real part of any} eigenvalue of~$A$.

\adbd{
The following definition of (generalized) cutoff is essentially that of Barbour, Brightwell, Luczak (2022, Section 1.2,
equation~(1.1)).
Let $(X^N)_{N \ge 1}$ be a sequence of pure jump Markov chains with state space $S_N$ \adbb{and let~$\pi^N$
be a distribution on~$S_N$.}
For $E_N$ a subset of the state space $S_N$, let $(\tit_N (X),\, X \in E_N)$
be a collection of non-random times, and let $(w_N)$ be a sequence of numbers
such that $\lim_{N \to \infty} \inf_{X \in E_N} \tit_N (X)/w_N = \infty$.}

\adbd{We say that $X^N$ exhibits {\it cutoff\/} \adbb{with respect to~$\pi^N$} at time $\tit_N (X)$ on $E_N$
with {\it window width} $w_N$, if there exist (non-random) constants
$(s(\eps), \varepsilon > 0)$
such that, for any $\varepsilon > 0$ and for all $N$ large enough,
\begin{eqnarray}
& \dtv ( {\mathcal L}_X (X^N (\tit_N (X) - s (\eps) w_N)), \pi^N) > 1-\eps \label{def.cutoff-low}\\
& \dtv ( {\mathcal L}_X (X^N (\tit_N (X) + s (\eps) w_N)), \pi^N) < \eps \label{def.cutoff-up}
\end{eqnarray}
uniformly for all $X \in E_N$.}

For \adbh{$0 < \delta < \delta_0$, write
\eq\label{ADB-cX-def}
   \cX_N(\delta) \Def B_M(N\bc,N\d) \cap S_N,
\en
and} let~$\bX\und$ denote a Markov process on~$\cX_N(\d)$, having the same transition
rates as~$\bX^N$, except that the rates for transitions out of~$\cX_N(\d)$ are set to zero.
\adbd{Since~$\bX\und$ is a chain on a finite state space, and is irreducible, in view of Assumption~2,
it has an equilibrium distribution~$\pi\und$.}
We are now in a position to state the main theorem of the present article.

\begin{theorem}
\label{thm:main}
Suppose that $\bX^N$ is a sequence of density dependent Markov population processes satisfying Assumptions~1--3.
\adbd{Fix any} $0 < \d \le \adbj{\d_1}$, where~$\d_1$ is as in Remark~\ref{ADB-d_1-def}.
Then, for any compact set $\cK \subset \cB(\bc) \setminus \{\bc\}$,
$\bX^N$ exhibits cutoff with respect to~$\pi\und$ at time $t_N(N^{-1}\bX)$
on~$N\cK := \{\bX \in S_N\colon\, N^{-1}\bX \in \cK\}$
with constant window width; that is, we can take $\tit_N (X) = t_N(N^{-1}\bX)$ and \adbj{$w_N = 1$} 
in the definition of cutoff above.

Moreover, $\pi\und$  is well concentrated \adbd{around~$N\bc$; that is, 
there exist positive constants \adbj{$b_1$ and~$b_2$} 
such that
\eq\label{ADB-conc-1.2}
  \pi^{N,\delta}\{\adbk{\|\bX - N\bc\|_M} \ge  m \} \Le 4d\exp\left( - \frac {m^2}{\adbj{N b_1  + m b_2}}\right).
\en
}
\end{theorem}


\adbd{The second part of Theorem~\ref{thm:main} follows from Theorem~\ref{ADB-equilibrium}.
The first} is re-stated as Theorem~\ref{ADB-general-starting-point},
which is itself derived from Theorems \ref{lower-bound} and~\ref{upper-bound}.
The argument runs as follows.  In Section~\ref{Density-dependent}, we show that, under reasonable assumptions,
a Markov population process is `contractive' in a ball $B_M(N\bc,N\d_0)$.
From this, in Section~\ref{concentration},
we deduce concentration of the process $\bx^N$ \adbj{around} the solution~$\by_{[\bx^N(0)]}$ of the system~\Ref{eq:diff-eq}
that has the same initial state~$\bx^N(0)$, up to times of order $O(N)$,
provided that \adbm{$\bx^N(0)$} is near enough to~$\bc$.
We then show in Section~\ref{cutoff}
that $\bx^N(t_N(\bx^N(0))-\grba{s})$ is $O(N^{-1/2})$ close to $\by_{[\bx^N(0)]}(t_N(\bx^N(0))-\grba{s})$, and that
$\|\by_{[\bx^N(0)]}(t_N(\bx^N(0))-\grba{s})-\bc \adbk{\|_M} > \adbj{\phi}(\grba{s})N^{-1/2}$,
with $\adbj{\phi}(\grba{s}) > 1$ growing with~$\grba{s}$,
from which the lower bound on the speed of convergence is deduced.

For the upper bound, we use contraction to couple a copy of the
process $\bX^N$ starting in $\bX^N(t_N(\bx^N(0)))$ with another copy starting with the equilibrium
distribution~$\pi\und$, showing that the \adbk{$\|\cdot\|_M$}-distance between them can, with high probability, be reduced to less
than $\th N^{1/2}$, for any prescribed $\th > 0$, after an elapsed time $t(\th)$ not depending on~$N$.  \adbj{We then
use Proposition~\ref{ADB-close-TV}}
to show that, after a further elapsed time~$t'$, not depending on~$N$,
the total variation distance between~$\law(\bX^N(t_N(\bx^N(0)) + t(\th) + t'))$ and the equilibrium
distribution~$\pi\und$ is at most \adbj{$k_*\th$}, for some fixed~\adbj{$k_* > 0$}.  Thus, choosing $\th = \e/\adbj{k_*}$, the
second claim is justified.  

\grba{The bounds that we obtain in Section~\ref{cutoff} are of the form
\[
     \dtv\bigl(\law(\bX^N(\adbj{t_N(\bx^N(0))}+s)),\pi\adbj{\undi}\bigr)
           \ \ge\ 1 - C \exp\bigl\{-k e^{2\r |s|}\bigr\},
\]
for negative $s$, and 
\[
     \dtv\bigl(\law(\bX^N(\adbj{t_N(\bx^N(0))}+s)),\pi\adbj{\undi}\bigr) \Le  C' e^{-\r s},
\]
for positive~$s$, where $C$, $k$ and $C'$ are constants.  We show in Example~1 that these convergence rates
are in fact best possible,
up to the values of the constants: our example is a simple 1-dimensional immigration-death model,
where the total variation distance can
be approximated explicitly, and matches the convergence rates above.}  
We conclude with a \grba{further} simple example.

\section{Preliminaries}\label{Density-dependent}
\setcounter{equation}{0}

\adbk{Th first step in our arguments is to} show that, under Assumptions~1--3, there is a norm with respect to which both the
\adbk{ODE system~\Ref{ADB-DE-first}} and the Markov
process~\adbk{$\bx^N := N^{-1}\bX^N$} exhibit `contractive' behaviour in some region around~$\bc$.  The norm is derived
in the standard way from
a real $d \times d$ symmetric positive definite matrix~$M$.  The matrix~$M$ gives rise to an inner product
$\langle \cdot , \cdot \rangle_M$ on
$\C^d$, defined by $\langle {\bz}, {\bw} \rangle_M = {\bz}^T M \overline{\phantom |\!\bw}$, and the norm $\| \cdot \|_M$
is defined by $\|{\bz}\|_M = \langle {\bz}, {\bz} \rangle_M^{1/2}$.
We write $H(\bz,\bw) := \|\bz-\bw\|_M$ and $G(\bz) := H(\bz,\bc)$,
and we let~$\cQ^N$ denote the generator of \adbm{$\bx^N \grbu{=} N^{-1} \bX^N$}.
\adbe{The functions $G$ and~$H$ provide the basis for deterministic and stochastic Lyapounov arguments.}

\begin{theorem} \label{thm:main1}
Under Assumptions~1--3,
there exists a $d \times d$ symmetric positive definite matrix~$M$ and a constant $\delta_0 > 0$, such that:
\begin{itemize}
\item [(i)] if ${\by}$ is a solution of~(\ref{eq:diff-eq}) with $\|{\by}(0) - \bc \|_M \le \delta_0$, then,
         for all $t \ge 0$,
$$ 
    \frac{d}{dt} \| {\by}(t) - \bc \|_M \Le -\rho \| {\bf y}(t) - \bc \|_M;
$$
\item [(ii)] if ${\by}$ and ${\bz}$ are two solutions of~(\ref{eq:diff-eq}) with
  $\|{\by}(0) - \bc\|_M \le  \delta_0$ and $\|{\bz}(0) - \bc\|_M \le \delta_0$, then, for all $t \ge 0$,
$$
    \frac{d}{dt} \| {\by}(t) - {\bz}(t)\|_M \Le -\rho \| {\by}(t) - {\bz}(t)\|_M.
$$
\end{itemize}
Moreover, there exist $K_1, K_2 \in \R$ such that the following hold whenever $N$ is sufficiently large:
\begin{itemize}
\item [(iii)] For all ${\bX} \in S_N$ with
$\delta_0 \ge G(N^{-1}{\bX}) \ge K_1 N^{-1/2}$, we have
$$
   {\mathcal Q}^N G(N^{-1}{\bX})  \le - \rho G(N^{-1}{\bX});
$$
\item [(iv)]
There exists a Markovian coupling $(\bU^N,\bV^N)$ of two copies of~$\bX^N$, whose
generator~$\cA^N$ is such that, for all ${\bU}, {\bV} \in {S}_N \cap B_M(N\bc,N\d_0)$ with
$H({\bU}, {\bV}) \ge K_2$,
$$
   \cA^N H({\bU},{\bV}) \Le - \rho H({\bU}, {\bV}).
$$
\end{itemize}
\end{theorem}

To obtain the matrix~$M$, we use the following
lemma, which may be  known in various contexts, but we have not been able to find a reference.


\begin{lemma} \label{lem:matrix-M}
Suppose $A$ is a $d \times d$ real matrix, such that all eigenvalues of $A$ have real part strictly less than $-\rho < 0$.
Then there is a real $d\times d$ symmetric positive definite matrix $M$ such that
\eq\label{ADB-coercivity}
   \langle {\bx}, A {\bx} \rangle_M \adb{\Le} - \rho \|{\bx}\|_M^2\ \mbox{ for all }\ {\bx} \in \R^d.
\en
\end{lemma}

Note: if $A$ is diagonalisable, and we assume only that all the eigenvalues have real part {\em at most} $-\rho$,
then we can obtain a matrix $M$
such that $\langle {\bf x}, A {\bf x} \rangle_M \le - \rho \|{\bf x}\|_M^2$ for all ${\bf x} \in \R^d$.  However,
this slightly cleaner result will not be of use to us in what follows.

\begin{proof}
We choose a constant $\mu > 0$ such that $\Re (\lambda) + \mu \le - \rho$ for all eigenvalues $\lambda$ of~$A$.

By standard theory, we can put the matrix $A$ into Jordan Normal Form.  This means that, for each
eigenvalue $\lambda$ of $A$, of algebraic
multiplicity $m_\lambda$, there is a basis for the nullspace of the matrix $(A-\lambda I)^{m_\lambda}$
consisting of vectors
${\bf v}_{\lambda,1}, \dots, {\bf v}_{\lambda,m_\lambda}$ such that each ${\bf v}_{\lambda,i}$ is either
an eigenvector of $A$ with eigenvalue $\lambda$, or
it satisfies $A {\bf v}_{\lambda,i} = \lambda {\bf v}_{\lambda,i} + {\bf v}_{\lambda,i-1}$.  We multiply
each ${\bf v}_{\lambda,i}$ by a suitable
real scalar to obtain a basis $B_\lambda= \{ {\bf u}_{\lambda,i} : i=1, \dots, m_\lambda \}$ of the nullspace,
each element of which is either an
eigenvector with eigenvalue $\lambda$, or satisfies
$A {\bf u}_{\lambda, i} = \lambda {\bf u}_{\lambda,i} + \mu {\bf u}_{\lambda, i-1}$.  Moreover, if
${\bf u}_{\lambda,i}$ appears in the basis
$B_\lambda$, we may take its conjugate to appear in $B_{\overline \lambda}$: i.e., we may assume that
${\bf u}_{\overline \lambda, i} = \overline\bu_{\lambda,i}$ for each non-real $\lambda$ and each $i$.
Finally, taking the unions of the bases $B_\lambda$, we obtain a basis $B$ for $\C^d$.

Now we let $P$ be a $d \times d$ matrix whose columns are the vectors ${\bf u}_{\lambda,i}$ appearing
in the basis~$B$, and set
$M = (P^{-1})^T\overline {P^{-1}}$.  We claim that $M$ has the required properties.

First of all, we claim that $M$ is real.  To see this, we note that $\overline P$ can be obtained from $P$
by exchanging some pairs of columns, so by
post-multiplying by a matrix $C$ with the properties that $C = \overline C = C^T = C^{-1}$.  Therefore we have
$$
  \overline{M} = \overline{(P^{-1})^T} P^{-1} = (\overline{P}^{-1})^T \overline{\overline{P}^{-1}}
    = ((PC)^{-1})^T \overline{(PC)^{-1}}
$$
$$
    = (P^{-1})^T (C^{-1})^T \overline{C^{-1}} \overline{P^{-1}} = (P^{-1})^T \overline{P^{-1}} = M.
$$
Thus indeed $M$ is real.

The matrix $M$ is chosen so that, for any pair ${\bf u} = {\bf u}_{\lambda,i}$ and
${\bf u}' = {\bf u}_{\lambda',i'}$ of distinct vectors in the basis $B$, we have
$$
\langle {\bf u}, {\bf u}' \rangle_M =
{\bf u}^T M \overline{{\bf u}'} = {\bf u}^T (P^{-1})^T \overline{P^{-1}} \overline{{\bf u}'}
= (P^{-1}{\bf u})^T \overline{P^{-1} {\bf u}'},
$$
and this is zero since $P^{-1}{\bf u}$ and $P^{-1}{\bf u}'$ are different standard basis vectors in $\C^d$.
Similarly we have
$\langle {\bf u}, {\bf u} \rangle_M = 1$ for each ${\bf u} \in B$.  This implies that $M$ is symmetric
and positive definite, and that $B$ is an
orthonormal basis of $\C^d$ with respect to the inner product $\langle \cdot, \cdot \rangle_M$.

Now we write any ${\bf x} \in \R^d$ as a linear combination of the basis elements
${\bf x} = \sum_{\lambda} \sum_{i=1}^{m_\lambda} \alpha_{\lambda,i} {\bf u}_{\lambda,i}$, so
$\| {\bf x} \|_M^2 = \sum_\lambda \sum_{i=1}^{m_\lambda} | \alpha_{\lambda,i}|^2$.  For the
real eigenvalues $\lambda$, each $\alpha_{\lambda,i}$ is real. For the non-real eigenvalues $\lambda$, we have
$\alpha_{\overline{\lambda},i} = \overline{\alpha}_{\lambda,i}$ for each~$i$.  Let now $\adbj{\cP}$ be the set of
pairs $(\lambda,i)$ such that
$A {\bf u}_{\lambda,i} = \lambda {\bf u}_{\lambda,i} + \mu {\bf u}_{\lambda,i-1}$.  Then we have
$$
  A{\bf x} \Eq \sum_{\lambda} \sum_{i=1}^{m_\lambda} \lambda \alpha_{\lambda,i} {\bf u}_{\lambda,i}
     + \mu \sum_{(\lambda,i) \in J} \alpha_{\lambda,i} {\bf u}_{\lambda, i-1},
$$
and so
\begin{equation} \label{eq:eq1}
\langle {\bf x}, A{\bf x} \rangle_M = \sum_{\lambda} \sum_{i=1}^{m_\lambda} \overline \lambda |\alpha_{\lambda,i}|^2
+ \mu \sum_{(\lambda,i) \in J} \alpha_{\lambda,i-1} \overline{\alpha_{\lambda,i}}.
\end{equation}
For $(\lambda,i) \in \adbj{\cP}$ with $\lambda$ real, we note that
$\alpha_{\lambda,i-1} \alpha_{\lambda,i} \le \frac12(\alpha_{\lambda,i}^2 + \alpha_{\lambda,i-1}^2)$.
For $(\lambda,i) \in \adbj{\cP}$ with $\lambda$ non-real, we have that $(\overline{\lambda}, i)$ is also in $\adbj{\cP}$, and
$$
\alpha_{\lambda,i-1} \overline{\alpha_{\lambda,i}} + \alpha_{\overline{\lambda},i-1} \overline{\alpha_{\overline{\lambda},i}}
= \alpha_{\lambda,i-1} \overline{\alpha_{\lambda,i}} + \overline{\alpha_{\lambda,i-1}} \alpha_{\lambda,i}
= 2 \Re ( \alpha_{\lambda,i-1} \overline{\alpha_{\lambda,i}} ) \le 2 |\alpha_{\lambda,i-1}| \, |\alpha_{\lambda,i}|
$$
$$
\le |\alpha_{\lambda,i-1}|^2 + |\alpha_{\lambda,i}|^2
= \frac12 \left( |\alpha_{\lambda,i-1}|^2 + |\alpha_{\lambda,i}|^2 + |\alpha_{\overline{\lambda},i-1}|^2 + |\alpha_{\overline{\lambda},i}|^2 \right).
$$
Hence we have
\begin{equation} \label{eq:eq2}
\sum_{(\lambda,i) \in J} \alpha_{\lambda,i-1} \overline{\alpha_{\lambda,i}}
\le \frac12 \sum_{(\lambda,i) \in J} \left( |\alpha_{\lambda,i-1}|^2 + |\alpha_{\lambda,i}|^2 \right)
\le \sum_\lambda \sum_{i=1}^{m_\lambda} |\alpha_{\lambda,i}|^2.
\end{equation}
We can also combine terms corresponding to complex conjugates in the first sum in (\ref{eq:eq1}), noting that
\begin{equation} \label{eq:eq3}
\overline{\lambda} |\alpha_{\lambda,i}|^2 + \lambda |\alpha_{\overline{\lambda},i}|^2 =
2 \Re(\lambda) | \alpha_{\lambda,i}|^2 =
\Re(\lambda)|\alpha_{\lambda,i}|^2 + \Re(\overline{\lambda}) |\alpha_{\overline{\lambda},i}|^2,
\end{equation}
and therefore, substituting using (\ref{eq:eq2}) and (\ref{eq:eq3}) in (\ref{eq:eq1}), we have
$$
\langle {\bf x}, A{\bf x} \rangle_M \le \sum_{\lambda} \sum_{i=1}^{m_\lambda} (\Re(\lambda) + \mu) |\alpha_{\lambda,i}|^2
\le - \rho \sum_{\lambda} \sum_{i=1}^{m_\lambda} |\alpha_{\lambda,i}|^2 = - \rho \| {\bf x}\|_M^2,
$$
as required.
\end{proof}

\adbk{In what follows, we largely work with the $M$-norm and its associated inner product $\langle \cdot \rangle_M$.}
\adbk{For comparing the usual Euclidean norm~$|\cdot|$ to~$\|\cdot\|_M$, we note that}
\eq\label{ADB-CM-def}
    c_0(M)|\bx| \Le \|\bx\|_M \Le c_1(M)|\bx|, \qquad \bx \in \re^d,
\en
where $c^2_0(M)$ and $c_1^2(M)$ are the smallest and largest eigenvalues of~$M$.
We recall from~\Ref{ADB-BM-def} that $B_M(\bz,\e) := \{\bw \in \re^d\colon \|\bw - \bz\|_M \le \e\}$, and
we write
\eq\label{ADB-J*M-def}
   J^*_M \Def \mJJ \|J\|_M.
\en
\adbk{Furthermore, for any compact set $\cK \subset \cB(\bc)$, we define 
\eq\label{L-R-eps-def}
      R^*(\cK) \Def \sup_{\by' \in \cK} \sJJ r_{\bJ}(\by')\quad\mbox{and}\quad L(\cK) \Def \sup_{\by' \in \cK} \adbk{\|DF(\by')\|_M},
\en
where~$F$ is as in~\Ref{eq:diff-eq}, and where $DF(\by)_{ij} := \partial F_i/\partial y_j$; here, for any $d \times d$ matrix~$A$,
\[
   \|A\|_M \Def \sup_{\bz\colon \|\bz\|_M \le 1} \|A\bz\|_M
\]
denotes the $M$-operator norm.}

\medskip

\begin{proofof}{Theorem~\ref{thm:main1} (i) and~(ii)}
We now turn to the proof of Theorem~\ref{thm:main1}, parts~(i) and~(ii),
\adbb{showing that $\|\bx - \bc\|_M$ is a Lyapunov function for the ODE system~\Ref{eq:diff-eq}.}
We note that (ii) implies~(i), since ${\bf y}(t) = \bc$ for all~$t$ is a solution to the differential equation.
\adbb{Choose~$\r' > \r$ in such a way that all eigenvalues of $A$ still have real part strictly
less than $-\rho'$.}

Let ${\bf y}$ and ${\bf z}$ be two solutions of the differential equation
$\dot {\bf y} = \sum_{{\bf J} \in \cJ} {\bf J} r_{\bf J}({\bf y})$,
and set ${\bf w}(t) = {\bf y}(t) - {\bf z}(t)$.  As $M$ is symmetric, we have 
\begin{eqnarray} \label{eq:fifty}
\frac{d}{dt} \| {\bf w}(t) \|_M^2 &=& \frac{d}{dt} \left( {\bf w}(t)^T M {\bf w}(t) \right) = 2 {\bf w}(t)^T M \frac{d}{dt} {\bf w}(t) \nonumber \\
&=& 2 \left\langle {\bf w}(t) , \sum_{{\bf J} \in \cJ} {\bf J} ( r_{\bf J}({\bf y}(t)) - r_{\bf J}({\bf z}(t))) \right\rangle_M.
\end{eqnarray}
Our plan is now to approximate $r_{\bf J}({\bf y}(t)) - r_{\bf J}({\bf z}(t))$ by $\nabla r_{\bf J}(\bc) ({\bf y}(t) - {\bf z}(t))$,
so that (\ref{eq:fifty}) is approximately equal to $2 \langle {\bf w}(t), A {\bf w}(t) \rangle_M$, which is
at most $-2\rho' \| {\bf w}(t) \|_M$
by our choice of~$\r'$.  Our approximation will be accurate enough, provided we work within a sufficiently small
neighbourhood of~$\bc$, as we now discuss.

\adbb{Choose $\delta_* > 0$ so that $B_M(\bc, \delta_*) \subseteq \widehat S$; this is possible, since~$\bc$ is in the interior
of~$\widehat S$, by Assumption~3.}
For any $\eps > 0$, by continuity of the functions $\nabla r_{\bf J}$, there is some $\delta = \delta(\eps) > 0$ with 
$\delta \le \delta_*$ such that,
for all ${\bf y} \in B_M(\bc,\delta)$, $| \nabla r_{\bf J}({\bf y})^T - \nabla r_{\bf J}(\bc)^T | < \eps/ c_0(M)$ for each
${\bf J} \in \cJ$.  Then, for all ${\bf y}, {\bf z} \in B_M(\bc, \delta)$, we can apply the
Mean Value Theorem to the line segment between $\bf y$ and $\bf z$ to obtain that, for each ${\bf J} \in \cJ$,
\begin{equation} \label{eq:error}
\big| r_{\bf J}({\bf y}) - r_{\bf J}({\bf z}) - \nabla r_{\bf J}(\bc) ({\bf y} - {\bf z}) \big|
\le (\eps/c_0(M)) | {\bf y} - {\bf z}| \le \eps \| {\bf y} - {\bf z} \|_M.
\end{equation}

Now, as long as ${\bf y}(t)$ and ${\bf z}(t)$ both remain in $B_M(\bc,\delta(\eps))$,
\begin{eqnarray}
   \frac{d}{dt} \| {\bf w}(t) \|_M^2
   &=& 2 \left\langle {\bf w}(t) , \sum_{{\bf J} \in \cJ} {\bf J} ( r_{\bf J}({\bf y}(t)) - r_{\bf J}({\bf z}(t)))
             \right\rangle_M \non \\
   &\le& 2 \left\langle {\bf w}(t), \sum_{{\bf J} \in \cJ} {\bf J} \nabla r_{\bf J}(\bc) ({\bf y}(t) - {\bf z}(t))
       \right\rangle_M \non\\
   &&\quad\mbox{} + 2 \| {\bf w}(t) \|_M \sum_{{\bf J} \in \cJ} \| {\bf J}\|_M
     \big| r_{\bf J}({\bf y}(t)) - r_{\bf J}({\bf z}(t)) - \nabla r_{\bf J}(\bc) ({\bf y}(t) - {\bf z}(t)) \big| \non \\
   &\le& 2 \langle {\bf w}(t), A {\bf w}(t) \rangle_M
           + 2 \eps \| {\bf w}(t) \|_M^2 \sum_{{\bf J} \in \cJ} \|{\bf J} \|_M \non \\
   &\le& -2 \rho' \| {\bf w}(t) \|_M^2 + 2 \eps \| {\bf w}(t) \|_M^2 \sum_{{\bf J} \in \cJ} \|{\bf J} \|_M,
          \label{ADB-square-deriv}
\end{eqnarray}
\adbb{by our choice of~$\r'$,} and so
\eqa
     \frac{d}{dt} \| {\bf w}(t) \|_M &=& \frac{1}{2 \| {\bf w} (t) \|_M} \frac{d}{dt} \| {\bf w}(t)\|_M^2 \non\\
     &\le& \Big(- \rho' + \eps \sum_{{\bf J} \in \cJ} \|{\bf J}\|_M \Big)\| {\bf w}(t) \|_M
               \Le - \rho \| {\bf w}(t)\|_M,  \label{ADB-mod-deriv}
\ena
provided that~$\eps$, and accordingly~$\delta(\eps)$, are chosen sufficiently small:
\adbb{we choose~$\eps$ to satisfy
\eq\label{ADB-eps-def1}
   \eps \sum_{{\bf J} \in \cJ} \|{\bf J}\|_M  \Eq \tfrac12(\r' - \r).
\en
Note that, by an analogous argument,
\eq\label{ADB-other-bound}
   \frac{d}{dt} \| {\bf w}(t) \|_M \ \ge\ -\r^* \|\bw(t)\|_M,
\en
where $\r^* := \sup_{\by\colon \|\by\|_M=1}\langle \by, -A \by \rangle_M + \eps\sJJ \|J\|_M $, and~$\eps$ is as before.
}

Taking ${\bf y}(0) \in B_M(\bc,\delta(\eps))$, \adbb{with~$\eps$ as in~\Ref{ADB-eps-def1},}
and setting ${\bf z}(t) = \bc$ for all~$t$, we can integrate~\eqref{ADB-mod-deriv}
with respect to~$t$, to deduce that ${\bf y}(t) \in B_M(\bc,\delta(\eps))$ for all $t > 0$.
Thus any two solutions ${\bf y}(t)$ and ${\bf z}(t)$ starting in $B_M(\bc,\delta(\eps))$ both remain in
$B_M(\bc,\delta(\eps))$ for all $t > 0$.  Hence the analysis above tells us that
$$
\frac{d}{dt} \| {\bf y}(t)  - {\bf z}(t) \|_M  \le - \rho \| {\bf y}(t) - {\bf z}(t)\|_M
$$
holds for all $t > 0$.  This establishes parts~(i) and~(ii) of Theorem~\ref{thm:main1},
 \adbb{with $\d_0 = \d(\eps)$ for~$\eps$ as in~\Ref{ADB-eps-def1}.}
\end{proofof}

\medskip

Now we turn to the analysis of the continuous-time jump Markov chain $\bX^N(t)$. 
We shall use, here and later, an estimate for a difference of the form $\| {\bf z} + {\bf J} \|_M - \| {\bf z} \|_M$, which is
essentially best possible in cases where $\| {\bf J} \|_M$ is rather smaller than $\| {\bf z} \|_M$.
We state the result in a general setting.

\begin{lemma} \label{lem:inner}
For any vectors ${\bf a}$ and ${\bf b}$ in an inner product space, we have
$$
\Big| \| {\bf a} + {\bf b} \| - \| {\bf a} \| - \frac{\langle {\bf a}, {\bf b} \rangle}{\|{\bf a} \|} \Big|
\le 2 \frac{\| {\bf b} \|^2}{\| {\bf a} \|}.
$$
\end{lemma}

\begin{proof}
We write
$$
\| {\bf a} + {\bf b} \| - \| {\bf a} \| =
\frac{\| {\bf a} + {\bf b} \|^2 - \| {\bf a} \|^2}{\| {\bf a} + {\bf b} \| + \| {\bf a} \|}
= \frac{2 \langle {\bf a}, {\bf b} \rangle + \| {\bf b} \|^2}{\| {\bf a} + {\bf b} \| + \| {\bf a} \|},
$$
and therefore
\begin{eqnarray*}
\Big| \| {\bf a} + {\bf b} \| - \| {\bf a} \| - \frac{\langle {\bf a}, {\bf b} \rangle}{\|{\bf a} \|} \Big|
&\le& 2 \big| \langle {\bf a}, {\bf b} \rangle \big| \left| \frac{1}{2\| {\bf a} \|} - \frac{1}{\| {\bf a} + {\bf b} \| + \| {\bf a} \|} \right|
+ \frac{\| {\bf b} \|^2}{\| {\bf a} + {\bf b} \| + \| {\bf a} \|} \\
&\le& \| {\bf a} \| \, \| {\bf b} \| \frac{\big| \, \| {\bf a} + {\bf b} \| - \| {\bf a} \| \, \big|} {\| {\bf a} \| (\| {\bf a} + {\bf b} \| + \| {\bf a} \|)}
+ \frac{\| {\bf b} \|^2}{\| {\bf a} \|} \\
&\le& \| {\bf a} \| \, \| {\bf b} \| \frac{ \| {\bf b} \|} {\| {\bf a} \|^2} + \frac{\| {\bf b} \|^2}{\| {\bf a} \|},
\end{eqnarray*}
which is the required result.
\end{proof}

\begin{proofof}{Theorem~\ref{thm:main1} (iii) and~(iv)}
\adbb{Let ${\bX \in S_N}$ be such that $\bx := N^{-1}\bX \in B_M(\bc,\d(\eps))$ and that
$\|\bx - \bc\|_M > K_1 N^{-1/2}$, 
for~$\eps$ as in~\Ref{ADB-eps-def1} and~$K_1$ sufficiently large.} Then
$$
   \adb{\cQ^N G}(\bx) \Eq \sum_{{\bf J} \in \cJ} \bigl(\| \bx + \adbj{(\bJ/N)} - \bc\|_M
           - \| \bx - \bc \|_M \bigr) N r_{\bf J}(\bx).
$$
We shall see that this sum is approximately equal to
$$
S({\bf x}) \Def \sum_{{\bf J} \in \cJ} \frac{\langle \bx - \bc ,\bJ/N \rangle_M}
{\big\| \bx - \bc \big\|_M}\, N \big(r_{\bf J}(\bc) + \nabla r_{\bf J} (\bc) (\bx - \bc) \big).
$$
As $\sum_{{\bf J} \in \cJ} {\bf J} r_{\bf J}(\bc) = {\bf 0}$ and $\sum_{{\bf J} \in \cJ} {\bf J} \nabla r_{\bf J}(\bc) = A$,
we have
\begin{eqnarray*}
S({\bf x}) &=&
\frac{1}{\big\| \bx - \bc \big\|_M} \left \langle \bx - \bc ,
\sum_{{\bf J} \in \cJ} {\bf J} r_{\bf J}(\bc)
+ \sum_{{\bf J} \in \cJ} {\bf J} \nabla r_{\bf J}(\bc) (\bx - \bc) \right \rangle_M \\
&=& \frac{1}{\big\| \bx - \bc \big\|_M} \left \langle \bx - \bc, A (\bx - \bc) \right \rangle_M\\
&\le& -\rho' \big\| \bx - \bc \big\|_M.
\end{eqnarray*}

Our aim is accordingly to bound the difference $\Delta_{\bf J}({\bf x})$ between corresponding terms of the
sums $\cQ^N G(\bx)$ and $S({\bf x})$:
\eqs
   \Delta_{\bf J}({\bf x}) &=& \bigl( \| \bx + ({\bf J}/N) - \bc \|_M
           - \| \bx - \bc \|_M \bigr) N r_{\bf J}(\bx) \\
  &&\mbox{} - \frac{\langle \bx - \bc ,{\bf J}/N \rangle_M}
       {\bigl\|\bx - \bc \bigr\|_M}\, N
             \bigl(r_{\bf J}(\bc) + \nabla r_{\bf J} (\bc)(\bx - \bc) \bigr),
\ens
so that $\adb{\cQ}^N G(\bx) - S({\bf x}) = \sum_{{\bf J} \in \cJ} \Delta_{\bf J}({\bf x})$.

For each ${\bf J} \in \cJ$, we apply Lemma~\ref{lem:inner} with ${\bf a} = \bx - \bc$
and ${\bf b} = {\bf J}/N$, and obtain
\begin{equation} \label{eq.inner}
\Big| \, \| \bx + ({\bf J}/N) - \bc \|_M - \| \bx - \bc\|_M
- \frac{\langle \bx -\bc, {\bf J}/N \rangle_M}{\| \bx - \bc \|_M} \Big|
\Le 2 \,\frac{\| {\bf J} \|_M^2}{N^2 \| \bx - \bc \|_M}.
\end{equation}

Assume now that $\bx \in B_M(\bc,\delta(\eps))$.
For each ${\bf J} \in \cJ$, we use (\ref{eq.inner}) and (\ref{eq:error}) to see that
\begin{eqnarray*}
  |\Delta_{\bf J}({\bf x})|
&\le& \Big| \big\| \bx + ({\bf J}/N) - \bc \big\|_M - \big\| \bx - \bc \big\|_M -
       \frac{\langle {\bf J}/N, \bx - \bc \rangle_M}{\| \bx - \bc \|_M} \Big| N r_{\bf J}(\bx) \\
&& \quad + \frac{\big| \langle {\bf J}/N, \bx - \bc \rangle_M \big|}{\| \bx - \bc \|_M}
   N \big| r_{\bf J}(\bx) - r_{\bf J}(\bc) - \nabla r_{\bf J} (\bc)(\bx - \bc) \big| \\
&\le& 2 \frac{\| {\bf J} \|_M^2 }{N^2 \| \bx - \bc \|_M} N r_{\bf J}(\bx)
     + \frac{1}{N} \| {\bf J} \|_M N \eps \| \bx - \bc \|_M . 
\end{eqnarray*}

Combining the calculations above, 
\adbk{writing $R := R^*(B_M(\bc, \delta(\eps)))$, with $R^*(\cK)$ as in~\Ref{L-R-eps-def}, and  
with~$J^*_M$ as in~\Ref{ADB-J*M-def},} we now observe that, for $\bx \in B_M(\bc, \delta(\eps))$,
\begin{eqnarray*}
  \adb{\cQ}^N G(\bx) &\le& S({\bf x}) + \sum_{{\bf J} \in \cJ} |\Delta_{\bf J}({\bf x})| \\
     &\le& -\rho' \| \bx - \bc \|_M
            + \frac{2 \adbj{(J^*_M)^2 R}}{N \| \bx - \bc \|_M}
            + \eps \sum_{{\bf J} \in \cJ} \| {\bf J} \|_M \| \bx - \bc \|_M \\
     &=& \| \bx - \bc \|_M \left( - \rho'
            + \frac{2 \adbj{(J^*_M)^2 R} }{N \| \bx - \bc \|_M^2}
            + \eps \sum_{{\bf J} \in \cJ} \| {\bf J} \|_M \right).
\end{eqnarray*}
To obtain part~(iii) of the theorem, we need to ensure that each of the two ``error terms'' in the bracket in the final 
expression is at most
 $\frac12 (\rho'-\rho)$.  For the second term, \adbb{this follows from~\Ref{ADB-eps-def1}. }
For the first term, it suffices to have $\|\bx - \bc \|_M$ greater than $K_1 N^{-1/2}$, for a suitable constant $K_1$.

\medskip

To prove part~(iv) of the theorem, our first task is to define a suitable coupling.  Let $\bU^N$ and $\bV^N$
be two copies of the chain $\bX^N$,
with $\bU^N(t) = {\bU}$ and $\bV^N(t) = \bV$ at some time~$t$; write $\bu := N^{-1}\bU$ and $\bv := N^{-1}\bV$.
For each ${\bf J} \in \cJ$, the two copies have possible transitions
to ${\bf U} + {\bf J}$ and ${\bf V} + {\bf J}$
respectively, at rates $N r_{\bf J}({\bf u})$ and $N r_{\bf J}({\bf v})$ respectively.  \adbj{We use a synchronous coupling,
pairing} such transitions as far as
possible; in other words, 
the two chains make the~$\bJ$ transition together at rate
$N \min \left(r_{\bf J}({\bf u}), r_{\bf J}({\bf v})\right)$.
Also, the chain with the larger rate for this transition makes the~$\bJ$ transition alone at rate
$N |r_{\bf J}({\bf u}) - r_{\bf J}({\bf v})|$, while the other chain does not move.
If the two chains make a transition together, then
the distance between the chains, measured as $H(\bU^N(t),\bV^N(t)) = \| \bU^N(t) - \bV^N(t) \|_M$, does not change.
Thus the distance between the two coupled copies of the chain changes only when one chain jumps and the other does not.

\adbb{Suppose now that 
$\bu,\bv \in B_M(\bc,\delta(\eps))$, where
$\d(\eps)$ is defined as for~\Ref{eq:error}, with~$\eps$ as in~\Ref{ADB-eps-def1}. }
For each ${\bf J} \in \cJ$, we set
$$
c_{\bf J} \Def \frac{| \nabla r_{\bf J}(\bc)| }{c_0(M)} + \eps;
$$
from (\ref{eq:error}), it follows that
\begin{eqnarray*}
|r_{\bf J}({\bf y}) - r_{\bf J}({\bf z}) | &\le& |\nabla r_{\bf J}(\bc) ({\bf y} - {\bf z})| + \eps \| {\bf y} - {\bf z}\|_M \\
&\le& | \nabla r_{\bf J}(\bc)| \, | {\bf y}- {\bf z} | + \eps \| {\bf y} - {\bf z}\|_M \\
&\le& c_{\bf J} \| {\bf y} - {\bf z}\|_M,
\end{eqnarray*}
whenever ${\bf y}$ and ${\bf z}$ are in $B_M(\bc,\delta(\eps))$.

Let $\cJ_1 := \cJ_1(\bu,\bv)$ be the set of ${\bf J} \in \cJ$ such that $r_{\bf J} (\bu) \ge r_{\bf J} (\bv)$,
and let $\cJ_2 := \cJ_2(\bu,\bv)$ be the set of ${\bf J} \in \cJ$ such that $r_{\bf J} (\bu) < r_{\bf J} (\bv)$;
thus~$\cJ$ is the disjoint union of $\cJ_1$ and $\cJ_2$.

The contribution from jump ${\bf J} \in \cJ_1$ to the generator for the distance between the two coupled chains is
given by $N^{-1} (\| {\bf U} - {\bf V} + {\bf J} \|_M - \| {\bf U} - {\bf V} \|_M)$, 
multiplied by the excess in the rate of the transition for the
copy started at $\bf u$, which is $N (r_{\bf J} (\bu) - r_{\bf J} (\bv))$.  Similarly for ${\bf J} \in \cJ_2$ the contribution is
$N^{-1} (\| {\bf U} - {\bf V} - {\bf J} \|_M - \| {\bf U} - {\bf V} \|_M)$ multiplied by 
$N(r_{\bf J} (\bv) - r_{\bf J} (\bu))$.

Thus, recalling that $\cA^N$ is defined to be the generator of the coupling, 
\eqs
\lefteqn{\cA^N H({\bf U}, {\bf V})} \\
&&\Eq\sum_{{\bf J} \in \cJ_1} \big( \| {\bf U} - {\bf V} + {\bf J} \|_M - \| {\bf U} - {\bf V} \|_M \big)
N \big( r_{\bf J} (\bu) - r_{\bf J} (\bv) \big) \\
&&\mbox{}\qquad + \sum_{{\bf J} \in \cJ_2} \big( \| {\bf U} - {\bf V} - {\bf J} \|_M - \| {\bf U} - {\bf V} \|_M \big)
N \big( r_{\bf J} (\bv) - r_{\bf J} (\bu) \big).
\ens
For each ${\bf J} \in \cJ_1$, we see that
\eqs
\lefteqn{\Big| \big( \| {\bf U} - {\bf V} + {\bf J} \|_M
            - \| {\bf U} - {\bf V} \|_M \big) N \big( r_{\bf J} (\bu) - r_{\bf J} (\bv) \big)} \\
 &&\qquad\qquad\qquad\qquad\qquad\mbox{}  - \frac{ \langle {\bf U} - {\bf V}, {\bf J} \rangle_M} {\| {\bf U} - {\bf V} \|_M}
      \nabla r_{\bf J}(\bc)({\bf U} - {\bf V}) \Big|\\
&&\Le \Big| \| {\bf U} - {\bf V} + {\bf J} \|_M - \| {\bf U} - {\bf V} \|_M
       - \frac{ \langle {\bf U} - {\bf V}, {\bf J} \rangle_M} {\| {\bf U} - {\bf V} \|_M}
            \Big| N \Big| \big( r_{\bf J} (\bu) - r_{\bf J} (\bv) \big) \Big| \\
&&\mbox{}\qquad + \frac{ | \langle {\bf U} - {\bf V}, {\bf J} \rangle_M |} {\| {\bf U} - {\bf V} \|_M}
     \Big| N \big( r_{\bf J} (\bu) - r_{\bf J} (\bv) \big)
        - \nabla r_{\bf J}(\bc)({\bf U} - {\bf V}) \Big|,
\ens
which, using Lemma~\ref{lem:inner} and~\Ref{eq:error}, is in  turn at most
$$
\frac{2 \| {\bf J} \|_M^2 c_{\bf J}}{\| {\bf U} - {\bf V} \|_M} \|{\bf U} - {\bf V}\|_M
+ \| {\bf J} \|_M \eps \| {\bf U} - {\bf V} \|_M.
$$
An identical bound holds for ${\bf J} \in \cJ_2$.

Hence
\eqs
\lefteqn{\left| \cA^N H({\bf U}, {\bf V})
  - \sum_{{\bf J} \in \cJ} \frac{ \langle {\bf U} - {\bf V}, {\bf J} \rangle_M} {\| {\bf U} - {\bf V} \|_M} 
           \nabla r_{\bf J}(\bc)({\bf U} - {\bf V}) \right|}\\
&&\Le
2 \sum_{{\bf J} \in \cJ} \| {\bf J} \|_M^2 c_{\bf J}
+ \eps \| {\bf U} - {\bf V} \|_M \sum_{{\bf J} \in \cJ} \| {\bf J} \|_M.
\ens

We also have 
\begin{eqnarray*}
\lefteqn{ \sum_{{\bf J} \in \cJ} \frac{ \langle {\bf U} - {\bf V}, {\bf J} \rangle_M} {\| {\bf U} - {\bf V} \|_M}
\nabla r_{\bf J}(\bc)({\bf U} - {\bf V})} \\
&=&
   \frac{1}{\| {\bf U} - {\bf V} \|_M} \left\langle {\bf U} - {\bf V}, 
               \left(\sum_{{\bf J} \in \cJ} {\bf J} \nabla r_{\bf J}(\bc)\right) ({\bf U} - {\bf V}) \right\rangle_M \\
&=& \frac{\langle {\bf U} - {\bf V}, A({\bf U} - {\bf V}) \rangle_M}{\| {\bf U} - {\bf V} \|_M}\\
&\le& - \rho' \| {\bf U} - {\bf V} \|_M,
\end{eqnarray*}
where the final inequality follows from Lemma~\ref{lem:matrix-M}. Hence
\begin{eqnarray*}
\cA^N H({\bf U}, {\bf V})
&\le& - \rho' \| {\bf U} - {\bf V} \|_M
+ 2 \sum_{{\bf J} \in \cJ} \| {\bf J} \|_M^2 c_{\bf J}
+ \eps \| {\bf U} - {\bf V} \|_M \sum_{{\bf J} \in \cJ} \| {\bf J} \|_M \\
&\le& \| {\bf U} - {\bf V} \|_M \left( -\rho' + \eps \sum_{{\bf J} \in \cJ} \| {\bf J} \|_M
+ \frac{2 \sum_{{\bf J} \in \cJ} \| {\bf J} \|_M^2 c_{\bf J}}{\| {\bf U} - {\bf V}\|_M} \right).
\end{eqnarray*}
As before, \adbb{provided $\eps$ is as in~\Ref{ADB-eps-def1}} and $K_2$ is large enough,
we have
$$
\cA^N H({\bf U}, {\bf V}) \Le - \rho H({\bf U},{\bf V}) \quad\mbox{for all } \bu,\bv \in B_M(\bc,\delta(\eps)) \mbox{ and }
  \| {\bf U} - {\bf V} \|_M \ge K_2
$$
This establishes the final part of the theorem, \adbb{again with $\d_0 = \d(\eps)$ for~$\eps$ as in~\Ref{ADB-eps-def1}.}
\end{proofof}

We end this section by giving a lemma describing the possible behaviours of our set of jumps $\cJ$, in particular,
showing that Assumption~2 is a natural one, and only rules out fairly pathological cases.
The lemma also gives an upper bound (in the case when the set of jumps $\cJ$ is spanning) on the length of a
shortest path between two states made up of jumps in $\cJ$, in terms of the number of jumps required,
and the total distance covered by the jumps on the path.

\begin{lemma} \label{intvec}
Let $\mathcal P$ be any set of integer vectors in $\Z^d$.   
Then one of the following holds:
\begin{itemize}
\item [(i)]
there is some non-zero vector ${\bf v} \in \R^d$ such that ${\bf v} \cdot {\bf p} \ge 0$ for all ${\bf p} \in {\mathcal P}$; 
\item [(ii)]
there is some strict sublattice of $\Z^d$ containing each vector in ${\mathcal P}$; 
\item [(iii)]
there is some finite subset ${\mathcal Q} = \{ {\bf q}_1, \cdots , {\bf q}_k \}$ of ${\mathcal P}$ that is spanning.
\end{itemize}

Moreover, in case {\rm (iii)}, for any norm $\|\cdot \|_M$ on $\R^d$, there are constants $\mu$ and $\nu$ such that
every vector ${\bf z} \in \Z^d$ can be written as ${\bf q}_1 + \cdots + {\bf q}_n$, where each ${\bf q}_i$ is in $\mathcal Q$, 
$n \le \mu \| {\bf z} \|_M$ and $\sum_{i=1}^n \| {\bf q}_i \|_M \le \nu \| {\bf z}\|_M$.    
\end{lemma}

\begin{proof}
Let ${\mathcal V}$ be the set of all non-negative integer combinations of vectors in $\mathcal P$.  If ${\mathcal V}$ is closed under negation,
then it forms a sublattice of $\Z^d$, so either we are in case~(ii), or ${\mathcal V} = \Z^d$ (and so $\mathcal P$ is spanning). 

Suppose instead that there is some vector ${\bf z} \in {\mathcal V}$ such that $-{\bf z}$ is not in ${\mathcal V}$.  
We will show that condition (i) holds.  

Consider the convex hull ${\rm conv}({\mathcal V})$ of $\mathcal V$: note that this is a cone (closed under multiplication by positive scalars).  
If $-{\bf z} \in {\rm conv}({\mathcal V})$, then, by Carath\' eodory's Theorem, $-{\bf z}$ is a convex combination of (at most) $d+1$ 
vectors ${\bf v}_1, \cdots, {\bf v}_{d+1}$ in ${\mathcal V}$.  
  
There is thus a solution to the linear system 
$$
\lambda_1 {\bf v}_1 + \cdots + \lambda_{d+1} {\bf v}_{d+1} + {\bf z} = {\bf 0};  \quad \lambda_1 + \cdots + \lambda_{d+1} = 1
$$
with all the $\lambda_i$ non-negative.  As the ${\bf v}_i$ and ${\bf z}$ are integer vectors, this amounts to a system of $d+1$ linear
equations with integer coefficients in $d+1$ unknowns.  Thus there is a solution with the 
$\lambda_i$ all rational.  Clearing denominators, this entails a solution to 
$$
\mu_1 {\bf v}_1 + \cdots + \mu_{d+1} {\bf v}_{d+1} + \mu_0 {\bf z} = {\bf 0}
$$
with all the $\mu_i$ non-negative integers, and $\mu_0 \not= 0$.  Thus we can write
$$
-{\bf z} = \mu_1 {\bf v}_1 + \cdots + \mu_{d+1} {\bf v}_{d+1} + (\mu_0-1) {\bf z} \, \in {\mathcal V},
$$
which is a contradiction. 

We conclude that $-{\bf z} \notin {\rm conv}({\mathcal V})$, and hence there is a hyperplane in $\R^d$ separating $-{\bf z}$ from ${\rm conv}({\mathcal V})$.  
Take ${\bf v}$ to be normal to this hyperplane, with ${\bf v} \cdot (-{\bf z}) < 0$.  If any vector $\bf p$ in $\mathcal P$ has 
${\bf v} \cdot {\bf p} < 0$, then some multiple of $\bf p$ lies on the same side of the hyperplane as $- {\bf v}$, a contradiction.   
Therefore, if ${\mathcal V}$ is not closed under negation, condition (i) holds. 

We are left with the case where ${\mathcal V} = \Z^d$.  In particular, each of the $2d$ unit vectors $\pm {\bf e}_i$ can be written as a sum
of vectors in $\mathcal P$.  For each of these unit vectors $\bf f$, we let $\ell = \ell({\bf f})$ be the smallest number so that 
there is a multiset ${\bf p}_1, \cdots, {\bf p}_\ell$ of vectors in $\mathcal P$ with ${\bf p_1} + \cdots + {\bf p}_\ell = {\bf f}$.  
Let $\ell_0$ be the maximum of the $\ell({\bf f})$, and take ${\mathcal Q}$ to be a set of at most $2d\ell_0$ vectors from ${\mathcal P}$ 
so that each of the $\pm {\bf e}_i$ is a sum of at most $\ell_0$ vectors in ${\mathcal Q}$.  
Now, for any vector ${\bf z}$ in $\Z^d$, we can write ${\bf z}$ as a sum of $\| {\bf z} \|_1$ vectors $\pm {\bf e}_i$, and thus 
as a sum of at most $\ell_0 \| {\bf z} \|_1$ vectors in ${\mathcal Q}$.  Thus $\mathcal Q$ is spanning and condition (iii) holds.  

Furthermore, in this last case, as all norms in $\R^d$ are equivalent, there is a constant $\gamma$ 
such that $\| {\bf z} \|_1 \le \gamma \| {\bf z} \|_M$ for all ${\bf z} \in \Z^d$, so the number of vectors in the sum making up 
any vector ${\bf z}$ is at most $\ell_0 \gamma \| {\bf z} \|_M := \mu \| {\bf z} \|_M$, and the sum of the $M$-norms of the summands 
is at most $\max \big\{ \| {\bf q} \|_M : {\bf q} \in {\mathcal Q} \big\} \mu \| {\bf z} \|_M := \nu \| {\bf z} \|_M$, 
as claimed.   
\end{proof}

\begin{remark} \label{ADB-spanning-rk}{\rm
\adbb{ Lemma~\ref{intvec} gives an upper bound on the length of any of the paths specified in~\Ref{ADB-irreducible},
in terms of the number of jumps required, and of the total distance covered by the jumps on the path.}}
\end{remark}

\section{Concentration}\label{concentration}
\setcounter{equation}{0}
Under Assumptions 1--3,
we now investigate how fast the process~$\bX^N$ of Section~\ref{Density-dependent} converges to
a quasi-equilibrium (in the sense of \BP\ (2012, \adbm{Theorem~4.1})) near~$N\bc$, if \grbu{$\bx^N(0) = N^{-1} \bX^N(0) \in \cB(\bc)$.}

\subsection{Concentration of measure}
\label{sub:conc} 

\subsubsection{Concentration over short time intervals}

\adbk{This section is concerned with the evolution} of the process $\bx^N \grbu{=} N^{-1}\bX^N$,
 away from the
fixed point~$\bc$.  The results are fairly standard, having their origins in Kurtz~(1970,1971), and having been further 
developed in Darling \& Norris (2008) --- see also the references therein.
We reproduce them here because they are integral to what follows, in a form most useful for our purposes.
The first lemma establishes concentration for a martingale associated with~$\bx^N$, and the second translates this
into concentration over finite intervals of the distribution of~$\bx^N$ around the solution~$\by_{[\bx^N(0)]}$ of the
ODE system~\Ref{eq:diff-eq} with the same starting point.

\begin{lemma}\label{ADB-exp-bound-MG}
Let $\bX^N := (\bX^N(t),\, t \ge 0)$ be a Markov population process on $S_N \subset \Z^d$ with transition rates as given 
in~\Ref{transition-rates}, satisfying \adbm{\hbox{Assumptions~1--3;}} write $\bx^N := N^{-1}\bX^N$.  Let 
$\t_{\cK} := \inf\{t \ge 0\colon\,\bx^N \notin \cK\}$ for 
some compact set~$\cK$.
Define
\eq\label{ADB-MN-def}
    \adbk{\tm^N(t) \Def \bx^N(t) - \bx^N(0) - \int_0^t F(\bx^N(u))\,du.} 
\en
Then~$\tm^N(\cdot \wedge \t_{\cK})$ is a $d$-dimensional zero mean martingale, and, 
for any $z,\adbm{T} > 0$, 
\eqs
     \lefteqn{ \pr\Bigl[\sup_{0 \le t \le T}\|\tm^N(t\wedge \t_{\cK})\adbk{\|_M} \ge z\Bigr] }\\ 
                 &&\Le \z_{N,T,\cK}(z) \Def
                      2d \exp\Bigl\{ -\frac {Nz}{2d\adbk{J^*_M}}\min\Bigl(1,\frac z{\adbj{ed \adbk{J^*_M} T R^*(\cK)}} \Bigr) \Bigr\}.
\ens
\end{lemma}

\begin{proof}
For $\bth \in \re^d$, define
 \[
    Z_{N,\bth}(t) \Def e^{\adbk{\lan \bth, \tm^N(t)\ran_M}} \exp\Bigl\{- \int_0^t \f_{N,\bth}(\bx^N(u-))\,ds\Bigr\},
 \]
where
\[
    \f_{N,\bth}(\by) \Def N \sJJ r_\bJ(\by)\Bigl(e^{N^{-1}\adbk{\lan \bth, \bJ \ran_M} } - 1 - N^{-1}\adbk{\lan \bth, \bJ \ran_M} \Bigr).
\]
Then~$Z_{N,\bth}(\cdot \wedge \t_{\cK})$ is a non-negative finite variation martingale, with $Z_{N,\bth}(0) = 1$, and, for $\by \in \cK$, 
\eqs
    |\f_{N,\bth}(\by)| &\le& \frac N2 \sJJ r_\bJ(\by) e^{N^{-1}|\adbk{\lan \bth, \bJ \ran_M}|} N^{-2}|\adbk{\lan \bth, \bJ \ran_M}|^2 \\
    &\le& N \Bigl\{\sup_{\by' \in \cK}\sJJ r_\bJ(\by')\Bigr\} g(N^{-1}\mJJ |\adbk{\lan \bth, \bJ \ran_M}|) 
           \Le N R^*(\cK) g(N^{-1}\adbk{J^*_M\,\|\bth\|_M}),
\ens
where $g(x) := \half x^2 e^x \le (e/2)x^2$ if \adbk{$0 \le x \le 1$}.

Observe that
\[
    \inf\{t \ge 0\colon\, \|\tm^N(t\wedge \t_{\cK})\adbk{\|_M} > z\}
               \ \ge\ \min_{1\le i\le d} \inf\{t \ge 0\colon\, |\adbk{\lan \tm^N(t\wedge \t_{\cK}),\tbe\uii \ran_M}| > z/d\},
\]
\adbk{where $\tbe\uii$, $0 \le i \le d$, are eigenvectors of~$M$ such that $\lan \tbe\uii, \tbe\uj \ran_M = \d_{ij}$,
$1 \le i,j \le d$.}
Taking $\bth = \th\be\uii$ for \adbk{$\th > 0$} and $i \in [d]$, and considering the stopping time 
$\min\bigl\{T,\inf\{t \ge 0\colon\,\adbk{\lan \tm^N(t\wedge \t_{\cK}),\tbe\uii \ran_M} > z/d\}\bigr\}$, it follows that, 
if $\adbk{\th J^*_M} \le N$, then
\eqs
   \lefteqn{1 \ \ge\ \exp\{\th z/d - N^{-1}\adbk{(e/2)\th^2 (J^*_M)^2 T R^*(\cK)} \}}\\
   &&\hskip1.5in \times\, \pr[\inf\{t \ge 0\colon\,\adbk{\lan \tm^N(t\wedge \t_{\cK}),\tbe\uii \ran_M} > z/d\} \le T].
\ens
Taking $\th := \adbk{(N/J^*_M)\min\{1,(z/d)/(e J^*_M T R^*(\cK))\}}$, 
for which choice
\[
    N^{-1}\adbk{(e/2)\th^2(J^*_M)^2} T R^*(\cK)  \Le \th z/(2d),
\]
it follows that
\eqs
   \lefteqn{\pr[\inf\{t \ge 0\colon\,\adbk{\lan \tm^N(t\wedge \t_{\cK}),\tbe\uii \ran_M} > z/d\} \le T] }\\
      && \Le \exp\Bigl\{ -\frac {Nz}{2d\adbk{J^*_M}}\min\Bigl(1,\frac z{\adbk{ed J^*_M T R^*(\cK)}} \Bigr) \Bigr\},
\ens
yielding the lemma.
\end{proof}

Let
\eq\label{ADB-apr-1}
     \cY_\e(\bx,T) \Def \{\by' \in \re^d\colon\, \inf_{0\le t\le T}\|\by' - \by_{[\bx]}(t)\adbk{\|_M} \le \e\}
\en
denote the set of all points
within \adbk{$M$-distance}~$\e$ of the set $(\by_{[\bx]}(t),\,0 \le t\le T) \subset \re^d$.

\begin{lemma}\label{ADB-exp-bound-XN}
Let $\bX^N := (\bX^N(t),\, t \ge 0)$ be a Markov population process on $S_N \subset \Z^d$ 
satisfying Assumptions~1--3, and let $\bx^N(t) := N^{-1}\bX^N(t)$.
\adbu{Let~$\cK$ be a compact subset of~$\widehat S$}. 
Define 
\[A(T,\e) \Def \bigl\{\inf\{t > 0\colon\, \|\bx^N(t)- \by_{[\bx^N(0)]}(t)\adbk{\|_M} > \e\} \le \adbm{T} \bigr\}.\]
Then, if $\cY_\e(\bx^N(0),T) \subset \cK$, 
it follows that
\[
   \pr[A(T,\e)] \Le \z_{N,T,\cK}(\e e^{-TL(\cK)}),
\]
where $\z_{N,t,\cK}$ is as defined in Lemma~\ref{ADB-exp-bound-MG}.
\end{lemma}

\begin{proof}
 From~\Ref{ADB-MN-def},
 \[
      \bx^N(t) \Eq \bx^N(0) + \int_0^t F(\bx^N(u))\,du + \tm^N(t),\qquad 0 \le t \le T,
 \]
and the solution~$\by := \by_{[\bx^N(0)]}$  of the ODE system~\Ref{eq:diff-eq} satisfies
\[
    \by(t) \Eq \bx^N(0) + \int_0^t F(\by(u))\,du, \qquad 0 \le t \le T.
\]
Taking the difference of these two equations, a standard Gronwall argument gives
\[
    \sup_{0 \le t \le T}\|\bx^N(t) - \by_{[\bx^N(0)]}(t)\adbk{\|_M} \Le e^{L(\cK)T} \sup_{0 \le t \le T}\|\tm^N(t)\adbk{\|_M},
\]
\grba{and therefore
\eqs
\pr[A(T,\e)] &\le& \pr \biggl( \sup_{0 \le t \le T}\|\bx^N(t) - \by_{[\bx^N(0)]}(t)\adbk{\|_M} \ge \eps \biggr) \\
    &\le& \pr \biggl( \sup_{0 \le t \le T}\|\tm^N(t)\adbk{\|_M} \ge \eps e^{-T L(\cK)} \biggr).
\ens}
The lemma now follows from Lemma~\ref{ADB-exp-bound-MG}, since the event 
$A(T,\e)$ is the same, whether~\adbk{$\bx^N$} is stopped outside $\cY_\e(\bx^N(0),T)$ or not.
\end{proof}

Typically, when applying the lemma, we can take $\cK = \cY_{\e_0}(\bx^N(0),T)$, for a suitable chosen~$\e_0$, or even
$\cK = B_M(\bc,\d_0)$.

The conclusion of Lemma~\ref{ADB-exp-bound-XN} is that the random process~$\bx^N$ is concentrated around the
deterministic path over any bounded time interval.  The factor $e^{-TL(\cK)}$ in the bound means that the concentration
can weaken exponentially fast as the length of the interval increases, as is appropriate if the solutions of the deterministic
equations from close initial points diverge from one another over time.  However, in the neighbourhood of an attracting
equilibrium of the deterministic equations, this does not happen, and more can be said.  This is the substance of the
next section.

\subsubsection{Concentration near $\bc$}

We now use Theorem~\ref{thm:main1} above to show that the distribution of~$\bX^N$ stays
concentrated around \adbk{the mean of the truncated version~$\bX\undi$} for a long time, 
\adbk{provided that~$\bX^N(0)$ is close enough to~$\bc$}.  
For this purpose, we first recall~\BBL~(2017, Theorem 3.1), 
which provides a concentration inequality for contracting continuous time jump Markov chains.

\begin{theorem}
\label{Wconcentration}
Let $Q:=(Q(x,y): x,y\in S)$ be the $Q$-matrix of a stable, conservative, non-explosive continuous-time Markov chain
$X:=(X(t),\,t\ge0)$ with discrete state space $S$.
Writing $q_x = - Q(x,x)$, let $\widehat{S}$ be a subset of $S$, 
for which $q = \sup_{x \in \widehat{S}} \{q_x\} < \infty$. Let $f\colon S \to \R$ be a function such that 
$(P^t f)(x) = \E_x f (X(t))$ exists for all $t \ge 0$ and $x \in S$, and suppose that $\beta$ is a constant such that
\begin{eqnarray}
\label{cond-1}
|(P^s f)(x) - (P^s f) (y) | \le \beta
\end{eqnarray}
for all $s \ge 0$, all $x \in \widehat{S}$ and all $y \in N(x)$, where $N(x) = \{y \in S: Q(x,y) > 0\}$.
Assume also that the continuous function $\alpha: \R^+ \to \R^+$ satisfies
\begin{eqnarray}
\label{cond-2}
\sum_{y \in S} Q(x,y) \Big ( (P^s f)(x) - (P^s f)(y) \Big )^2 \le \alpha (s),
\end{eqnarray}
for all $x \in \widehat{S}$ and $s \ge 0$.

Define $a_t = \int_{s=0}^t \alpha (s) ds$. Finally, let $A_t := \{X(s) \in \widehat{S} \mbox{ for all } 0 \le s < t\}$. Then, for all $x_0 \in \widehat{S}$, $t \ge 0$ and $m \ge 0$,
\begin{eqnarray}
\label{conc-ineq}
\P_{x_0} \Big ( \Big \{ |f(X(t)) - (P^t f)(x_0)| > m \Big \} \cap A_t \Big ) \le  2 e^{-m^2/(2a_t + 2 \beta m/3)}.
\end{eqnarray}
\end{theorem}

\begin{remark}
\label{remark-conc-ineq}{\rm
Clearly, if there exists a time $t_0 > 0$ such that~\eqref{cond-1} and~\eqref{cond-2} are satisfied
for all $s \le t_0$ (rather than necessarily for all $s \ge 0$), then the conclusion~\eqref{conc-ineq} of
Theorem~\ref{Wconcentration} will also hold for all $t \le t_0$.}
\end{remark}

Theorem~\ref{conc-thm} below is a straightforward application of Theorem~\ref{Wconcentration} and Remark~\ref{remark-conc-ineq}.

Let $(X(t),\,t\ge0)$ be a continuous time Markov chain on a discrete state space~$S$, with
transition rates $(Q (x,y),\,x,y\in S)$; 
let~$d$ be a metric on~$S$.

Suppose that a coupling of two copies $(X\ui,X\ut)$ of~$X$, with starting states $x,y$ respectively, can be defined,
with the property that, for some $S_0 \subseteq S$ and $D,\r > 0$,
\eq\label{exponentially-close}
   \ex_{x,y} d(X\ui(t),X\ut(t)) \Le De^{-\r t}, \quad 0\le t\le t_0,
\en
for all $x \in S_0$ and $y \in N(x) = \{y\in S\colon Q (x,y) > 0\}$.  Define
\[
    A_t \Def \{X(s) \in S_0,\ 0\le s < t\}.
\]

\begin{theorem}\label{conc-thm}
Let $X$ be  a stable, conservative, non-explosive continuous-time Markov chain with discrete state space $S$ and $Q$-matrix $Q$.
Suppose that $q_0 := \sup_{x \in S_0}\sum_{y\in N(x)} Q(x,y) < \infty$.  Suppose there is a coupling of two copies
$(X\ui,X\ut)$ of~$X$ such that~\eqref{exponentially-close} holds. Suppose further that
$f\colon S \to \R$ satisfies $|f(x) - f(x')| \le L d(x,x')$ for all $x,x'\in S$. It then follows
that, for all $x_0 \in S_0$, $0 < t \le t_0$ and $m > 0$, we have
\[
  \Pr_{x_0} \Bigl[ \Bigl\{\Bigl| f(X(t))-\E_{x_0} [f(X(t))] \Bigr|\ge m\Bigr\}\cap A_t \Bigr]
    \Le 2\exp\left( - \frac {m^2}{\adbj{q_0L^2D^2/\r  + 2LD m/3}}\right).
\]
\end{theorem}

\begin{proof}
\adbj{In view of~\Ref{exponentially-close}, writing $(P^t f)(x) := \E_x f (X(t))$ as in Theorem~\ref{Wconcentration},
we have $|(P^s f)(x) - (P^s f)(y)| \le LDe^{-\r s}$ for all $x \in S_0$, $y \in N(x)$ and $0 \le s \le t_0$, and hence
\[
     \sum_{y' \in N(x)} Q(x,y') \Bigl( (P^s f)(x) - (P^s f)(y') \Bigr)^2 \le q_0 L^2 D^2 e^{-2\r s}.
\]
Thus we can take $\b = LD$ in~\Ref{cond-1} and $a_t \le q_0 L^2 D^2/(2\r)$, together with $\hS = S_0$,
in Theorem~~\ref{Wconcentration}, which, together with Remark~\ref{remark-conc-ineq}, proves the theorem.}
\end{proof}

It would be easy to combine Part~(iv) of Theorem~\ref{thm:main1} with a supermartingale argument
to establish~\Ref{exponentially-close}
for coupled copies of \adbb{processes~$\bX^N$ satisfying Assumptions 1--3, taking $d(\bX,\bY)$ to be $\|\bX-\bY\|_M$,}
were it not for the requirement
in Theorem~\ref{thm:main1}~(iv) that \adbm{$H(\bU,\bV) \ge K_2$}.  This entails some extra work.
For later use, we  define
\begin{equation}
\label{def-K3}
K_3 := \max \{K_2, 8 J_M^*\}
\end{equation}  
and
\eq\label{ADB-hz-def}
   \hz_{N}(z) \Def \z_{N,\r^{-1},B_M(\bc,\d_0)}(z e^{-\r^{-1}L(B_M(\bc,\d_0))}),
\en
where~$\z_{N,T,\cK}(z)$ is as defined in Lemma~\ref{ADB-exp-bound-MG}.
Note that thus
\eq\label{ADB-hz-bnd}
    \hz_{N}(z) \Eq \adbe{2d} \exp\{-Nk_1 z \min(1,k_2z)\} \quad\mbox{for constants } k_l = k_l(\d_0),\ l=1,2.
\en

\begin{proposition}\label{ADB-3.5}
 Let $\adbj{\ts_1}(\d) := \inf\{t \ge 0\colon\, \|\bx^N(t) - \bc\|_M > \d\}$.  Then, for any \adbe{$0 < \d' < \d \le \d_1$},
and for any $T_N > 0$,
 \[
     \pr_{\by'}[\adbj{\ts_1}(\d) \le T_N] \Le \lceil \r T_N \rceil \hz_N(\e'),
 \]
uniformly for $\by' \in B_M(\bc,\d')$, where $\e' := \adbk{\tfrac1{2}} (\d - \d')$ if $\d' \ge \d/(3 - 2e^{-1})$,
and $\e' := \adbk{\d(1 - e^{-1})/(3 - 2e^{-1})}$ otherwise.
\end{proposition}

\begin{proof}
First, consider $\d' \ge \d/(3 - 2e^{-1})$, for which values of~$\d'$ we have
\adbk{$\e' = \tfrac1{2} (\d - \d') \le \d'(1 - e^{-1})$}.
For such~$\d'$, and for any $\by'  \in B_M(\bc,\d')$,
let~$\by_{[\by']}$ denote the solution to~\Ref{eq:diff-eq} with $\by(0) = \by'$.  Then,
by the definition of~$\e'$, and in view of Theorem~\ref{thm:main1} (i),
 $\cY_{\e'}(\by',\r^{-1}) \subset B_M(\bc,\d)$, and $\|\by_{[\by']}(\r^{-1}) - \bc\|_M \le \d' e^{-1}$, so that,
if \adbm{$|\bz| \le \e'$}, then
\[
   \|\by_{[\by']}(\r^{-1}) + \bz - \bc\|_M \Le \d' e^{-1} + \|\bz\|_M \Le \d' e^{-1} +  \d'(1 - e^{-1}) \Eq \d'.
\]
 Hence, \adbm{taking $\bz = \bx^N(\r^{-1}) - \by_{[\by']}(\r^{-1})$, the event}
$$
  \tA \Def \{\adbj{\ts_1}(\d) \le \r^{-1}\} \cup \{\bx^N(\r^{-1}) \notin B_M(\bc,\d')\}
$$
\adbm{is contained in the event $A(\r^{-1},\e')$, as defined in Lemma~\ref{ADB-exp-bound-XN}. It thus}
follows from Lemma~\ref{ADB-exp-bound-XN} 
that
\[
     \pr_{\by'}[\tA] \Le \hz_N(\e'),
\]
uniformly for $\by' \in B_M(\bc,\d')$, and, on $\tA^c$, $\bx^N(\r^{-1}) \in B_M(\bc,\d')$.  Using the Markov property to
repeat the argument $\lfloor \r T_N \rfloor$ times for the time intervals $[j\r^{-1},(j+1)\r^{-1}]$
for $1 \le j \le \lfloor \r T_N \rfloor$, the conclusion of the proposition follows for $\d' \ge \d/(3 - 2e^{-1})$.

If $\by' \in B_M(\bc,\d')$ for $\d' < \d/(3 - 2e^{-1})$, then $\by' \in B_M(\bc,\d/(3 - 2e^{-1}))$, and the conclusion
for $\d' = \d/(3 - 2e^{-1})$ can be invoked, completing the proof.
\end{proof}

We are now in a position to prove the version of~\Ref{exponentially-close} that we need.
For $\bX \in S_N$, define its set of neighbours by
\eq\label{ADB-nbhd-def}
   N(\bX) \Def \{\bX' \in S_N\colon\, \bX' = \bX + \bJ \mbox{ for some }\bJ \in \cJ\},
\en
\adbe{and, for any $\d > 0$, write
\eq\label{ADB-calX-def}
   \cX^N(\d) \Def B_M(N\bc,N\d)\cap S_N.
\en
}

\begin{lemma}\label{ADB-coupling-lemma}
 Under Assumptions 1--3, \adbi{with~$\d_1$ is as in Remark~\ref{ADB-d_1-def},} we can
construct a coupling $(\wX_1^N,\wX_2^N)$ of copies \adbi{of~$\bX^{N,\d_1}$} such that there exists
$0 < \trh \le \r$, as well as $\a, D > 0$ and $N_0 \in \nat$, with the property that, for all $N \ge N_0$, 
\eq\label{exponentially-close-2}
   \ex_{\bX,\bX'} \|\wX_1^N(t) - \wX_2^N(t)\|_M \Le De^{-\trh t}, \quad 0\le t\le \adbd{\r^{-1}}\a N,
\en
for $\bX \in \cX^N(\d_1/2)$ and $\bX' \in N(\bX)$.
\end{lemma}

\smallskip

\begin{proof}   
\adbi{Throughout the proof, 
we take~$\nu$ to be as given in Lemma~\ref{intvec}, and write
\eq\label{ADB-tilde-nu-def}
  \tnu \Def \nu + 1/2.
\en
}%
Using~\eqref{def-K3}, it then follows that, \adbi{for any $\bJ\ui,\bJ\ut \in \cJ$},
\eq
\begin{aligned}\label{ADBi-1}
   \|(\bX+\bJ\ui) - (\bX'+\bJ\ut)\|_M &\ \ge\ K_3/2 \quad \mbox{if}\ \|\bX-\bX'\|_M > K_3; \\
   \|(\bX+\bJ\ui) - (\bX'+\bJ\ut)\|_M &\Le K_3 \tnu \quad \ \mbox{if}\ \|\bX-\bX'\|_M \le \nu K_3,
\end{aligned} 
\en
\adbi{where $K_3$ is as in~\Ref{def-K3}.}
\ignore{if $\|\bX-\bX'\|_M > K_3$, then
$\|(\bX+\bJ\ui) - (\bX'+\bJ\ut)\|_M \ge K_3/2$, for any $\bJ\ui,\bJ\ut \in \cJ$. On the other hand, if
$\|\bX-\bX'\|_M \le \nu K_3$, then $\|(\bX+\bJ\ui) - (\bX'+\bJ\ut)\|_M \le K_3 (\nu  + 1/2)$.}

\smallskip

For $(\bU^N,\bV^N)$ as in  Part~(iv) of Theorem~\ref{thm:main1}, let
\eqs
   \t &:=& \inf\{t \ge 0\colon\, \|\bU^N(t) - \bV^N(t)\|_M \le K_3\}; \\
   \sdd(\d) &:=& \inf\bigl\{t \ge 0\colon\,
               \max\{\|\bU^N(t) - N\bc\|_M,\|\bV^N(t) - N\bc\|_M\} \ge N\d - \adbi{\nkt}\bigr\},
\ens
and write $\sdd := \sdd(\d_1)$.  \adbb{Note that, if $N \ge 16\adbi{\nkt}/\d_1$,
then $\d_1 - N^{-1}\adbi{\nkt} \ge 15\d_1/16$ and
$ \tfrac12 \d_1 +  N^{-1}\adbi{\nkt} \le 9\d_1/16$.  Thus,
defining $\e'(\d_1) := 3\d_1/\adbk{16}$, it follows from Proposition~\ref{ADB-3.5}
\adbm{with $\d = 15\d_1/16$ and $\d' = 9\d_1/16$} that, for any $T_N > 0$,
and for any $\bX \in \cX^N(\d_1/2)$ and $\bX' \in N(\bX)$,
\eq\label{ADB-extremes-bnd}
    \P[\sdd \le T_N \giv (\bU^N(0),\bV^N(0)) = (\bX,\bX')]  \Le \lceil \r T_N \rceil \hz_N(\e'(\d_1)) \ =:\ p_N(\d_1,T_N),
\en
say,} 
\adbm{for all $N \ge N_2 := \max\{N_1,16\adbi{\nkt}/\d_1\}$,
where~$N_1$ is such that the conclusion of Theorem~\ref{thm:main1}\,(iv) holds for 
all $N \ge N_1$.}
\adbd{Recalling~\Ref{ADB-hz-bnd},
we have} $\hz_N(\e'(\d_1)) \le a_1 e^{-Na_2}$, for finite constants $a_1,a_2$, 
\adbm{fixed by the choices of~$\d_0$ and~$\d_1$ in Theorem~\ref{thm:main1} and Remark~\ref{ADB-d_1-def}}.
As a result, we 
deduce that, \adbi{for all~$N \ge N_2$},
\eq\label{ADB-alpha-def0}
    \adbi{p_N(\d_1,T)  \Le a_1 a_2 N e^{-a_2 N} \mbox{ for all } T \le a_2 N/(2\r).}
\en

\adbm{From now on, suppose always that $N \ge N_2$.}
Observe, from Part~(iv) of Theorem~\ref{thm:main1}, that
$$
    \bigl\{e^{\r (t\wedge\t\wedge\sdd)}H\bigl(\bU^N(t\wedge\t\wedge\sdd), \bV^N(t\wedge\t\wedge\sdd)\bigr),\,t\ge0\bigr\}
$$
is a non-negative supermartingale.  \adbi{Using $\E_{\bX,\bX'}$ to denote expectation conditional on
$(\bU^N(0),\bV^N(0)) = (\bX,\bX')$, it thus follows that
\eqs
     \lefteqn{\E_{\bX,\bX'}\{e^{\r\t}(K_3/2)I[\t \le \sdd]\}
     + \E_{\bX,\bX'}\{e^{\r\sdd}H\bigl(\bU^N(\sdd), \bV^N(\sdd)\bigr)I[\sdd \le \t]\}} \\
        &&\Le H(\bX,\bX'), \phantom{XXXXXXXXXXXXXXXXXXXXXXXX}
\ens
so that, for all $\bX,\bX'$ such that $\|\bX-\bX'\|_M \le K_3 \tnu$, \non
\eqa
     &&\E_{\bX,\bX'}\{e^{\r\t}I[\t \le \sdd]\} \Le 2\tnu; \label{ADB-coupling-1}\\
     &&\E_{\bX,\bX'}\{e^{\r\sdd}H\bigl(\bU^N(\sdd), \bV^N(\sdd)\bigr)I[\sdd < \t]\} \Le K_3 \tnu. 
        \phantom{XXXXX}\label{ADBi-2}
\ena
}

Suppose that $\bX,\bX' \in \cX^N(\d_1)$ with $\|\bX-\bX'\|_M \le K_3$.
By Assumption~2, it is possible to write
\eq\label{ADBi-0}
   \bX' \Eq \bX + \sum_{k=1}^n {\bf J}_k,
\en
where ${\bf J}_k \in \cJ$ for $k=1, \ldots, n$, and, by Lemma~\ref{intvec}, there exists~$\nu$ (the same for all $\bX,\bX'$) such that
$\sum_{k=1}^n \| {\bf J}_k \|_M \le \nu \| \bX - \bX'\|_M \le \nu K_3$. 

Now, for  two {\it independent\/} copies $\wX^N_1$ and~$\wX^N_2$ \adbi{of~$\bX^{N,\d_1}$}, starting
with $\wX^N_1(0) = \bX$ and~$\wX^N_2(0) = \bX'$, define the events
\eqs
    E_1 &:=& \{\wX^N_1(1/\adbi{(\r N)}) = \wX_2^N(1/\adbi{(\r N)})\};\\
    E_2 &:=& \{\|\wX^N_1(t) - \wX^N_2(t)\|_M \le \nu K_3,\ 0\le t\le 1/\adbi{(\r N)}\},
\ens
where $\nu > 0$ is as in Lemma~\ref{intvec}.
Then, for all $\bX,\bX' \in \cX^N(\d_1)$ with $\|\bX-\bX'\|_M \le K_3$
\adbi{and also with $\max\{\|\bX - N\bc\|_M,\|\bX' - N\bc\|_M\} \le N\d_1 - \nkt$}, we have
\eq\label{ADB-coupling-2}
     \P[E_1 \cap E_2 \giv (\wX^N_1(0),\wX^N_2(0)) = (\bX,\bX')]
             \ \ge\ \hktt \ >\ 0,
\en
\adbi{where $\hktt$ can be taken to be the same for all~$N$.}
To see this, \adbb{as in Remark~\ref{ADB-d_1-def},}
note that, for instance, the second process could make no jumps in the time interval $[0,1/\adbi{(\r N)}]$,
and the first
could follow a path between $\bX$ and~$\bX'$, \adbi{made up of the jumps in~\Ref{ADBi-0} in some order,}
without ever being further than $\nu K_3$ from~$\bX'$, \adbi{and thus remaining in~$\cX^N(\d_1)$}.
The probability of such an event is \adbi{uniformly bounded below by a positive quantity
that does not depend on~$N$, in view of Remark~\ref{ADB-d_1-def}.}

 So we construct a coupling~$(\wX_1^N,\wX_2^N)$ of
copies \adbi{of~$\bX^{N,\d_1}$}, as follows.
If $\|\wX^N_1(0) - \wX^N_2(0)\|_M > K_3$,
they evolve in a first stage with the transition rates of the pair $(\bU^N,\bV^N)$, until either
\eq
  \begin{aligned}\label{ADBi-a1}
   &\|\wX^N_1(t) - \wX^N_2(t)\|_M\ \le\ K_3\ \mbox{ or }   \\
   &\max\{\|\wX_1^N(t) - N\bc\|_M,\|\wX_2^N(t) - N\bc\|_M\}\ \ge\  N {\d_1} - \adbi{\nkt}. 
   \end{aligned}
\en
In the former \adbi{of these events}, the two processes start a second stage, in which they evolve independently until
either
\eq\begin{aligned} \label{ADBi-a2}
   & \wX^N_1(t) \Eq \wX^N_2(t),\ \mbox{ or } \\
   & \|\wX^N_1(t) - \wX^N_2(t)\|_M\ >\ \nu  K_3,\ \mbox{ or }  \\
   & \max\{\|\wX_1^N(t) - N\bc\|_M,\|\wX_2^N(t) - N\bc\|_M\}\ \ge\  N {\d_1} - \adbi{\nkt}; 
   \end{aligned}
\en
they then continue identically in the first event, and restart the first stage if
the second event happens, with the construction repeating.
\adbi{If, at any time,} \hbox{$\max\{\|\wX_1^N(t) - N\bc\|_M,\|\wX_2^N(t) - N\bc\|_M\}\ \ge\  N {\d_1} - \adbi{\nkt}$,}
the processes continue independently. If $\|\wX^N_1(0) - \wX^N_2(0)\|_M \le K_3$,
the construction starts in the second stage.
\adbi{For use with this coupling construction, define the stopping times
\eqa
   \t_1 &:=& \inf\{t \ge 0\colon\, \|\wX_1^N(t) - \wX_2^N(t)\|_M \le K_3\}; \non\\
   \sdd_1 &:=& \inf\bigl\{t \ge 0\colon\,
            \max\{\|\wX^N_1(t) - N\bc\|_M,\|\wX^N_2(t) - N\bc\|_M\} \ge N\d_1 - \adbi{\nkt}\bigr\};\non\\
   \t_2 &:=& \min\bigl\{1/(\r N),\inf\{t \ge 0\colon\, \|\wX_1^N(t) - \wX_2^N(t)\|_M > \nu  K_3\}\bigr\}. \non \\
   \t_3 &:=& \min\bigl\{\t_1 + 1/(\r N),\inf\{t > \t_1\colon\, \|\wX_1^N(t) - \wX_2^N(t)\|_M > \nu  K_3\}\bigr\}.
   \label{ADB-tau-defs}
\ena
Note also, using~\Ref{ADB-coupling-2} and invoking the strong Markov property, that
\eqa
  \lefteqn{ \pr\bigl[\{\wX_1^N(\t_1 + 1/(\r N)) = \wX_2^N(\t_1 + 1/(\r N))\} }\label{ADBi-0.5}\\
       &\cap& \kern-8pt \{\|\wX^N_1(\t_1 + t) - \wX^N_2(\t_1 + t)\|_M \le \nu K_3,\ 0\le t\le 1/(\r N)\} \giv \cF_{\t_1}\bigr]
         \ \ge\ \hktt, \non
\ena
\adbi{where $\cF_t := \s\{(\wX_1^N(s),\wX_2^N(s)),\,0\le s \le t\}$.}
The main effort in what follows is to find $\trh > 0$ small enough that, writing
\[
    \tcX^N(\d_1) \Def \{(\bZ,\bZ')\colon\,\bZ,\bZ'\in\cX^N(\d_1);\, H(\bZ,\bZ') \le \adbm{\tnkt} \},
\]
and letting $\E_{\bZ,\bZ'}$ denote expectation conditional on $(\wX^N_1(0), \wX^N_2(0)) = (\bZ,\bZ')$,
the quantity
\eq\label{ADB-key-distance-bnd}
  \Psi(T) \Def \max_{(\bZ,\bZ')\in\tcX^N(\d_1)}\, \sup_{0\le t\le T}\,
     \E_{\bZ,\bZ'}\{e^{\trh (t\wedge \sdd_1) } H(\wX^N_1(t\wedge \sdd_1), \wX^N_2(t\wedge \sdd_1)) \}
\en
can be uniformly bounded for any $T > 0$.
\adbm{Note that $\Psi(T) \le 2N\d_1 e^{\trh T} < \infty$ for all finite $T > 0$.}
}

To start with, for
\eq\label{ADBi-3}
   (\bX,\bX')\ \in\ \cX^N_+ \Def  \{(\bZ,\bZ') \in \adbi{\tcX^N(\d_1)\colon\, H(\bZ,\bZ') > K_3}\},
\en
we \adbi{write
\eqa
  \lefteqn{ e^{\trh (t\wedge \sdd_1)} H(\wX^N_1(t\wedge \sdd_1), \wX^N_2(t\wedge \sdd_1))} \non\\
  &=& e^{\trh t}H(\wX^N_1(t), \wX^N_2(t))\bigl\{I[t < \t_1 < \sdd_1] + I[\t_1 \le t < (\t_3 \wedge \sdd_1)]\bigr\}
      \label{ADBi-4}\\
    &&\mbox{}
               + e^{\trh(t\wedge \sdd_1)}H(\wX^N_1(t\wedge \sdd_1), \wX^N_2(t\wedge \sdd_1))
                  \{I[\t_1 < (\t_3 \wedge \sdd_1) \le t] + I[\sdd_1 \le \t_1]\}, \non
\ena
and} bound the expectation of the right hand side in four parts.

First, because
$$
        e^{\r (t\wedge\t_1\wedge\sdd_1)} H\bigl(\wX^N_1(t\wedge\t_1\wedge\sdd_1), \wX^N_2(t\wedge\t_1 \wedge\sdd_1)\bigr)
$$
is a non-negative  supermartingale, we have
\eq\label{ADB-big-tau-1}
   \E_{\bX,\bX'}\{ e^{\r t}H(\wX^N_1(t), \wX^N_2(t)) \adbi{I[t < (\t_1 \wedge \sdd_1)]}\} \Le H(\bX,\bX'),
\en
\adbi{for any $\bX,\bX' \in \adbi{\cX^N(\d_1)}$ such that} $H(\bX,\bX') > K_3$.
From this, it follows that, \adbi{for $(\bX,\bX') \in \cX^N_+$},
\eq\label{ADB-big-tau}
   \E_{\bX,\bX'}\{ e^{\r t}H(\wX^N_1(t), \wX^N_2(t)) \adbi{I[t < \t_1 < \sdd_1]}\} \Le \tilde{\nu} K_3.
\en
Next, for $(\bX,\bX') \in \cX^N_+$, and for any $0 <\trh \le \r$, using \adbi{H\"older's inequality}
as well as~\Ref{ADB-coupling-1}, we have
\eqa
  \lefteqn{ \E_{\bX,\bX'}\{ e^{\trh t}H(\wX^N_1(t), \wX^N_2(t)) I[\t_1 \le t < (\t_3 \wedge \sdd_1)]\} }\non\\
         &\le& \adbi{\nu K_3\, e^{\trh/(\r N)}\E_{\bX,\bX'} \{e^{\trh\t_1} I[\t_1 < \sdd_1] \} } \non\\
        &\le& \adbi{\nu K_3\, e^{\trh/(\r N)}\bigl(\E_{\bX,\bX'} \{e^{\r\t_1} I[\t_1 < \sdd_1] \}\bigr)^{\trh/\r} } \non\\
     &\le&  \adbi{\nu} K_3\, \adbj{\{e^{1/N}(2\nu+1)\}^{\trh/\r}} . \label{ADB-medium-tau}
\ena

\adbi{
For the third term in~\Ref{ADBi-4},  the argument is more involved.  First, note that
$H(\wX^N_1(\t_3 \wedge \sdd_1), \wX^N_2(\t_3 \wedge \sdd_1)) \le \tnu K_3$
when $\t_1 < (\t_3 \wedge \sdd_1)  \le t$, using~\Ref{ADBi-1} and the definition of~$\t_3$.
 Thus, by the strong Markov property, it follows
for $(\bX,\bX') \in \cX^N_+$ and for $0 \le t \le T$ that, when $\t_3 < \sdd_1$, we have
\eqa
  \lefteqn{ \E_{\bX,\bX'}
           \bigl\{  e^{\trh (t \wedge \sdd_1)} H(\wX^N_1(t \wedge \sdd_1), \wX^N_2(t \wedge \sdd_1))
               \giv \cF_{(\t_3 \wedge \sdd_1)} \bigr\}}\non\\
        && \times I[\t_1 < (\t_3 \wedge \sdd_1) \le t] I[\t_3 < \sdd_1] I[\wX^N_1(\t_3) \ne \wX^N_2(\t_3)] \non\\
    &\le& e^{\trh \t_3} \E_{\wX^N_1(\t_3), \wX^N_2(\t_3)}
         \bigl\{ e^{\trh((t \wedge \sdd_1) - \t_3)}
                H(\bU^N((t \wedge \sdd_1) - \t_3),\bV^N((t \wedge \sdd_1) - \t_3)) \bigr\}
                  \non\\
     && \times   I[\t_3 \le t] I[\t_3 < \sdd_1] I[\wX^N_1(\t_3) \ne \wX^N_2(\t_3)] \non\\
     &=& e^{\trh \t_3} \E_{\wX^N_1(\t_3), \wX^N_2(\t_3)}
         \bigl\{ e^{\trh((t - \t_3) \wedge (\sdd_1 - \t_3))} \non\\
      && \quad\times          H(\bU^N((t - \t_3) \wedge (\sdd_1 - \t_3)),\bV^N((t - \t_3) \wedge (\sdd_1 - \t_3)) \bigr\}
                  \non\\
     && \qquad\times   I[\t_3 \le t] I[\t_3 < \sdd_1] I[\wX^N_1(\t_3) \ne \wX^N_2(\t_3)] \non\\
    &\le&  e^{\trh \t_3} \Psi(T)   I[\t_3 < \sdd_1] I[\wX^N_1(\t_3) \ne \wX^N_2(\t_3)],
        \label{ADBi-5}
\ena
where the last line uses the definition~\Ref{ADB-key-distance-bnd} of~$\Psi(T)$.
Moreover,  on the event $\{\t_1 < (\t_3 \wedge \sdd_1) \le t\} \cap \{\t_3 < \sdd_1\}$,
if $\wX^N_1(\t_3) = \wX^N_2(\t_3)$, it follows that
\eq\label{ADBi-5.5}
   H(\wX^N_1(t \wedge \sdd_1), \wX^N_2(t \wedge \sdd_1)) 
           \Eq 0.
\en
Then, when $\sdd_1 \le \t_3$, we have
\eqa
  \lefteqn{\E_{\bX,\bX'}
     \bigl\{ e^{\trh (t \wedge \sdd_1)}  H(\wX^N_1(t \wedge \sdd_1), \wX^N_2(t \wedge \sdd_1))
         \giv \cF_{(\t_3 \wedge \sdd_1)} \bigr\} }\non\\
    && \qquad\times      I[\t_1 < (\t_3 \wedge \sdd_1) \le t] I[\sdd_1 \le \t_3]  \non\\
    && \Le e^{\trh \sdd_1} H(\wX^N_1(\sdd_1), \wX^N_2(\sdd_1))
        I[\t_1 < \sdd_1 \le t] I[\sdd_1 \le \t_3]  \non\\
    && \Le  \adbm{\tnkt} e^{\trh \sdd_1} I[\t_1 < \sdd_1 \le t]  I[\sdd_1 \le \t_3].
        \label{ADBi-6}
\ena
}

\adbi{We now take expectations.  \adbm{Observe that
$\wX^N_1(\t_3 \wedge \sdd_1) = \wX^N_2(\t_3 \wedge \sdd_1)$ with conditional probability at least $\hktt$,
given $\{\t_1 < \sdd_1\}$, in view of~\Ref{ADBi-0.5}. 
Hence, for the term in~\Ref{ADBi-5},  for $(\bX,\bX') \in \cX^N_+$, we have}
\eqs
  \lefteqn{ \E_{\bX,\bX'}\{ e^{\trh \t_3} I[\t_3 < \sdd_1] I[\wX^N_1(\t_3) \ne \wX^N_2(\t_3)]\} }\\
   &&\Le \E_{\bX,\bX'}\{ e^{\trh (\t_1 + \adbm{1/(\r N)})} I[\t_1 < \sdd_1] I[\wX^N_1(\t_3) \ne \wX^N_2(\t_3)]\} \\
   &&\Le e^{\trh/(\r N)}(1 - \hktt) \E_{\bX,\bX'}\{ e^{\trh \t_1} I[\t_1 < \sdd_1]\} \\
   &&\Le (1 - \hktt) \{e^{1/N}(2\nu + 1)\}^{\trh/\r},
\ens
using \Ref{ADB-coupling-1}, \Ref{ADBi-0.5}  and H\"older's inequality.
Then, for the term~\Ref{ADBi-6}, we have
\eqs
   \lefteqn{ \E_{\bX,\bX'}\{e^{\trh \sdd_1} I[\t_1 < \sdd_1 \le t]  I[\sdd_1 \le \t_3]\}} \\
   &&\Le  \E_{\bX,\bX'}\{e^{\trh (\t_1 + \adbm{1/(\r N)})} I[\t_1 < \sdd_1]\} 
   \Le  \{e^{1/N}(2\nu + 1)\}^{\trh/\r},
\ens
as above.  Hence, combining the contributions from \Ref{ADBi-5} and~\Ref{ADBi-6}, it follows that,
for $(\bX,\bX') \in \cX^N_+$ and for $0 \le t \le T$,
we have
\eqa
   \lefteqn{\E_{\bX,\bX'}\{ e^{\trh t}H(\wX^N_1(t), \wX^N_2(t)) I[\t_1 < (\t_3 \wedge \sdd_1) \le t]\} }\non\\
        &&\Le \{e^{1/N}(2\nu + 1)\}^{\trh/\r}\{\Psi(T) (1 - \hktt) + \adbm{\tnkt}\},
           \label{ADB-small-tau}
\ena
completing the treatment of the third term.}

\adbj{Finally, 
for $(\bX,\bX') \in \cX^N_+$, 
the fourth term in~\Ref{ADBi-4} is bounded in~\Ref{ADBi-2} by~$K_3\tnu$.}
\ignore{
\eq\label{ADB-tau-dash}
   2N\d_1 \E_{\bX,\bX'}\{ e^{\trh \sdd_1} I[\sdd_1  \le \t_1]\} \Le 2\d_1 (2\nu +1),
\en
since, on $\{\sdd_1  \le \t_1\}$, $H(\wX^N_1(\sdd_1), \wX^N_2(\sdd_1)) \ge K_3/2$, in view of~\Ref{ADBi-1}.
}%
\adbi{Thus, for $(\bX,\bX') \in \cX^N_+$, 
adding the bounds from \Ref{ADB-big-tau}, \Ref{ADB-medium-tau} and~\Ref{ADB-small-tau} to~$K_3\tnu$, we have
\eqa
  \lefteqn{ \E_{\bX,\bX'} \bigl\{e^{\trh (t\wedge \sdd_1)} H(\wX^N_1(t\wedge \sdd_1), \wX^N_2(t\wedge \sdd_1))\bigr\} }
             \non\\
  && \Le  \adbm{2\tnkt\bigl(1 + \{e^{1/N}(2\nu + 1)\}^{\trh/\r}\bigr)} \label{ADBi-12} \\
   &&\mbox{}\qquad  + \{e^{1/N}(2\nu + 1)\}^{\trh/\r}\Psi(T) (1 - \hktt).\non
\ena
}

\adbi{On the other hand, using similar arguments, for
$$
   (\bX,\bX')\ \in\ \cX^N_- \Def   \{(\bZ,\bZ') \in \adbi{\tcX^N(\d_1)}\colon\,H(\bZ,\bZ') \le K_3\},
$$
we write
\eqa
   \lefteqn{\E_{\bX,\bX'}\{ e^{\trh (t\wedge \sdd_1)}H(\wX^N_1(t\wedge \sdd_1), \wX^N_2(t\wedge \sdd_1))\} } \non\\
    &&\Eq \E_{\bX,\bX'}\{ e^{\trh (t\wedge \sdd_1)}H(\wX^N_1(t\wedge \sdd_1), \wX^N_2(t\wedge \sdd_1)) \non\\
     && \quad\times (I[t < \t_2 < \sdd_1)] + I[\t_2 \le t] I[\t_2 < \sdd_1] + I[\sdd_1 \le \t_2])\}.
\ena
For the first term, we immediately have
\eq \label{ADBi-8}
   \E_{\bX,\bX'}\{ e^{\trh (t\wedge \sdd_1)}H(\wX^N_1(t\wedge \sdd_1), \wX^N_2(t\wedge \sdd_1))
             I[t < \t_2 < \sdd_1)]\}  
    \Le e^{\trh/(\r N)}\nkt. 
\en
Then, using the strong Markov property much as for~\Ref{ADBi-5}, it follows that
\eqa
\lefteqn{ \E_{\bX,\bX'}\{ e^{\trh (t\wedge \sdd_1)}H(\wX^N_1(t\wedge \sdd_1), \wX^N_2(t\wedge \sdd_1))
             I[\t_2 \le t]I[\t_2 < \sdd_1)]\} } \non\\
    &&\Le e^{\trh/(\r N)}\Psi(T)(1 - \hktt), \phantom{XXXXXXXXXXXXXX} \label{ADBi-9}
\ena
for $0 \le t \le T$, and it is immediate that
\eq\label{ADBi-10}
 \E_{\bX,\bX'}\{ e^{\trh (t\wedge \sdd_1)}H(\wX^N_1(t\wedge \sdd_1), \wX^N_2(t\wedge \sdd_1))
             I[\sdd_1 \le \t_2] \} 
     \Le  e^{\trh/(\r N)} \adbm{\tnkt}. 
\en
Thus, for $(\bX,\bX')\ \in\ \cX^N_- $, 
we have
\eqa
   \lefteqn{\E_{\bX,\bX'}\{ e^{\trh (t\wedge \sdd_1)}H(\wX^N_1(t\wedge \sdd_1), \wX^N_2(t\wedge \sdd_1))\} } \non\\
   &&\Le \adbm{2\tnkt} e^{\trh/(\r N)} + e^{\trh/(\r N)} \Psi(T) (1 - \hktt).  \label{ADBi-11}
\ena
}

\adbi{Comparing~\Ref{ADBi-11} with~\Ref{ADBi-12}, it follows that, 
for $0 < \trh \le \r$  and for $0 \le t \le T$,
\eq\label{ADBi-Psi-eq}
  \Psi(T) \Le
     \{e^{1/N}(2\nu + 1)\}^{\trh/\r}\Psi(T) (1 - \hktt)  + C_{\Ref{ADBi-C-def}},
\en
where
\eq\label{ADBi-C-def}
      C_{\Ref{ADBi-C-def}} \Def \adbm{2\tnkt\bigl(1  +  \{e(2\nu+1)\}^{\trh/\r}\bigr)}.
\en
\ignore{
     &&\Le  \nu K_3 e^{\trh/N} + (1-\h)\E_{\bX,\bX'}\{ e^{\trh \t_2}\}\Psi(T)
               + 2N\d_1 e^{\trh T} p_N(\d_1,T_N) \non\\
     &&\Le  (1-\h) e^{\trh/N} \Psi(T) + 2N\d_1 e^{\trh T} p_N(\d_1,T_N).  \label{ADB-small-distance}
\ena
Combining \Ref{ADB-big-tau} -- \Ref{ADB-small-distance} shows that there exists $N_0 \ge 1$ such that,
for $0 < \trh \le \r$ and $N \ge N_0$, and for $0 \le t \le T \le T_N$,
\[
   \Psi(T) \Le K_3\{\tilde{\nu} + (2 \tilde{\nu})^{\trh/\r} e^{\trh/N}\} + (1-\h)\Psi(T)\, (2\tilde{\nu})^{\trh/\r} e^{\trh/N}
        + 4N\d_1 e^{\trh T} p_N(\d_1,T_N),
\]
\adbc{with~$\tilde{\nu}$ as in~\Ref{ADB-tilde-nu-def}.}
}
So, 
choosing~$\trh$ to satisfy
\eq\label{ADBi-tr-def}
    \{e(2\nu + 1)\}^{\trh/\r}(1-\hktt) \Eq 1-\hktt/2,
\en
it follows, for 
any $T > 0$, that
\eq\label{Psi-bnd}
   \Psi(T) \Le 2C_{\Ref{ADBi-C-def}}/\hktt. 
\en
}

\adbi{From the definition~\Ref{ADB-key-distance-bnd} of~$\Psi(T)$, and because, for $\bX' \in N(\bX)$,
we have $H(\bX,\bX') \le K_3$,
it thus follows, for all $\bX,\bX' \in \cX^N(\d_1)$ with $\bX' \in N(\bX)$, 
that
\eq\label{ADB-contract-MPP-basic}
    \E_{\bX,\bX'}\{e^{\trh(t \wedge \sdd_1)} H(\wX^N_1(t \wedge \sdd_1), \wX^N_2(t \wedge \sdd_1))\}
              \Le 2C_{\Ref{ADBi-C-def}}/\hktt, 
\en
for $\trh$ chosen to satisfy~\Ref{ADBi-tr-def}, and for $C_{\Ref{ADBi-C-def}}$ as defined in~\Ref{ADBi-C-def}.
Hence it follows that
\[
   e^{\trh t}\E_{\bX,\bX'}\{ H(\wX^N_1(t), \wX^N_2(t))I[t \le \sdd_1]\}
              \Le 2C_{\Ref{ADBi-C-def}}/\hktt,
\]
and thus that, for all $\bX,\bX' \in \cX^N(\d_1)$ with $\bX' \in N(\bX)$, 
\[
    \E_{\bX,\bX'}\{ H(\wX^N_1(t), \wX^N_2(t)) \} \Le (2C_{\Ref{ADBi-C-def}}/\hktt)e^{-\trh t} 
                  + \adbm{2N\d_1}\pr_{\bX,\bX'}[\sdd_1 < t].
\]
By \Ref{ADB-extremes-bnd} and~\Ref{ADB-alpha-def0}, for any $\bX \in \cX^N(\d_1/2)$ and $\bX' \in N(\bX)$,
\[
    \adbm{N}\pr_{\bX,\bX'}[\sdd_1 < a_2 N/(2\r)] \Le a_1 a_2 \adbm{N^2} e^{-a_2 N} \Le a_1 a_2 e^{-\r' a_2 N/(2\r)}
\]
for any $\r' \le \r$, if \adbm{$N \ge N_2$} is large enough that also $\adbm{N^2} e^{-a_2 N/2} \le 1$.  
Hence, for such~$N$, it follows that
\[
     \E_{\bX,\bX'}\{ H(\wX^N_1(t), \wX^N_2(t)) \} \Le De^{-\trh t}, \qquad 0 \le t \le \a N,
\]
with
\eq\label{ADB-alpha-def1}
  D \Eq (2C_{\Ref{ADBi-C-def}}/\hktt) + 2\d_1 a_1 a_2 \quad\mbox{and}\quad \a \Eq a_2/(2\r).
\en
This proves the lemma.}
\end{proof}

We now use Lemma~\ref{ADB-coupling-lemma} and Theorem~\ref{conc-thm}
to prove that, if $f\colon S_N \to \R$ is an $\|\cdot \|_M$-Lipschitz function,
then $f(\bX^N(t))$ is concentrated about \adbm{$\ex_\bX f(\bX^{N,\d_1}(t))$} for a long time,
\adbm{if~$\bX$ is close enough to~$N\bc$}.

\begin{theorem}\label{ADB-conc-MPP}
Under Assumptions 1--3, \adbm{with~$\d_0$ as in Theorem~\ref{thm:main1} and~$\d_1$ as specified in Remark~\ref{ADB-d_1-def}, and
with~$\a$ as in \Ref{ADB-alpha-def1},}
there exist positive constants $D,\trh,r^*$ and~$N_1$ such that,
for any $\bX\in\cX^N(\d_1/4)$, $0 < t \le \a N$ and $m\ge0$, and for any
$\|\cdot\|_M$-Lipschitz function $f\colon\cX^N(\d_0) \to \R$ with constant~$L$,
\eqs
    \lefteqn{\Pr_{\bX} \Bigl( \bigl\{\bigl| f(\bX^N(t))-\E_{\bX} [f(\bX^{\adbj{N,\d_1}}(t))] \bigr|\ge m \bigr\}\Bigr) }\\
     &&\Le 2\exp\left( - \frac {m^2}{\adbj{Nr^* L^2 D^2/\trh  + 2LD m/3}}\right)
             + \lceil  \r\a N \rceil \hz_N(\d_1/\adbk{8}),
\ens
as long as $N \ge N_1$, \adbd{with~$\hz_N(\cdot)$ as given in~\Ref{ADB-hz-bnd}.} 

In particular, there is a constant $v = v(\trh) > 0$ such that, for $N \ge N_1$ and for $\bX \in \cX^N(\d_1/4)$,
\eq\label{ADB-epidemic-var}
    \Var_{\bX} f(\bX\adbj{\undi}(t)) \Le N v L^2, \quad 0 < t \le \a N,
\en
\adbm{and hence also
\eq\label{ADB-ev2}
     \ex_\bX \|\bX\undi(t) - \ex\{\bX\undi(t)\}\|_M \Le d \sqrt{Nv}, \quad 0 < t \le \a N.
\en
}
\end{theorem}

\begin{proof}
We take $S := \adbm{S_N}$ and $S_0 := \cX^N(\d_1/2)$ in Theorem~\ref{conc-thm},
and define $q_0 := Nr^*$, where
$$
    r^*  \Def \adbm{R^*(B_M({\bf c},\d_1/2))} \Eq \sup_{\by' \in \adbm{B_M({\bf c},\d_1/2)}}\adbj{\sum_{\bJ\in\cJ}}r_{\bJ}(\by') .
$$
\adbi{Note that, on the event $\{\bX^N(s) \in \cX^N(\d_1/2),\,0\le s\le t\}$,
we have $\bX^N(s) = \bX^{N,\d_1}(s)$, $0 \le s \le t$.}
Hence, for~$f$ as in the statement of the theorem,
we can deduce from Lemma~\ref{ADB-coupling-lemma} and Theorem~\ref{conc-thm}
that, for  $D,\a$ and~$\trh$  as in Lemma~\ref{ADB-coupling-lemma}, and
for any $\bX\in\cX^N(\d_1/4)$, $0 < t \le \a N$ and $m\ge0$, we have
\eqa
  \lefteqn{\Pr_{\bX} \Bigl( \bigl\{\bigl| f(\bX^N(t))-\E_{\bX} [f(\bX^{\adbj{N,\d_1}}(t))] \bigr|\ge m \bigr\}
       \cap \{\bX^N(s) \in \cX^N(\d_1/2),\,0\le s\le t\} \Bigr)   }\non\\
    &&\Le 2\exp\left( - \frac {m^2}{\adbj{(Nr^* L^2 D^2/\trh)  + 2LD m/3}}\right),
        \phantom{XXXXXXXXXXXXXX}\label{ADB-epidemic-conc-1}
\ena
if $N \ge N_0$, \adbm{where~$N_0$ is as in Lemma~\ref{ADB-coupling-lemma}}. 

Then, there exists $N'_0$ such that, for $\bX \in \cX^N(\d_1/4)$,
$$
   \Pr_{\bX}\adbj{\bigl[\{\bX^N(s) \in \cX^N(\d_1/2),\,0\le s\le t\}^c\bigr]} \Le \lceil  \a N \rceil \hz_N(\d_1/\adbk{8})
$$
for all $0 < t \le \a N$ and $N \ge N'_0$, by Proposition~\ref{ADB-3.5}.
The first statement of the theorem now follows, with $N_1 = \max \{N_0,N'_0\}$.

The second follows from the bound by calculation, since, for any random variable~$Z$, 
$\ex Z^2 = 2\int_0^\infty z\pr[|Z| > z]\,dz$.  \adbm{Then, since, for $\tbe\uii$, $1\le i\le d$, as 
defined in the proof of Lemma~\ref{ADB-exp-bound-MG}, we have
\[
    \by \Eq \sum_{i=1}^d \lan \by, \tbe\uii \ran_M \,\tbe\uii,
\]
and since the functions $f_i(\by) := \lan \by, \tbe\uii \ran_M$ are $M$-Lipschitz with constant~$1$, it follows easily that,
for $0 < t \le N\a$,
\eqs
    \ex_\bX\| \bX\undi(t) - \ex\{\bX\undi(t)\} \|_M &\le& \sum_{i=1}^d \ex_\bX| f_i(\bX\undi(t)) - \ex\{f_i(\bX\undi(t))\} |\\
             &\le& d \sqrt{Nv},
\ens
with the last inequality following from~\Ref{ADB-epidemic-var}.
}
\end{proof}

It is shown in the next section that, \adbj{for $\bX^N(0)$ close enough to~$N\bc$, 
the distribution of~$\bX^N\adbj{(t)}$ for a long time} remains close to that of
the equilibrium distribution~$\pi\und$ of~$\adbm{\bX^{N,\d}}$, for appropriate choice of~$\d$.
We now show~$\pi\und$ to be well concentrated, establishing the second part of Theorem~\ref{thm:main}.
Recall the expression for~$\hz_N(z)$ \adbj{given in}~\Ref{ADB-hz-bnd}.

\begin{theorem}\label{ADB-equilibrium}
 For any $0 < z < \d \le \adbm{\d_1}$, 
\[
   q_N(z,\d) \Def \pi\und\{B^c_M(N\bc,Nz)\} \Le   \adbk{3\hz_{N}\bigl(z(1 - e^{-1/2})/2\bigr)}.
\]
\end{theorem}

\begin{proof}
 Since~$\pi_{N,\d}$ is the equilibrium distribution of $\bX^{N,\d}$, it follows that
 \eqs
     q_N(z,\d) &:=& 
     \int_{B_M(N\bc,N\d)}
                 \pr_{\bY'}[\bX^{N,\d}(\r^{-1}) \notin B_M(N\bc,Nz)]\,\adbj{\pi^{N,\d}}(d\bY')\\
          &\le& \pi_{N,\d}\{B^c_M(N\bc,N\adbm{(ze^{1/2}\wedge\d_1)})\}\\
          &&\mbox{}\  + \int_{B_M(N\bc,N\adbm{(ze^{1/2}\wedge\d_1)})}
                   \pr_{\bY'}[\bX^{N,\d}(\r^{-1}) \notin B_M(N\bc,Nz)]\,\pi\adbj{^{N,\d}}(d\bY').
 \ens
Now, 
for $\bY' \in B_M(N\bc,N\adbm{(ze^{1/2}\wedge\d_1)})$, it follows from Theorem~\ref{thm:main1} (i) 
that the solution~$\by$ to the ODE system~\Ref{eq:diff-eq}
starting in $\by' := N^{-1}\bY'$ satisfies \hbox{$\|\by(\r^{-1}) - \bc\|_M \le ze^{-1/2}$.} Hence, 
taking $\e''(z) := z(1 - e^{-1/2})/\adbk{2}$, so that
$\|\by(\r^{-1}) + \adbj{\bw} - \bc\|_M < z$ for any~$\adbj{\bw}$ with $\|\adbj{\bw}\adbk{\|_M} \le \e''(z)$, it follows from
Lemma~\ref{ADB-exp-bound-XN} that
\[
   \pr_{\bY'}[\bX^{N,\d}(\r^{-1}) \notin B_M(N\bc,Nz)] \Le \hz_N(\e''(z)),\quad \bY' \in B_M(N\bc,N\adbm{(ze^{1/2}\wedge\d_1)}).
\]
Thus $q_N(z,\d) \le q_N(\adbm{(ze^{1/2}\wedge\d_1)},\d) + \hz_N(\e''(z))$;
\adbm{note also that $q_N(\d_1,\d) = 0$, by the definition of~$q_N$, because $\d \le \d_1$}.  
Iterating the inequality then gives
\[
     q_N(z,\d) \Le \sum_{r \ge 0} \hz_N(\e''(ze^{r/2})), \quad 0 < z < \d .
\]

Now, \adbm{from~\Ref{ADB-hz-bnd}, $\hz_N(\e''(x)) = 2d \exp\{-c_a N x\min\{c_b x,1\}\}$}, for suitable constants $c_a$ and~$c_b$.
\adbm{Since, for $a \ge 1$ and $x > 0$, we have $(ax\wedge 1) \ge (x\wedge 1)$, it follows that
\[
    q_N(z,\d) \Le 2d\sum_{r \ge 0} \bigl(\exp\{-c_a N \e''(z)\min\{c_b \e''(z),1\} \bigr)^{e^{r/2}}, \quad 0 < z < \d .
\]
Then, for $0 < u \le 1/3$, by comparison with a geometric sum,
\[
     \sum_{j \ge 0} u^{e^{j/2}} \Le u + \frac u{1-u} \ <\ 3u.
\]
Hence, if $\hz_{N}(\e''(z)) \le 2d/3$,
it follows that $q_N(z,\d) \le 3 \hz_N(\e''(z))$, 
and this inequality is trivially true otherwise.  In view of the definition of $\e''(z)$,}
this establishes the theorem.
\end{proof}


\adbm{It follows from Theorem~\ref{ADB-equilibrium} and~\Ref{ADB-hz-bnd} that}
\[
    q_N(m/N,\d) \adbj{\Le} \adbm{6}d \exp\{-k_1 m \min(1,k_2m/N)\},
\]
for suitable constants $k_1$ and~$k_2$, and the inequality $\min(1,x) \ge x/(1+x)$ in $x \ge 0$ then yields
a bound in the form given in~\Ref{ADB-conc-1.2}.
Note \adbe{also} that Theorem~\ref{ADB-equilibrium} implies that there exists a constant $v_{\infty}$ such that,
\adbe{ if $f\colon S_N \to \R$ is a $\|\cdot \|_M$-Lipschitz function with constant~$L$, then}
\begin{equation}\label{equilibrium-var}
\Var_{\pi\und} f(\grbu{\bX}^N(\infty)) \Le N v_{\infty} L^2,
\end{equation}
uniformly for all $0 < \d \le \adbm{\d_1}$.
Alternatively, \adbe{the bound}~\eqref{equilibrium-var} can be deduced from Barbour, Luczak, Xia (2018a, Lemma 5.1).

\adbe{
In the proof of Theorem~\ref{lower-bound}, we need the following result, comparing the distributions of two
realizations of~$\bX^N$ that start close to one another, after evolving for  a fixed length of time.
It can be deduced by combining
 \BLX~(2018a, Theorem~3.1 and Remark~6.4) and Assumption~2; for the sake of completeness, we give
a separate proof here.  For $\bX_1,\bX_2 \in S_N$, we define $n_0(\bX_1,\bX_2)$ to be the smallest $n_0 \in \Z_+$
such that 
\adbf{$\bX_2 = \bX_1 + \sum_{i=1}^{n_0} \bJ'_i$, where $\bJ'_i \in \{\bJ_i,-\bJ_i\}$ for some $\bJ_i \in \cJ$,}
$1 \le i \le n_0$; 
it is shown in Lemma~\ref{intvec} that, for a suitable constant~$\mu$,
\eq\label{ADB-n0-bnd}
    n_0(\bX_1,\bX_2) \Le \mu \|\bX_1-\bX_2\|_M.
\en
We define
\eq\label{ADB-rate-shift}
   r_1 \Eq r_1(\d_1) \Def \adbm{\max_{\bJ'\in\cJ}|\bJ'|} \adbj{\sup_{\by \in B_M(\bc,\d_1)}\sJJ} |\nabla r_{\bJ}({\by})|,
\en
noting that, if $\bJ_0 \in \cJ$ and both $\by$ and $\by + N^{-1}\bJ_0$ belong to~$B_M(\bc,\d_1)$, then
\eq\label{ADB-rate-shift-2}
   \adbf{\sJJ} N|r_\bJ(\by + N^{-1}\bJ_0) - r_\bJ(\by)| \Le r_1.
\en
}

\adbe{
\begin{proposition}\label{ADB-close-TV}
Let $\bX^N_1$ and $\bX^N_2$ be two copies of~$\bX^N$, 
with $\bX^N_l(0) \in B_M(N\bc,\adbk{N\d_1}/2)$, $l=1,2$. 
Then there exist $k_*,u_* > 0$ such that
\eq\label{ADB-full-term-1}
     \dtv\bigl(\law(\bX^N_1(u_*)),\law(\bX^N_2(u_*))\bigr)
       \Le k_* N^{-1/2}\|\bX^N_1(0) - \bX^N_2(0)\|_{\adbj{M}}.
\en
\end{proposition}
}

\begin{proof}
 We  begin by establishing~\Ref{ADB-full-term-1} for copies of a modified process~$\hbX^N$ that has
 the same set of possible jumps as~$\bX^N$, but has transition rates~$N\hr_{\bJ}(\cdot)$
that only differ from those of~$\bX^N$ far from~$N\bc$.  
We define
 \[
    \hr_{\bJ}(\bc + \by) \Def \begin{cases}
                         r_{\bJ}(\bc + \by) &\ \mbox{if } \|\by\adbk{\|_M \le \d_1};\\
                         r_{\bJ}(\bc + \adbm{\d_1}\by/|\by|) &\ \mbox{if } \|\by\adbk{\|_M > \d_1}.
                        \end{cases}
\]
With this definition, 
\eq\label{ADB-rate-shift-3}
    \sJJ N|\adbj{\hr_\bJ}(\by + N^{-1}\bJ_0) - \adbj{\hr_\bJ}(\by)| \Le r_1 \quad
            \mbox{and}\quad \min_{\bJ \in \cJ} \adbj{\hr_\bJ}(\by) \ \ge\ r_0,
\en
now for {\it all\/} $\by \in \hS$.

\adbe{
Take two copies $\hbX_1^N$ and~$\hbX_2^N$ with $\hbX^N_1(0) = \hbX^N_2(0) + \bJ_0$ for some $\bJ_0\in \cJ$,
and write $\hbx^N_l(t) := N^{-1}\hbX^N_l(t)$, $l=1,2$.
\adbf{Couple} $\hbX_1^N$ and~$\hbX_2^N$ over an interval $[0,t_0]$,  by matching jumps as far as
possible for all jump vectors
$\bJ \in \cJ \setminus \{\bJ_0\}$, with a more careful treatment of jumps of~$\bJ_0$.
Think of the two copies as each making ``green'' jumps by each~$\bJ$ at rate
$N\bigl(r_\bJ(\hbx^N_1(t))\wedge r_\bJ(\hbx^N_2(t))\bigr)$, and ``red'' jumps by~$\bJ$
at rate $N\bigl((r_\bJ(\hbx^N_1(t))-r_\bJ(\hbx^N_2(t))) \vee 0\bigr)$ for~$\hbX^N_1$ and
at rate $N\bigl((r_\bJ(\hbx^N_2(t))-r_\bJ(\hbx^N_1(t))) \vee 0\bigr)$ for $\hbX^N_2$;
the {\it green\/} jumps are coupled so as to occur  together at the given rate.
For the particular jump~$\bJ_0$, we break the rates up further by regarding the copies as taking
``blue'' jumps by~$\bJ_0$ at rate~$Nr_0$, and ``green'' jumps by~$\bJ_0$ only at rate
$N\bigl\{\min (r_{\bJ_0}(\hbX^N_1(t)),r_{\bJ_0}(\hbx^N_2(t))) - r_0\bigr\}$, with red jumps by~$\bJ_0$ as above.
}

\adbe{
We start by generating the blue jumps by~$\bJ_0$, so that, as far as possible, the copy~$\hbX^N_2$ makes exactly one more blue jump
than~$\hbX^N_1$ in the interval $[0,t]$.  If this is the case, and if there are no red jumps during this interval, then
$\hbX^N_1(t) = \hbX^N_2(t)$.  Note that, in either process, the blue
jumps  form a Poisson process with constant rate $Nr_0$, irrespective of the states of the processes,
so that the number in the interval $[0,t]$ has a Poisson distribution with mean~$Nr_0t$.
Now the Poisson distribution~$\Po(\la)$ with mean~$\la$ is unimodal, so that, if $Z \sim \Po(\la)$, then
\eq\label{ADB-Poisson-diff}
     \dtv(\law(Z),\law(Z+1)) \Eq \max_{j \in \Z_+} \pr[Z = j] \Le (2e\la)^{-1/2};
\en
see, for example, \adbh{Barbour \& Jensen~(1989), Remark on p.78}.
Accordingly, taking $\la = Nr_0t$, we generate the blue jumps in the two processes
by first choosing their numbers $(Z_1(t),Z_2(t))$ \adbf{so that $Z_2(t) = Z_1(t)+1$},
except on an event whose probability is bounded using~\Ref{ADB-Poisson-diff}, and
then choosing the times of the blue jumps within the time interval $[0,t]$.  On the event
$\{(Z_1(t),Z_2(t)) = (\ell, \ell+1)\}$, we take $\ell+1$ independent
uniform random variables $(B, B_1, \dots, B_\ell)$ over $[0,t]$, place the blue jumps of $\hbX^N_1$ at $B_1, \dots, B_\ell$,
and the blue jumps of $\hbX^N_2$
at $B, B_1, \dots, B_\ell$.  The complement of the event $\cZ_t := \{Z_2(t) = Z_1(t) + 1\}$ has probability
at most $(2eNr_0 t)^{-1/2}$,
and we deem the coupling to have failed if it occurs: we then complete the coupling in any way that realizes
the corresponding marginal
conditional distributions of $\hbX^N_1$ and~$\hbX^N_2$, and do not consider the case any further.
}

\adbe{
Next, we condition on the times of the blue jumps, and generate the green jumps of the two processes --- which always coincide ---
as well as the red jumps of the two processes separately, over the interval $[0,t]$, \adbm{by means of the Doob--Gillespie
algorithm. The green and red rates specified
above are used, together with independent sequences of independent uniform $U[0,1]$ and standard exponential random variables, to
determine the sequence of jumps and their times, respectively.} 
Let~$R(t)$ denote the total number of red jumps in the interval $[0,t]$.
As observed above, on the event $\cZ_t \cap \{R(t)=0\}$, it follows that $\hbX^N_1(t) = \hbX^N_2(t)$.
Now, if $R(s) = \ell$, the rate associated with the occurrence of red jumps at time~$s$ is at most \adbj{$(\ell + 1)r_1$}, in view
of~\Ref{ADB-rate-shift-3}, since then
\eq\label{ADB-R(s)-bnd}
   n_0(\hbX^N_1(s),\hbX^N_2(s)) \Le  R(s) + I[B > s] \Le \ell + 1.
\en
\adbj{In particular,} $\pr[R(t) \ge 1] \le 1 - e^{-r_1 t} \le r_1 t$, implying that
\eqa
    &&\dtv\bigl(\law(\hbX^N_1(t)),\law(\hbX^N_2(t))\bigr) \Le \pr[\hbX^N_1(t) \ne \hbX^N_2(t)]
             \phantom{XXXXXXXXX}\non \\
      && \qquad \Le  \pr\bigl[(\cZ_t \cap \{R(t)=0\})^c\bigr] \Le (2eNr_0 t)^{-1/2} + r_1 t. \label{ADB-bnd-1}
\ena
The bound~\Ref{ADB-bnd-1} decreases as~$t$ increases until \adbm{$r_0 t$ and~$r_1 t$ are} of magnitude comparable to~$N^{-1/3}$, 
in which case it is
of order~$O(N^{-1/3})$.  To obtain smaller bounds for larger values of~$t$,  the consequences of having $R(s) \ge 1$ need
to be quantified.
}

\adbe{
Now define
\[
    \D^N(t) \Def \sup_{\bX \in S_N} \mJJ \dtv\bigl(\law(\hbX^N(t) \giv \hbX^N(0) = \bX),
            \law(\hbX^N(t) \giv \hbX^N(0) = \bX + \bJ)\bigr).
\]
Then, for $s \in (0,t)$, much as for~\Ref{ADB-bnd-1}, we have
\eqa
  \lefteqn{\dtv\bigl(\law(\hbX^N_1(t)),\law(\hbX^N_2(t))\bigr)}   \non\\
     &\le& \pr[\cZ_s^c]   + \ex\Bigl\{I[\cZ_s]\dtv\bigl(\law(\hbX^N(t-s) \giv \hbX^N(0) = \hbX^N_1(s)),
              \label{ADB-recursion-1}\\
      &&\hskip1.6in      \law(\hbX^N(t-s) \giv \hbX^N(0) = \hbX^N_2(s))\bigr)\Bigr\}. \phantom{XX}\non
\ena
Note that $n_0(\hbX^N_1(s),\hbX^N_2(s)) \le R(s)$  on the event~$\cZ_s$, so that then, \adbm{for some $0 \le K(s) \le R(s)$,}
\eq\label{ADB-difference-at-s}
     \hbX^N_1(s) \Eq \hbX^N_2(s) + \adbm{\bT_{\cJ}(K(s))}, \quad\mbox{where}\quad \adbm{\bT_{\cJ}(k) \Def \sum_{j=1}^{k} \bJ'_j},
\en
and where either $\{\bJ'_j,-\bJ'_j\} \cap \cJ \ne \emptyset$: \adbm{since not every red jump need 
increase $n_0(\hbX^N_1(s),\hbX^N_2(s))$, it may be that $K(s) < R(s)$.}
Thus, using~\Ref{ADB-difference-at-s}, and by the triangle inequality for total variation distance,
we have
\eqa
  \lefteqn{\dtv\bigl(\law(\hbX^N_1(t) \giv \hbX^N_1(s)),\law(\hbX^N_2(t)  \giv \hbX^N_2(s))\bigr) I[\cZ_s]} \non\\
    &&\Le \adbm{\sum_{j=1}^{K(s)}} \dtv\bigl(\law(\hbX^N(t-s) \giv \hbX^N(0) = \hbX^N_2(s) + \bT_{\cJ}(j)), \non\\
     &&\qquad\qquad\qquad\qquad   \law(\hbX^N(t-s) \giv \hbX^N(0) = \hbX^N_2(s) + \bT_{\cJ}(j-1))\bigr) \non\\
    &&\Le R(s) \D^N(t-s).
\ena
Now, again using the maximal rate of occurrence of red jumps, we have 
\[
     \frac d{du}\, \ex\{R(u)\} \Le r_1\adbj{\ex\{R(u) + 1\}},\qquad u > 0,
\]
and hence $\ex\{R(s)\} \le e^{r_1 s} - 1 \le 2r_1 s$ in $s \le 1/r_1$.  In consequence, from~\Ref{ADB-recursion-1},
for any $0 < s < t \le 1/r_1$,
\eq\label{ADB-recursion-2}
    \D^N(t) \Le (2eNr_0 s)^{-1/2} + 2r_1 s \D^N(t-s).
\en
}

\adbe{
It remains to exploit the recursion~\Ref{ADB-recursion-2}, together with the initial bound
$\D^N(s) \le (2eNr_0 s)^{-1/2} + r_1 s$, from~\Ref{ADB-bnd-1}.  One way is as follows.  Fix
$t := t_0 := 1/(8r_1)$, so that, for $0 < s < t$, we have $2r_1 s < 1/4$, and consider times $t_k := 4^{-k}t$
for $0 \le k \le k_N$, where $4^{-k_N}r_1 t \ge N^{-1/3} > 4^{-k_N-1}r_1 t$.  Then~\Ref{ADB-recursion-2}
with $t = t_k$ and $s = t_k - t_{k+1} = 3t_{k+1}$ implies that
\[
    \D^N(t_k) \Le \frac1{\sqrt{6eNr_0 t_{k+1}}} + \frac{\D^N(t_{k+1})}4\,, \quad 0 \le k \le k_N-1.
\]
Iterating the recursion, for $t = 1/(8r_1)$, gives
\eqs
    \D^N(t) &\le& \sum_{k=0}^{k_N-1} \frac{4^{-k}}{\sqrt{6eNr_0 t_{k+1}}} + 4^{-k_N}\D^N(4^{-k_N}t) \\
    &\le& \sum_{k=0}^{k_N-1} \frac{2^{-k+1}}{\sqrt{6eNr_0 t}}
          + 4^{-k_N}\Bigl\{\frac{2^{k_N}}{\sqrt{2eNr_0t}} + 4^{-k_N}r_1 t\Bigr\}\\
    &\le& \frac{4}{\sqrt{6eNr_0 t}} + \frac{4\sqrt 2 N^{-1/6}}{\sqrt{2eNr_0t}} + 128 N^{-2/3} \Le k'_* N^{-1/2},
\ens
say, where $k'_* = k'_*(r_0/r_1)$.  This, together with~\Ref{ADB-n0-bnd} amd the triangle inequality for
total variation distance,
gives the statement of the theorem for the process~$\hbX^N$, with $k_* = k'_* \mu$ and $u_* = t = 1/(8r_1)$.
}

\adbe{
Now, by Proposition~\ref{ADB-3.5}, the process~$\bX^N$ starting from $\bX^N(0) \in B_M(N\bc,N\adbk{\d_1}/2)$ and running for
time~$u_*$, as defined above, leaves $B_M(N\bc,N\adbk{\d_1})$ with probability at most
\[
    q'_N \Def \lceil \r u_* \rceil \hz_N(\adbk{\d_1/4}).
\]
Processes $\bX^N$ and~$\hbX^N$ with the same starting point
$\bX \in B_M(N\bc,N\adbk{\d_1}/2)$ can be coupled so as to have identical paths until leaving $B_M(N\bc,N\adbk{\d_1})$, because their
transition rates are identical there, so that, if $\bX^N(0) = \hbX^N(0) = \bX \in B_M(N\bc,N\adbk{\d_1}/2)$,
\[
    \dtv\bigl(\law(\bX^N(u_*)),\law(\hbX^N(u_*))\bigr) \Le q'_N \Le k'' N^{-1/2},
\]
say, using a very generous bound for~$q'_N$.
This establishes the proposition, with $k_* = k'_* \mu + 2k''/c_0(M)$ \adbk{for $c_0(M)$ defined in~\Ref{ADB-CM-def}}, since
$\|\bX_1 - \bX_2\adbk{\|_M} \ge c_0(M)\bone\{\bX_1 \ne \bX_2\}$ for $\bX_1,\bX_2 \in S_N$}.
\end{proof}

\ignore{
It follows that, if $\bX^N_1$ and $\bX^N_2$ are copies of the process started in states
$\bX^N_1(0),\bX^N_2(0) \in B_M(N\bc,N\d_1/2)$ with $\|\bX^N_1(0) - \bX^N_2(0)\|_M \le N^{1/2}\eps/ k_*$,
then $d_{TV} (\law(\bX^N_1(u_*)),\law(\bX^N_2(u_*))) \le \eps$.
}

\section{Cutoff}\label{cutoff}
\setcounter{equation}{0}
Throughout the section, we assume that~\adbm{$\bX^N$} satisfies Assumptions 1--3.
We now use the results of Sections \ref{Density-dependent} and~\ref{concentration}  to show
that the distributions of~$\bX^N$ exhibit cutoff, thus proving the first part of Theorem~\ref{thm:main}.

From Theorem~\ref{ADB-conc-MPP}, for $\d_1 \le \d_0$ as in Remark~\ref{ADB-d_1-def},  
$\bX^N$ stays concentrated around its mean, if started in~$\cX^N(\d_1/4)$.  Our first step is to
show that the mean of~$N^{-1}\grbu{\bX}^N$ behaves like the solution of the deterministic equation~\Ref{eq:diff-eq}.

\begin{lemma}\label{mean-drift}
Let $\d_1$ be as in Remark~\ref{ADB-d_1-def}, and~$\a$ as specified in~\Ref{ADB-alpha-def1}.
Let $\by^N$ denote the solution of~\Ref{eq:diff-eq} with initial value~$\by^N_0 \in B_M(\bc,\d_1/4)$, and let
$\hby^N(t) := N^{-1}\ex_{N\by^N_0} \adbj{\bX\undi}(t)$, for $t\ge 0$.  Then there is a constant $C_{\ref{mean-drift}} > 0$,
and a positive integer $N_3$ such that, for all $0 < t \le \a N$ and $N \ge N_3$,
\[
    \|\hby^N(t) - \by^N(t)\|_M \Le C_{\ref{mean-drift}} N^{-1/2}.
\]
\end{lemma}

\begin{proof}
Since~\adbj{$\bX\undi$} is a Markov process on the finite state \adbj{space~$\cX^N(\d_1)$}, the flow~$\hby^N$ satisfies the
differential equation
\eqs
    \frac {d\hby^N}{dt}(t) &=& \ex_{N\by^N_0}\{F\uN(N^{-1}\adbj{\bX\undi}(t))\} \\
           &=&  F(\hby^N(t)) + \ex_{N\by^N_0}\{F\uN(N^{-1}\adbj{\bX\undi}(t)) - F(\hby^N(t))\},
\ens
where $F\uN$ denotes the drift function~$F$, as in~\Ref{eq:diff-eq}, except that all transitions
out of~$\cX^N(\adbj{\d_1})$ are suppressed.  Thus, easily,
\eqs
    \|F\uN(\bz) - F(\bz)\adbk{\|_M}
        &\le& \bone\{\|\bz - \bc\|_M > \d_1/2 - N^{-1}\Jmax\}  \sup_{\bz'\in B_M(\bc,\d_1)}\sJJ \|\bJ\adbk{\|_M} r_\bJ(\bz') \\
        &=:& C_1 \bone\{\|\bz - \bc\|_M > \d_1/2 - N^{-1}\Jmax\},
\ens
and hence
\eqs
  \lefteqn{\|\ex_{N\by^N_0}\{F\uN(N^{-1}\bX\und(t)) - F(\hby^N(t))\}\adbk{\|_M}} \\
    &\le& C_1 \pr_{N\by^N_0}[\|N^{-1}\bX\und(t) - \bc\|_M > \d_1/2 - N^{-1}\Jmax] \\
    &&\qquad\mbox{} +  \sup_{\bz'\in B_M(\bc,\d_1)}\|DF(\bz')\adbk{\|_M}\,\ex_{N\by^N_0}\|N^{-1}\bX\und(t) - \hby^N(t)\adbk{\|_M} .
\ens
Now, by Proposition~\ref{ADB-3.5} \adbm{with $\d' = \d_1/4$ and $\d = 7\d_1/16$},  if $N \ge 16\Jmax/\d_1$,
we have
$$
    \pr_{N\by^N_0}[\|n^{-1}\adbj{\bX\undi}(t) - \bc\|_M > \d_1/2 - N^{-1}\Jmax]
               \Le \adbd{\lceil \r\a N \rceil}\hz_N(3\d_1/\adbk{32}),
$$
for $0 \le t \le \a N$.
Then, using Theorem~\ref{ADB-conc-MPP},
\[
     \ex_{N\by^N_0}\|N^{-1}\adbj{\bX\undi}(t) - \hby^N(t)\adbk{\|_M} \Le C_2  N^{-1/2},\qquad  0 < t \le \a N,
\]
for a suitable constant~$C_2$. Hence we have
\eq\label{mean-de}
   \frac {d\hby^N}{dt}(t) \Eq F(\hby^N(t)) + \h^N(t),\qquad 0 < t \le \a N,
\en
where $\|\h^N(t)\|_M \le C_3 N^{-1/2}$ for all~$N$ large enough.

We now recall the arguments leading to \Ref{ADB-square-deriv} and~\Ref{ADB-mod-deriv}.
Writing $\bw^N := \hby^N - \by^N$, these can be used to
deduce that,
\[
     \frac d{dt}\|\bw^N(t)\|_M \Le -\r\|\bw^N(t)\|_M + C_3 N^{-1/2},
\]
if $\by^N_0 \in B_M({\bf c},\d_1/4)$ and $0 < t \le \a N$.  Integrating this differential inequality
gives the lemma, with $C_{\ref{mean-drift}} = C_3/\r$.
\end{proof}

Let $0 < \d < \d_0$.
Take any $\bX^N(0) \in \cX^N(\d)$ such that $\|N^{-1}\bX^N(0) - \bc\|_M > N^{-1/2}$.
For~$\by^N$ the solution of~\Ref{eq:diff-eq} with $\by^N(0) = \by^N_0 =: N^{-1}\bX^N(0)$,
recall from~\eqref{eq.cutoff-time} the definition of
\eq\label{ADB-tN-def}
    t_N(\by^N_0) \Def \inf\{t > 0\colon \|\adbb{\by^N}(t) - \bc\|_M = N^{-1/2}\}.
\en
We now prove that the distribution of~$\bX^N(t_N(\by^N_0))$ is well separated \adbm{from~$\pi\undi$},
at times $t_N(\by^N_0) + s$, for~$s$ sufficiently large negative.

\begin{theorem}\label{lower-bound}
Under the assumptions of this section,
suppose that $\bX^N(0) \in \cX^N(\d_1/4)$, write $\by^N_0 := N^{-1}\bX^N(0)$,
and define $t_N(\by^N_0)$ as in~\Ref{ADB-tN-def}.
Then there exist \adbe{positive constants $C_{\ref{lower-bound}}$ and $k_{\ref{lower-bound}}$, and a positive integer
$N_{\ref{lower-bound}}$} such that
\[
     \dtv\bigl(\law(\bX^N(t_N(\by^N_0)+s)),\adbj{\pi\undi}\bigr)
           \ \ge\ 1 - C_{\ref{lower-bound}}\adbe{\exp\bigl\{-k_{\ref{lower-bound}} e^{2\r |s|}\bigr\}},
\]
for all $-t_N(\by^N_0) \le s < 0$ and \adbe{$N \ge N_{\ref{lower-bound}}$}.
\end{theorem}

\begin{proof}
For $s < 0$, we compare the distribution of~$\bX^N(t_N(\by^N_0)+s)$ with that of~$\adbj{\Xinf \sim \pi\undi}$.
\adbe{From Theorem~\ref{ADB-equilibrium}, we have
\eq\label{eqm-not-far-away}
    \P\bigl[\|\adbj{\Xinf} - N\bc\|_M > \quarter N^{1/2}e^{\r |s|}\bigr] \Le 
      \adbm{3}\hz_N(z_N(s)),
\en
where~$\hz_N$ is as defined in~\Ref{ADB-hz-def} and
\eq\label{ADB-zNs-def}
     z_N(s) \Def \frac{e^{\r|s|}(1 - e^{-1/2})}{\adbk{8\sqrt N}}\,.
\en
Note that,  from~\Ref{ADB-hz-bnd},
\eq\label{ADB-equilib-bnd}
    \hz_N(z) \Le 2d\exp\{-Nk_3 z^2\},
\en
for a suitable choice of constant~$k_3$, if $0 < z \le 1$.
Hence, from \Ref{eqm-not-far-away}, \Ref{ADB-zNs-def} and~\Ref{ADB-equilib-bnd}, it follows that
\eq\label{eqm-not-far-away-1}
    \P\bigl[\|\adbj{\Xinf} - N\bc\|_M > \quarter N^{1/2}e^{\r |s|}\bigr] \Le 
      k_4\exp\bigl\{- k_5 e^{2\r|s|}\bigr\},
\en
for all $N \ge n_1$  and all $-t_N(\by^N_0) < s < 0$, for a suitable choice of constants $k_4,k_5$ and~$n_1$.}

Then, by Theorem~\ref{thm:main1}(i), with~$\by^N$ the solution of~\Ref{eq:diff-eq} with initial
value~$\by^N_0$, we have
\[
    \|\adbb{\by^N}(t_N(\by^N_0)+s) - \bc\|_M \ \ge\ N^{-1/2}e^{\r |s|}, \quad -t_N(\by^N_0) \le s < 0.
\]
\adbj{Now, again by Theorem~\ref{thm:main1}(i), and from the definition of~$t_N(\by^N_0)$, it follows that
\eq\label{ADB-rho-tN-bnd}
   \r t_N(\by^N_0) \Le \frac1{2}\log N + \log(\d_1/4),  
\en
so that, with~$\a$ as in~\Ref{ADB-alpha-def1}, we have $t_N(\by^N_0) \le \a N$ for all~$N$ large enough.}
Hence we can apply Lemma~\ref{mean-drift} to show that, for $s < 0$,
\[
   \|\adbb{\hby^N}(t_N(\by^N_0)+s) - \bc\|_M \ \ge\ N^{-1/2}e^{\r |s|} - C_{\ref{mean-drift}}N^{-1/2}.
\]

\adbe{Finally, by Theorem~\ref{ADB-conc-MPP} with $f(X) := \|X - N\bc\|_M$ and $m := \half N^{1/2}e^{\r|s|}$,
if~$s$ is such that $e^{\r|s|} \ge 4C_{\ref{mean-drift}}$ and $- t_N(\by^N_0) \le s < 0$, we have
\eqa
   \lefteqn{\P\bigl[\|\bX^N(t_N(\by^N_0)+s) - N\bc\|_M \le \quarter N^{1/2}e^{\r |s|}\bigr] }\non\\
   && \Le 
       2\exp\bigl\{-k_6 e^{2\r|s|}\bigr\} + k_7N e^{-k_8N} ,
                \label{ADB-transient-bnd-1}
\ena
for all~$N\ge n_2$, for suitable constants $k_6$--$k_{8}$ and~$n_2$.
\adbj{Since~\Ref{ADB-rho-tN-bnd} implies that $e^{2\r|s|} \le N(\d_1/4)^2$ for $- t_N(\by^N_0) \le s < 0$,
we can simplify~\Ref{ADB-transient-bnd-1} to
\eq\label{ADB-transient-bnd}
   \P\bigl[\|\bX^N(t_N(\by^N_0)+s) - N\bc\|_M \le \quarter N^{1/2}e^{\r |s|}\bigr] \Le k_9 \exp\bigl\{-k_{10} e^{2\r|s|}\bigr\},
\en
for all $N \ge n_2$ and for suitable constants $k_9,k_{10}$.}
Thus, for $N \ge \max\{n_1,n_2\}$ and any~$s$ such that $-t_N(\by^N_0) \le s < 0$,
it follows from \Ref{eqm-not-far-away-1} and~\Ref{ADB-transient-bnd} that
\[
     \dtv\bigl(\law(\bX^N(t_N(\by^N_0)+s)),\pi\undt\bigr)
              \ \ge\ 1 - k_{11}\exp\bigl\{-k_{12} e^{2\r |s|}\bigr\},
\]
for suitable choices of $k_{11}$ and~$k_{12}$,
establishing the lower bound.} 
\end{proof}

We now establish a complementary upper bound on the mixing time.

\begin{theorem}\label{upper-bound}
Under the assumptions of this section,
suppose that $\bX^N(0) \in \cX^N(\d_1/4)$, write $\by^N_0 := N^{-1}\bX^N(0)$, and define $t_N(\by^N_0)$
as in~\Ref{ADB-tN-def}.
Then there exist positive constants $C_{\ref{upper-bound}}$, $\hC_{\ref{upper-bound}}$ and~$N_{\ref{upper-bound}}$
such that
\[
     \dtv\bigl(\law(\bX^N(t_N(\by^N_0)+s)),\pi\adbj{\undi}\bigr) \Le  C_{\ref{upper-bound}}\, e^{-\r s},
\]
for all $0 < s \le \r^{-1}\{\half\log N - \hC_{\ref{upper-bound}}\}$ and $N \ge N_{\ref{upper-bound}}$.
\end{theorem}


\begin{proof}
First, consider two copies $(\wX^N_1,\wX^N_2)$ \adbj{of~$\bX\undi$, coupled as in \Ref{ADBi-a1} and~\Ref{ADBi-a2},
and starting at any points $\bX,\bX' \in \cX^N(\d_1/2)$.
Recall the definitions of $\t_1$ and~$\s_1$ from~\Ref{ADB-tau-defs}.} If $H(\bX,\bX') \le \tilde{\nu} K_3$,
\adbc{with~$\tilde{\nu}$ as in~\Ref{ADB-tilde-nu-def},} it follows
from~\Ref{ADB-key-distance-bnd} and~\Ref{Psi-bnd} that
\eq\label{ADB-close-start}
    \adbj{\E_{\bX,\bX'}\{H(\wX^N_1(s \wedge \sdd_1),\wX^N_2(s \wedge \sdd_1))\} \Le  2C_{\Ref{ADBi-C-def}}/\hktt}
\en
for all~$s \ge 0$.  
If $H(\bX,\bX') > \tilde{\nu} K_3$,
it follows from~\Ref{ADB-big-tau-1} that
\eq\label{ADBi-a6}
   e^{\r s}\E_{\bX,\bX'}\{H(\wX^N_1(s),\wX^N_2(s))I[s \le (\t_1\wedge\sdd_1)]\} \Le H(\bX,\bX').
\en
Furthermore, if~$\cF_{\t_1}:= \s\{(\wX^N_1(u),\wX^N_2(u)),\,0\le u\le \t_1\}$, then, by the strong
Markov property and~\Ref{ADB-close-start},
\eqs
    \E_{\bX,\bX'}\{H(\wX^N_1(s),\wX^N_2(s))I[\t_1 \le s \le \sdd_1] \giv \cF_{\t_1}\} \Le \adbj{2C_{\Ref{ADBi-C-def}}/\hktt},
\ens
for $s > 0$, giving
\eq\label{ADBi-a5}
     \E_{\bX,\bX'}\{H(\wX^N_1(s),\wX^N_2(s))I[\t_1 \le s \le \sdd_1]\} \Le \adbj{2C_{\Ref{ADBi-C-def}}/\hktt}
\en
for such~$s$ also.  \adbj{Finally, for $0 \le s \le \a N$, with~$\a$ as defined in~\Ref{ADB-alpha-def1}, and for
$N \ge N_5$,
we have
\eqa
   \E_{\bX,\bX'}\{ H(\wX^N_1(s), \wX^N_2(s)) I[\sdd_1  \le s]\}
           &\le& 2N\d_1 \pr[\sdd_1 \le \a N] \non\\ &\le& 2N\d_1 a_1 a_2 N e^{-a_2 N},  \label{ADBi-a4}
\ena
from~\Ref{ADB-alpha-def0}.}

\adbj{Combining \Ref{ADB-close-start}--\Ref{ADBi-a4}, it follows that,  for $N \ge N_5$
large enough that $2N\d_1 a_1 a_2 N e^{-a_2 N} \le 2C_{\Ref{ADBi-C-def}}/\hktt$, we have
\eq\label{ADB-medium-term}
    \E_{\bX,\bX'}\{H(\wX^N_1(s),\wX^N_2(s))\} \Le e^{-\r s}H(\bX,\bX') + 4C_{\Ref{ADBi-C-def}}/\hktt,
\en
for all $0 \le s \le \a N$ and for all $\bX,\bX' \in \cX^N(\d_1/2)$.
}

\adbj{We now apply the bound in~\Ref{ADB-medium-term} to coupled copies $\wX^N_1$ and~$\wX^N_2$ of~$\bX\undi$,
the former with starting state
$\bX = \wX^N_1(0) = \bX^N(t_N(\by_0^N))$, and the latter starting in $\bX' = \wX^N_2(0)$, chosen as an
independent realization from~$\pi\undi$.}
 \adbj{Let~$A_N$ denote the event that both $\bX$ and~$\bX'$ belong to~$\cX^N(\adbk{\d_1}/2)$.
By Theorem~\ref{ADB-conc-MPP} for $\bX^N(t_N(\by^N_0))$, and by Theorem~\ref{ADB-equilibrium}
for~$\pi\adbj{\undi}$,
\[
     \pr[A_N^c] \Le a'_1 e^{-Na'_2},
\]
for suitable constants $a'_1$ and~$a'_2$, 
provided that~$N$ is sufficiently large.}  On the other hand, 
it follows from Proposition~\ref{ADB-close-TV} that there exist $k_*,u_* > 0$, not depending on~$s$,
such \adbj{that
\eqa
    \lefteqn{ \dtv\bigl(\law(\wX^N_1(s+u_*) \giv \bX,\bX'),\law(\wX^N_2(s+u_*)\giv \bX,\bX')\bigr) I[A_N] } \non\\
       &&\Le k_* N^{-1/2} \adbm{\ex_{\bX,\bX'}\{H(\wX^N_1(s),\wX^N_2(s))\}} \non\\
       &&\Le k_* N^{-1/2}\{e^{-\r s}H(\bX,\bX') + 4C_{\Ref{ADBi-C-def}}/\hktt\}, \label{ADB-full-term}
\ena
from}~\Ref{ADB-medium-term}.  Taking expectations, it thus follows that
\eqa
    \lefteqn{ \dtv\bigl(\law(\bX^N(t_N(\by^N_0)+s+u_*)),\pi\adbj{\undi} \bigr) } \non\\
    &\le& k_* N^{-1/2}\{e^{-\r s}(\E\|\bX - N\bc\|_M + \E\|\bX' - N\bc\|_M) \non\\
     &&\mbox{}\hskip1.5in         \adbj{+ 4C_{\Ref{ADBi-C-def}}/\hktt}\} + a'_1 e^{-Na'_2}. \label{ADBi-a15}
\ena
Note that, from Lemma~\ref{mean-drift}, \adbm{\Ref{ADB-ev2} and the definition of~$t_N(\by^N_0)$, since 
$\bX =  \bX^N(t_N(\by_0^N))$ and $\bX^N(0) = N\by_0^N \in \cX^N(\d_1/4)$, 
we have
\[
   \E\|\bX - N\bc\|_M \Le N^{1/2}\{d \sqrt v + C_{\ref{mean-drift}} + 1).
\]
Also,} \adbj{from~\Ref{equilibrium-var},} $\ex\|\bX' - N\bc\|_M^2 \le N \adbk{v_{\infty}} $.
\adbj{Using these bounds in~\Ref{ADBi-a15} gives the inequality}
\eqa
   \lefteqn{ \dtv\bigl(\law(\bX^N(t_N(\by^N_0)+s+u_*)),\pi\adbj{\undi} \bigr) } \non\\
   &\le& k_* N^{-1/2}\{e^{-\r s}\adbk{\sqrt N}( C_{\ref{mean-drift}} \adbm{+ 1 + d \sqrt{v}} + \sqrt{v_\infty})
       + \adbj{4C_{\Ref{ADBi-C-def}}/\hktt}\} \non\\
     &&\mbox{}\hskip2in           + a'_1 e^{-Na'_2}.
                \label{ADB-main-bnd-1}
\ena
 \adbj{Now, for $N$ large enough that $a'_1 e^{-Na'_2} \le k_* N^{-1/2}C_{\Ref{ADBi-C-def}}/\hktt$, and if~$s$ is such that
 $5N^{-1/2}C_{\Ref{ADBi-C-def}}/\hktt \le \adbk{e^{-\r s}}$, it follows from~\Ref{ADB-main-bnd-1} that
 \eq\label{ADB-main-bnd}
   \dtv\bigl(\law(\bX^N(t_N(\by^N_0)+s+u_*)),\pi\adbj{\undi} \bigr) \Le k_* \adbk{e^{-\r s}}
      ( C_{\ref{mean-drift}} + \adbm{2 + d \sqrt{v}} + \sqrt{v_\infty}).
 \en
Hence}, for $N_{\ref{upper-bound}}$ chosen large enough,
taking $\hC_{\ref{upper-bound}} := \adbj{\log\{5C_{\Ref{ADBi-C-def}}/\adbm{\hktt}\}}$ and
$0 \le s \le \r^{-1}\{\half\log N - \hC_{\ref{upper-bound}}\}$,
we have the desired upper bound
\eq\label{ADB-epi-upper-bnd}
    \dtv\bigl(\law(\bX^N(t_N(\by^N_0)+s)),\pi\adbj{\undi}\bigr) \Le C_{\ref{upper-bound}} e^{-\r s},
\en
with $C_{\ref{upper-bound}} := e^{\r u_*}\adbk{k_*}(C_{\ref{mean-drift}} + \adbm{2 + d \sqrt v} + \sqrt{v_\infty})$.
\end{proof}

For starting points outside~$\adbj{\cX^N(\d_1/4)}$, things are somewhat similar.
Let \adbj{$\by_0 \notin B_M(\bc,\d_1/4)$} belong to the basin of attraction~$\cB(\bc)$ of the fixed point~$\bc$
of the deterministic equations~\Ref{eq:diff-eq}, and consider the
path of the solution~$\adbm{\by_{[\by_0]}}$ of~\Ref{eq:diff-eq} from~$\by_0$ until the time $T := T(\by_0;\adbj{\d_1/8})$
at which it first reaches~$B_M(\bc,\adbj{\d_1/8})$.
Since~$\adbm{\by_{[\by_0]}(\cdot)}$ is a continuous function in $[0,T]$, and since it cannot hit a point outside~$\cB(\bc)$,
the quantity
\eq\label{ADBi-33}
   \adbm{\e(\by_0) \Def \inf_{0 \le t \le T}\{ \|\by_{[\by_0]}(t) - \cB(\bc)^c\|_M\}}
\en
is strictly positive.  \adbm{Defining $\e'(\by_0) := \min\{\e(\by_0),\d_1/16\}$, recalling the definition
of $\cY_\eps(\bx,T)$ from~\Ref{ADB-apr-1}, and}
applying Lemma~\ref{ADB-exp-bound-XN}
with $\cK := \cY_{\e(\by_0)}(\by_0,T)$ shows that, if~$\bx^N(0) = \by_0$, the path $(\bx^N(t),\,0\le t \le T)$
lies within~$\cY_{\adbj{\e'(\by_0)}}(\by_0,T)$ with very high probability, bounded below by
$1 - \z_{N,T,\cK}(\adbj{\e'(\by_0)} e^{-TL(\cK)})$.
On this event, $\|\bx^N(T) - \by(T)\|_M \le \e'(\by_0) \le \d_1/16$, and, since 
$\by(T) \in B_M(\bc,\adbj{\d_1/8})$, it follows that $\|\bX^N(T) - N\bc\|_M \le 3N\d_1/16$; hence
\[
  \pr[\bX^N(T) \notin B_M(N\bc,N\adbj{\d_1/4})] \Le \z_{N,T,\cK}(\adbj{\e'(\by_0)} e^{-TL(\cK)}).
\]
Theorems \ref{lower-bound} and~\ref{upper-bound} can now be applied to the subsequent evolution of~$\bX^N$
from time~$T$ onwards.
All that is required is to account for the difference between $t_N(\by_0) = T + t_N(\by(T))$
and $T + t_N(\bx^N(T))$, where, as above, $T := T(\by_0;\adbj{\d_1/8})$.

\begin{theorem}\label{ADB-general-starting-point}
Suppose that \adbm{$\by_0 \in \cB(\bc) \setminus B_M(\bc,\adbj{\d_1/4})$}.
Then there exists an \adbm{$\h = \h(\by_0) > 0$} such that, if $\bX^N(0) \in B\adbk{_M}(N\by_0,N\h)$, and
if~$t_N(\cdot)$ is defined as in~\Ref{ADB-tN-def}, then the following two statements hold.
\begin{description}
 \item[1] There exists positive constants $C'_{\ref{lower-bound}}$ and~$\th$ and a positive integer $N'_{\ref{lower-bound}}$ such that
\[
     \dtv\bigl(\law(\bX^N(t_N(\bx^N(0))+s)),\pi\adbj{\undi}\bigr)
         \ \ge\ 1 - C'_{\ref{lower-bound}}\, \adbe{\exp\bigl\{-k'_{\ref{lower-bound}} e^{-2\r |s|}\}},
\]
for all \adbe{$-\th\log N \le s < 0$} and $N \ge N'_{\ref{lower-bound}}$.

\medskip
\item[2] There exist positive constants $C'_{\ref{upper-bound}},\hC'_{\ref{upper-bound}}$ and $N'_{\ref{upper-bound}}$ 
such that
\[
     \dtv\bigl(\law(\bX^N(t_N(\bx^N(0))+s)),\pi\adbj{\undi}\bigr) \Le  C'_{\ref{upper-bound}} e^{-\r s},
\]
for all $0 < s \le \r^{-1}\{\half\log N - \hC'_{\ref{upper-bound}}\}$ and $N \ge N'_{\ref{upper-bound}}$.
\end{description}
\adbe{The constants $C'_{\ref{lower-bound}}$, $k'_{\ref{lower-bound}}$, $C'_{\ref{upper-bound}}$,
$\hC'_{\ref{upper-bound}}$, $N'_{\ref{lower-bound}}$ and $N'_{\ref{upper-bound}}$
may depend on the choice of~$\by_0$.}
\end{theorem}

\begin{proof}
Choose $\h_1 < \min\{\d_1,\e(\by_0)\}/(\adbk{16})$, and let $\cK_1 := \cY_{2\h_1}(\by_0,T)$,
where $T := T(\by_0;\adbj{\d_1/8})$ is as defined above.
Then, by a Gronwall argument similar to that in Lemma~\ref{ADB-exp-bound-XN}, all solutions~$\by := \by_{[\by']}$
to~\Ref{eq:diff-eq}
with $\by(0) = \by' \in B(\by_0,\h_1 e^{-L(\cK_1)T})$ satisfy
\eq\label{ADB-DE-solns-close}
   \|\by_{[\by']}(t) - \by_{[\by_0]}(t)\adbk{\|_M} \le \h_1 \mbox{ for all } 0 \le t \le T.
\en
The set
\[
     \cK_1' \Def \{\bz \in \re^d\colon\, \bz = \by_{[\by']}(t)
               \mbox{ for some } t \in [0,T], \by' \in B\adbk{_M}(\by_0,\h_1 e^{-L(\cK_1)T})\},
\]
being the image of the compact set $[0,T] \times B\adbk{_M}(\by_0,\h_1 e^{-L(\cK_1)T})$ 
under the map $(t,\by') \mapsto \by_{[\by']}(t)$, is
itself compact. 
\adbj{All points} in~$\cK_1'$ lie on trajectories of the system~\Ref{eq:diff-eq} starting in
$B(\by_0,\h_1 e^{-L(\cK_1)T})$, \adbj{and,} by~\Ref{ADB-DE-solns-close} and by the choice of~$\h_1$,
\adbj{all such trajectories} lie wholly
within~$\cB(\bc)$ up to time~$T$, \adbj{at which time they belong \adbm{to~$B_M(\bc,\d_1/4)$}. Thereafter, they are}
attracted to~$\bc$.  \adbj{Thus $\cK_1' \subset \cB(\bc)$, and so}
$\h_1' := \inf_{\by' \in \cK_1',\bz \notin \cB(\bc)} |\by' - \bz|$ is strictly positive, as is
therefore $\h_2 := \tfrac12\min\{\h_1,\h_1'\}$.  Defining
\[
    \t_N(\cK_1) \Def \inf\{t \ge 0\colon\, \bx^N(t) \notin \cK_1 \cap \cB(\bc)\},
\]
it thus follows from Lemma~\ref{ADB-exp-bound-XN} that, starting with $\bx^N(0) = \by'$
for any $\by' \in B\adbk{_M}(\by_0,\h_1 e^{-L(\cK_1)T})$, we \adbe{have
\eqa
    \lefteqn{\pr_{\by'}\bigl[\{\t_N(\cK_1) \le T\} \cup \{\|\bx^N(T) - \by_{[\bx^N(0)]}(T)\adbk{\|_M} > \h_2\}\bigr]} \non\\
              &&\Le \z_{N,T,\cK_1}(\h_2 e^{-TL(\cK_1)}) \Le b_1 e^{-Nb_2}\ =:\ \z_N',  \label{ADB-zNd-def}
\ena
for suitable constants $b_1$ and~$b_2$,}
where $\z_{N,t,\cK}$ is as defined in Lemma~\ref{ADB-exp-bound-MG}. 
In particular, except on an
\adbe{event~$E_1(N,T)$} of probability at most $\z_N'$, $\bx^N(T) \in B_M(\bc,\adbj{\d_1/4})$,
and conditional on the value of~$\bx^N(T)$,
Theorems \ref{lower-bound} and~\ref{upper-bound} can be applied to the process $(\bx^N(T+t),\,t \ge 0)$.
\adbe{In consequence,  except on the event~$E_1(N,T)$, conditional on the value of~$\bx^N(T)$,
\eq\label{ADB-repeat-1}
     \dtv\bigl(\law(\bX^N(T + t_N(\bx^N(T)) + s)\adbj{\giv \bx^N(T)}),\pi\adbj{\undi}\bigr)\adbj{I[E_1^c(N,T)]}
           \ \ge\ 1 - \D_1(s),
\en
for all $-t_N(\bx^N(T)) \le s < 0$ and $N \ge N_{\ref{lower-bound}}$, where
\[
   \D_1(s) \Def C_{\ref{lower-bound}}\,\exp\bigl\{-k_{\ref{lower-bound}} e^{2\r |s|}\bigr\};
\]
and
\eq\label{ADB-repeat-2}
     \dtv\bigl(\law(\bX^N(T + t_N(\bx^N(T)) + s)\adbj{\giv \bx^N(T)}),\pi\adbj{\undi}\bigr)\adbj{I[E_1^c(N,T)]} \Le  \D_2(s),
\en
for all $0 < s \le \r^{-1}\{\half\log N - C'_{\ref{upper-bound}}\}$ and $N \ge N_{\ref{upper-bound}}$,
where
\[
    \D_2(s) \Def C_{\ref{upper-bound}}\, e^{-\r s}.
\]
}

The bounds \Ref{ADB-repeat-1} and~\Ref{ADB-repeat-2} give information about the conditional distribution of~$\bX^N$
near the time $T+t_N(\bx^N(T))$, whereas, for the statements of the current theorem, times near
$t_N(\bx^N(0)) = T + t_N(\by_{[\bx^N(0)]}(T))$ are needed. However, from Lemma~\ref{ADB-exp-bound-XN},
\adbe{for $0 < \adbj{\chi} < 1/2$},
\eq\label{ADB-starting-different}
   \|\bx^N(T) - \by_{[\bx^N(0)]}(T)\adbk{\|_M} \Le 
             N^{-1/2 + \adbj{\chi}},
\en
except on an \adbe{event~$E_2(N,T,\adbj{\chi})$} of probability at most
 \[
     \z_{N,T,\cK_1}(N^{-1/2 + \adbj{\chi}} e^{-TL(\cK_1)})
               \Le \adbj{b_1(\chi)} \exp\{-\adbj{b_2(\chi)} N^{2\adbj{\chi}}\} \ =:\ \z_N''
 \]
for suitable $\adbj{b_1(\chi),b_2(\chi)} > 0$; note that $\z_N'' \ge \z_N'$ for all~$N$ large enough.
Now, by the definition~\Ref{ADB-tN-def},
\eqs
   \lefteqn{\|\by_{[\bx^N(0)]}(T + t_N(\by_{[\bx^N(0)]}(T))) \adbm{- \bc}\|_M}\\
          &&\Eq \|\by_{[\by_{[\bx^N(0)]}(T))]}(t_N(\by_{[\bx^N(0)]}(T))) \adbm{- \bc}\|_M \Eq N^{-1/2}.
\ens
\adbj{Then, from Theorem~\ref{thm:main1} (ii), comparing the evolution of the solutions to the differential
equations~\Ref{eq:diff-eq}
with starting points $\bx^N(T)$ and~$\by_{[\bx^N(0)]}(T)$ over a time interval of length $t_N(\by_{[\bx^N(0)]}(T))$,}
the inequality~\Ref{ADB-starting-different} in turn implies that,
\adbm{except on the event~$E_2(N,T,\adbj{\chi})$,}
\eq\label{ADBi-20}
    \bigl|\,\|\by_{[\bx^N(T)]}(t_N(\by_{[\bx^N(0)]}(T))) \adbm{- \bc}\|_M - N^{-1/2}\bigr|
             \Le 
                N^{-1/2+\adbj{\chi}} \exp\{-\r t_N(\by_{[\bx^N(0)]}(T))\}.
\en
Note also that, \adbm{defining $\bu_0 := \by_{[\bx^N(0)]}(T)$ and taking $\bw(\cdot) = \by_{[\bu_0]}(\cdot) - \bc$ 
in~\Ref{ADB-other-bound}, it follows that
\eqs
     N^{-1/2} \Eq \|\by_{[\bu_0]}(t_N(\bu_0)) - \bc\|_M  &\ge& \|\by_{[\bu_0]}(0) - \bc\|_M\, e^{-\r^* t_N(\bu_0)}\\
                       &=& \|\by_{[\bx^N(0)]}(T) - \bc\|_M\, e^{-\r^* t_N(\bu_0)}.
\ens
Since, from~\Ref{ADB-DE-solns-close},
$\bigl|\,\|\by_{[\bx^N(0)]}(T) \adbm{- \bc}\|_M - \adbj{\d_1/8}\bigr| \le \adbk{\h_1} < \adbj{\d_1/16}$,
it follows that $\|\by_{[\bx^N(0)]}(T) - \bc\|_M \ge \adbj{\d_1/16}$, and hence that
\[
   \exp\{-\r^* t_N(\by_{[\bx^N(0)]}(T))\} \Le 16N^{-1/2}/\d_1.  
\]
Thus, \adbj{from~\Ref{ADBi-20}}, we deduce} that,
except on the event~$E_2(N,T,\adbj{\chi})$,
\eqa
   \lefteqn{\bigl|\,\|\by_{[\bx^N(T)]}(t_N(\by_{[\bx^N(0)]}(T))) \adbm{- \bc}\|_M - N^{-1/2}\bigr|} \non\\
                 &&\Le N^{-1/2 + \adbj{\chi}} N^{-\r/(2\r^*)}\adbm{(16/\d_1)^{\r/\r^*}}
                 \ =:\ N^{-1/2}\ps_N\adbj{(\chi)}, \label{ADBi-21}
\ena
say, where we now choose $\adbj{\chi = \chi_0} < \r/(2\r^*)$, so that $\lim_{N \to \infty}\ps_N\adbj{(\chi_0)} = 0$.
Once again using Theorem~\ref{thm:main1} (i), this implies \adbe{that
\eqa
    |t_N(\by_{[\bx^N(0)]}(T)) - t_N(\bx^N(T))|
    &=& |t_N(\bx^N(0)) - (T + t_N(\bx^N(T))| \non\\
    &\le& \frac1{\r}\adbj{|\log(1 - \ps_N(\chi_0))|}, \label{ADBi-22}
\ena
except on the event~$E_2(N,T,\adbj{\chi_0})$}.

These considerations imply that
$T + t_N(\bx^N(T))$ can be replaced by $t_N(\bx^N(0))$ in the expressions \Ref{ADB-repeat-1} and~\Ref{ADB-repeat-2},
by changing some of the constants, if necessary.  First, $\pr[E_1(N,T) \cup E_2(N,T,\adbj{\chi_0})] \le 2\z_N''$
is smaller than a fixed multiple of~$\D_1(s)$ for
all~$s$ such that $k_{\ref{lower-bound}} e^{2\r |s|} \le \adbj{b_2(\chi_0)} N^{2\adbj{\chi_0}}$.
This can be achieved by restricting~$s < 0$
to have $|s| \le \th\log N$, for any choice of $\th < \adbj{\chi_0}/\r$ and then for $N \ge n_3(\th)$.
\adbj{Then the quantity~$2\z_N''$} is smaller than~$\D_2(s)$ for all
$0 < s \le \half\r^{-1}\log N$, if $N \ge n_4$, say.  Furthermore, the inequalities in \Ref{ADB-repeat-1} and~\Ref{ADB-repeat-2}
\adbj{are} preserved when changing $T + t_N(\bx^N(T))$ to $t_N(\bx^N(0))$, if~$N$ is large enough that
\adbj{\Ref{ADBi-22} ensures that} $|t_N(\bx^N(0)) - (T + t_N(\bx^N(T))| \le \r^{-1}$,
by replacing~$k_{\ref{lower-bound}}$ with $e^{-2}k_{\ref{lower-bound}}$ and~$C_{\ref{upper-bound}}$ with
$eC_{\ref{upper-bound}}$.
\adbj{The statements of the theorem now follow, by taking expectations in \Ref{ADB-repeat-1} and~\Ref{ADB-repeat-2}.}
\end{proof}

The statement of Theorem~\ref{thm:main} for $\d =\adbj{\d_1}$
is implied by combining those of Theorems \ref{lower-bound}, \ref{upper-bound}
and~\ref{ADB-general-starting-point}.  Any compact set $\cK \subset \cB(\bc) \setminus \{\bc\}$ is covered by
the union of $B_M(\bc,\d_1/2)$ and a {\it finite\/} collection of balls of the form $B\adbk{_M}(\by_0,\h(\by_0))$, with
$\by_0 \in \cK \setminus \mathrm{int}(B_M(\bc,\d_1/2))$. \adbe{So take the maximum of the values of the quantities
$C'_{\ref{lower-bound}}$, $C'_{\ref{upper-bound}}$ and~$\hC'_{\ref{upper-bound}}$ that appear for these
balls in Theorem~\ref{ADB-general-starting-point}, and of the
corresponding constants $C_{\ref{lower-bound}}$, $C_{\ref{upper-bound}}$ and~$\hC_{\ref{upper-bound}}$
from Theorems \ref{lower-bound} and~\ref{upper-bound}; and take the minimum of the values of
$k'_{\ref{lower-bound}}$ and of~$k_{\ref{lower-bound}}$.  Using these values,} Theorem~\ref{thm:main} is established.

The following remark shows that other choices of~$\d$ can be used instead; that is, other distributions~$\pi\und$
can also act as quasi-equilibrium distributions.

\begin{remark}\label{equilibrium-dists}{\rm
Taking $s := \r^{-1}\{\half\log N - C'_4\}$, \adbm{Theorem~\ref{ADB-general-starting-point}} implies that
     $\dtv\bigl(\law(\bX^N(t_N(\bx^N(0))+s)),\pi\adbj{\undi}\bigr)$ is of order~$O(N^{-1/2})$.
Thus it is possible to couple random vectors $\bX$ and~$\bX'$ with the distributions~$\pi\adbj{\undi}$
and $\law(\bX^N(t_N(\bx^N(0))+s))$,
respectively, on the same probability space, in such a way that \adbm{$\pr[\bX \neq \bX'] = O(N^{-1/2})$}.
Both of these distributions are well concentrated near~$N\bc$, in view of Theorems \ref{ADB-conc-MPP}
and~\ref{ADB-equilibrium}.
With these starting points, copies of processes $\bX^N(\cdot)$ and $\bX\adbj{\undi}(\cdot)$ can be run, with identical
transitions, until the boundary of~$\cX^N(\d_1/2)$ is reached.  By Proposition~\ref{ADB-3.5}, the probability of
this occurring before time~$T_N$ is of order $O(T_N e^{-Na})$, for some $a > 0$.  Hence the distribution of~$\bX^N(t)$
remains close to~$\pi\adbj{\undi}$ for times of length growing exponentially with~$N$.  An analogous coupling argument
shows that the distributions $\pi\adbj{\undi}$ and~$\pi\und$ are close, for~$N$ large enough, for any $0 < \d \le \adbm{\d_1}$.
Any such choice of~$\pi\und$ can act as a quasi-equilibrium for~$\bX^N$ in the neighbourhood of~$N\bc$.}
\end{remark}

\ignore{
Taking any $\by_0 \in \cX^N(\d)$, consider the process starting from $\bx_0 := \lfloor N\by_0 \rfloor$.
The deterministic solution~$\by$ starting at~$\by_0$ first satisfies $\|\by(t) - \by_0\|_M \le \d_1/4$
at a time $t_1(\by_0)$, which is finite, provided that~$\by_0$ is in the interior of the
basin of attraction of the fixed point~$\bc$.  At the time~$t_1(\by_0)$, using standard martingale techniques as before,
$X^N(t_1(\by_0))$ is of order $O(N^{1/2})$ close to $N\by(t_1(\by_0))$,
and belongs to~$\cX^N(\d_1/2)$ with probability at least $1 - O(e^{-N^{1/3}})$.
Thereafter, we can use the previous argument to establish cutoff.
}

\nin{\bf Example~1.}
\adbe{The simple immigration--death process, with unit immigration rate and unit {\it per capita\/}
death rate, can be formulated as a Markov population process~$X^N$
on~$\Z_+$, satisfying Assumptions 1 and~2 with $\cJ = \{-1,1\}$, $r_{-1}(x) = x$ and $r_1(x) = 1$;
the quantity~$N$ represents a typical population size.  The differential equation
\[
    \frac{dy}{dt} \Eq F(y) \Eq 1 - y
\]
has fixed point~$1$ and basin of attraction~$\re_+$, and $dF/dx = -1$ throughout~$\re_+$, so that~$\r$ can be
taken to be any value less than~$1$.  Because the process is one-dimensional, $\|\cdot\|_M$ can be taken to be
the Euclidean norm, and, for $|x-1| > N^{-1/2}$, we have $t_N(x) = \half\log N + \log|1-x|$, from~\Ref{eq.cutoff-time}.
Any value $\d_1 \in (0,1)$ satisfies the requirements of Remark~\ref{ADB-d_1-def}, with corresponding
value of $r_0 = 1-\d_1$, and $r_1 = 1$ in~\Ref{ADB-rate-shift}.
}

\adbe{
The interest in this example is that the cutoff can be explicitly examined, since~$X^N$ has equilibrium
distribution $\pi^N = \Po(N)$, \adbj{the Poisson distribution with mean~$N$,} and
\[
   \law(X^N(t) \giv X^N(0) = j) \Eq \Po(N(1-e^{-t})) * \Bi(j, e^{-t}),
\]
where \adbj{$\Bi(n,p)$ denotes the binomial distribution, and}~$*$ denotes convolution.
In particular, $\law(X^N(t) \giv X^N(0) = 0) = \Po(N(1-e^{-t}))$, and it is easy to see that the total
variation distance between the distributions $\Po(N(1-e^{-t}))$ and~$\Po(N)$ is attained on a set of the
form $E(k_N) := \{0,1,2,\ldots,k_N\}$.  As a result, the Berry--Esseen theorem, for large~$N$, can be used to
approximate $\dtv\bigl(\Po(N(1-e^{-t})),\Po(N)\bigr)$.  Since $t_N(0) = \half\log N$, we investigate times
$t = \half\log N + s$, for which $Ne^{-t} = N^{1/2}e^{-s}$, noting that then, for any $s' \in \re \cup \{+\infty\}$,
\[
     \lim_{N \to \infty}\sup_{x \in \re}|\Po(N - N^{1/2}e^{-s'})\{[0,N - xN^{1/2}]\} - \Phi(e^{-s'}-x)| \Eq 0,
\]
where~$\Phi$ denotes the standard normal distribution function. Hence, for $t = \half\log N + s$,
writing
\[
  \Delta_N(s,x) \Def \Po(N - N^{1/2}e^{-s})\{[0,N - xN^{1/2}]\} - \Po(N)\{[0,N - xN^{1/2}]\},
\]
and noting, as above, that $\sup_{x\in\re}\Delta_N(s,x) = \dtv\bigl(\Po(N- N^{1/2}e^{-s}),\Po(N)\bigr)$, we deduce that
\eqs
    \lefteqn{\lim_{N \to \infty} \dtv\bigl(\Po(N- N^{1/2}e^{-s}),\Po(N)\bigr) \Eq \sup_{x \in \re} \{\Phi(e^{-s}-x) - \Phi(-x)\} }\\
      &&\qquad\quad\ \Eq\Phi(\half e^{-s}) - \Phi(-\half e^{-s})
      \Eq 1 - 2\Phi(-\half e^{-s}). \phantom{XXXX}
\ens
Hence
\[
    \lim_{N \to \infty}\dtv\bigl(\Po(N- N^{1/2}e^{-s}),\Po(N)\bigr)
      \Eq 1 - 2\Phi(-\half e^{-s}).
\]

For~$s$ negative,
\eq\label{ADB-id-1}
   2\Phi(-\half e^{-s}) \Le \frac2{e^{|s|}}\,\sqrt{\frac2\pi}\,\exp\bigl\{-e^{2|s|}/8\bigr\},
\en
and, for~$s$ positive,
\eq\label{ADB-id-2}
  1 - 2\Phi(-\half e^{-s}) \Le e^{-s}/\sqrt{2\pi},
\en
both bounds being tight as $|s| \to \infty$.  The bound in~\Ref{ADB-id-1} is broadly comparable to the bound given in
Theorem~\ref{ADB-general-starting-point} for negative~$s$, as regards the form of the dependence on~$|s|$,
in that there $\r<1$ can be taken arbitrarily close to~$1$;
however, the values of the constants there may be worse, and there is no factor $e^{|s|}$ in the denominator.
The bound in~\Ref{ADB-id-2} is closely comparable to the bound given in Theorem~\ref{ADB-general-starting-point} for positive~$s$,
with~$\r<1$ taken close to~$1$, though again the constant there may be worse.
}

\bigskip

\nin{\bf Example~2.}
Let~$\bX^N$ be a two-dimensional process in continuous time, representing
an SIR epidemic with immigration of susceptibles (Hamer, 1906).
In state $(X_1,X_2)^T \in \Z_+^2$, there are $X_1$ susceptibles and~$X_2$ infectives.
From any state $(X_1,X_2)^T$, there are three possible transitions, whose rates are as follows:
\begin{eqnarray*}
  (X_1,X_2) &\to& (X_1-1,X_2+1)\ \mbox{ at rate }\ \alpha X_1 X_2 / N \\
  (X_1,X_2) &\to& (X_1+1,X_2)\ \mbox{ at rate }\ \beta N \\
  (X_1,X_2) &\to& (X_1,X_2-1)\ \mbox{ at rate }\ \gamma X_2.
\end{eqnarray*}
Here, $\alpha$, $\beta$ and $\gamma$ are fixed constants, and the parameter~$N$ is a
measure of the typical population size.
The first transition corresponds to an infection: a susceptible encounters an infective
and becomes infected.  The second transition corresponds to immigration of susceptibles into
the population.  The third transition
corresponds to the recovery of an infective, after which they become immune to the disease.
Although this process is transient, eventually drifting to infinity along the $1$-axis,
it exhibits a quasi-equilibrium distribution, in the sense of Barbour \& Pollett~(2012, Section~4).
Starting with $\bX^N(0) = \lfloor N\by \rfloor$, for any $\by := (y_1,y_2)^T$ with $y_2 > 0$,
the process with high probability approaches an apparent equilibrium distribution~$\pi^N$,
which persists for a time whose mean grows exponentially with~$N$.

In the notation of Section~\ref{Density-dependent},  the set $\cJ$ consists of the three
vectors $(-1,1)^T$, $(1,0)^T$ and $(0,-1)^T$.  The functions $r_{\bf J}({\bf y})$ are
given by $\alpha y_1 y_2$, $\beta$ and $\gamma y_2$, respectively, and
\[
   F(\by) \Eq \left(\begin{matrix}
                       -\alpha y_1 y_2 + \b \\ \alpha y_1 y_2 - \g y_2
                    \end{matrix} \right).
\]
The solution~$\bc$ of $F(\by) = \bo$,
the fixed point of the differential equation~\Ref{eq:diff-eq}, is
given by $(\g/\a,\b/\g)^T$. The matrix~$A = DF(\bc)$ is then
\[
  A \Def \begin{pmatrix}
      - \frac{\alpha\beta}{\gamma} & - \gamma \\
      \frac{\alpha\beta}{\gamma} & 0
   \end{pmatrix}.
\]

The equilibrium innovations matrix is given by
$$
    \s^2 \Def \sJJ \bJ \bJ^T r_\bJ(\bc) \Eq \b\begin{pmatrix}
      2 & - 1 \\
      - 1 & 2
   \end{pmatrix},
$$
from which the covariance matrix~$N\Sigma$ of the discrete normal approximation to $\pi^N$,
where  $A\Sigma + \Sigma A^T + \s^2 = 0$, can be deduced:
$$
    \Sigma  \Eq \begin{pmatrix}
      \frac{\g}{\alpha\beta}\Bigl(\b + \frac{\g^2}{\a}\Bigr) & - \frac{\gamma}{\a} \\
      - \frac{\gamma}{\a} & \frac1{\g}\Bigl(\b + \frac{\g^2}{\a}\Bigr)
   \end{pmatrix}.
$$
Then \adbb{\BLX~(2018a, Theorem~5.3 and 2018b, Theorem~2.3)} show that
\[
    \dtv\bigl(\pi^N,\DN_2(N\bc,N\Sigma)\bigr) \Eq O(N^{-1/2}\log N),
\]
\adbj{where~$\DN_2$ denotes a bivariate discrete normal distribution.}

The eigenvalues of~$A$ are
$$
    \lambda_1, \lambda_2 \Eq - \frac{\alpha \beta}{2 \gamma}
              \pm \frac12 \sqrt{ \alpha \beta \Bigl( \frac{\alpha \beta}{\gamma^2} - 4\Bigr) }.
$$
We can thus take~$\r$ to be any value smaller than
\[
   \wrho \Def \begin{cases}
                  \frac{\alpha \beta}{2 \gamma} -
                    \frac12 \sqrt{ \alpha \beta \Bigl( \frac{\alpha \beta}{\gamma^2} - 4\Bigr)}
                              &\ \mbox{if}\ \alpha \beta \ge 4 \g^2;\\
                  \frac{\alpha \beta}{2 \gamma}  &\ \mbox{otherwise}.
              \end{cases}
\]
\adbb{Assumptions 1--3} are easily seen to be satisfied, and Theorem~\ref{ADB-general-starting-point}
can be applied to show that~$\bX^N$ exhibits cutoff in its approach to \adbj{quasi-}equilibrium.



\section*{Acknowledgements}
We are pleased to thank a referee, who did a marvellous job on an earlier version of the paper.

\end{document}